\documentclass[a4paper,11pt,english]{smfart}

\usepackage[english]{babel}
\usepackage[utf8]{inputenc}
\usepackage[T1]{fontenc}
\usepackage{lmodern}
\usepackage{mathtools}
\usepackage{array}
\usepackage{braket}
\usepackage{tikz}
\usepackage{todonotes}
\usepackage{color}
\definecolor{shadecolor}{gray}{0.90}
\RequirePackage{calrsfs}
\DeclareSymbolFont{rsfscript}{OMS}{rsfs}{m}{b}
\DeclareSymbolFontAlphabet{\mathrsfs}{rsfscript}

\usepackage{hyperref}

\usepackage{fancyvrb}

\usepackage[utopia,expert]{mathdesign}

\usepackage{lscape}

\usepackage[final]{showlabels}

\usepackage[leftbars]{changebar}

\usepackage{palatino}

\usetikzlibrary{cd}

\theoremstyle{plain}
\newtheorem{theoreme}{Theorem}[section]
\newtheorem{proposition}[theoreme]{Proposition}
\newtheorem{lemma}[theoreme]{Lemma}
\newtheorem{corollary}[theoreme]{Corollary}
\newtheorem{definition}[theoreme]{Definition}
\newtheorem{conjecture}[theoreme]{Conjecture}

\newtheorem{thmintro}{Theorem}

\def\arxivprefix{https://arxiv.org}

\theoremstyle{definition}

\theoremstyle{remark}
\newtheorem{remark}[theoreme]{Remark}

\DeclareMathOperator{\Hom}{Hom}

\DeclareMathOperator{\Rep}{Rep}
\DeclareMathOperator{\id}{id}

\DeclareMathOperator{\Tr}{Tr}

\DeclareMathOperator{\Irr}{Irr}
\DeclareMathOperator{\ev}{ev}
\DeclareMathOperator{\coev}{coev}
\DeclareMathOperator{\Gr}{Gr}
\DeclareMathOperator{\sGr}{sGr}

\DeclareMathOperator{\sdim}{sdim}
\DeclareMathOperator{\h}{ht}
\DeclareMathOperator{\Fr}{Fr}

\newcommand*{\fonctionNom}[5]{
  #1 \colon \left\{
  \begin{array}{ccc}
    #2 & \longrightarrow & #3 \\
    #4 & \longmapsto & #5
  \end{array}
  \right.
}

\newcommand*{\fonction}[4]{
  \begin{array}{ccc}
    #1 & \longrightarrow & #2 \\
    #3 & \longmapsto & #4
  \end{array}
}

\def\boitegrise#1#2{\begin{centerline}{\fcolorbox{black}{shadecolor}{~
    \begin{minipage}[t]{#2}{\vphantom{~}#1\vphantom{$A_{\displaystyle{A_A}}$}}
            \end{minipage}~}}\end{centerline}\medskip}

\newcommand*{\qint}[2]{
[#1]_{#2}
}
\newcommand*{\qfact}[2]{
[#1]_{#2}!
}
\newcommand*{\qbinom}[3]{
  \begin{bmatrix}
    #2\\#1
  \end{bmatrix}_{#3}
}

\newcommand{\qgr}[1]{\mathcal{U}_q(#1)}

\newcommand{\qgrroot}[1]{\mathcal{U}_{\xi}(#1)}
\newcommand{\qdbl}[1]{\mathcal{D}_q(#1)}
\newcommand{\qdblres}[1]{\mathcal{D}^{\mathrm{res}}_q(#1)}
\newcommand{\qdblroot}[1]{\mathcal{D}_{\xi}(#1)}

\newcommand{\br}[3]{
  \begin{bmatrix}
    #1;#2\\
    #3
  \end{bmatrix}
}
\newcommand{\dbr}[3]{
  \begin{bmatrix}
    K_{#1};#2;L_{#1}\\
    #3
  \end{bmatrix}
}

\title{Drinfeld double of quantum groups, tilting modules and $\mathbb{Z}$-modular data associated to complex reflection groups}
\author{{\sc Abel Lacabanne}}
\address{
Institut Montpelliérain Alexander Grothendieck (CNRS: UMR 5149), 
Université de Montpellier,
Case Courrier 051,
Place Eugène Bataillon,
34095 MONTPELLIER Cedex,
FRANCE}
\date{\today}
\email{abel.lacabanne@umontpellier.fr}

\begin{document}
\maketitle 

\pagestyle{myheadings}
\markboth{\sc A. Lacabanne}{\sc Tilting modules for Drinfeld double of quantum groups}

In this article, we construct a categorification of some Fourier matrices associated to complex reflections groups by Malle \cite{unip_malle}. To any spetsial imprimitive complex reflection group $W$, he attached a set of unipotent characters, which is in bijection with the set of unipotent characters of the corresponding finite reductive group when $W$ is a Weyl group. He also defined a partition of these characters into families, the analogue of Lusztig's nonabelian Fourier transform and eigenvalues of the Frobenius. In \cite{cuntz}, Cuntz showed that for each family, this defines a so-called $\mathbb{Z}$-modular datum. Therefore, one can associate a $\mathbb{Z}$-algebra, free of finite rank over $\mathbb{Z}$. If the stucture constants of such an algebra $R$ are positive, it is a classical problem to find a tensor category, with some extra structure, whose Grothendieck ring is precisely $R$. However, these structure constants belong in general to $\mathbb{Z}$ and the categorification is more involved. 

In the case of cyclic groups, Bonnafé and Rouquier \cite{asymptotic_cell} constructed a tensor triangulated category which categorifies the modular datum attached to the non-trivial family of the cyclic spets. In the present article, we construct fusion categories with symmetric center equivalent to $\Rep(G,z)$ for $G$ a cyclic group and $z\in G$ satisfying $z^2=1$. If the order of $G$ is even, we can therefore construct a new category, which is enriched over superspaces, so that its Grothendieck group has negative structure constants in general. As shown in \cite{super_application}, this defines a $\mathbb{Z}$-modular datum.

Our construction is as follows: we consider the representation theory of a quantum double associated to the complex simple Lie algebra $\mathfrak{sl}_{n+1}$ at a root of unity of even order. Similarly to the usual construction for $\qgr{\mathfrak{g}}$ at a root of unity $\xi$, we consider the semisimplification of the full subcategory of tilting modules. A partial modularization of an integral subcategory gives then the category with symmetric center $\Rep(G,z)$, which is denoted by $\mathbb{Z}(\mathcal{T}_\xi)\rtimes \mathcal{S}$.

We investigate this construction in full generality considering any simple complex Lie algebra $\mathfrak{g}$ and any root of unity. For certain values of $\xi$, we are able to compute explicitely the symmetric center of the category $\mathbb{Z}(\mathcal{T}_\xi)\rtimes \mathcal{S}$, see Theorem \ref{thm:int_deg}. When $\mathfrak{g}$ is a simple Lie algbera of type $A$ and $\xi$ is a root of unity of even order, the category $\mathbb{Z}(\mathcal{T}_\xi)\rtimes \mathcal{S}$ gives rises to a $\mathbb{Z}$-modular datum which we relate to some of the Fourier matrices constructed by Malle. The main result of this paper is given by the following theorem:
\begin{thmintro}
  Let $n\geq 1$ and $d\geq n$ be two integers, and $\mathfrak{g}=\mathfrak{sl}_{n+1}$.

  If $n$ is even, the symmetric center of $\mathbb{Z}(\mathcal{T}_\xi)\rtimes \mathcal{S}$ is equal to the category $\Rep(\mathbb{Z}/(n+1)\mathbb{Z})$. If $n$ is odd, the symmetric center of $\mathbb{Z}(\mathcal{T}_\xi)\rtimes \mathcal{S}$ is equal to the category $\Rep\left(\mathbb{Z}/(n+1)\mathbb{Z},\frac{n+1}{2}\right)$.

  In both cases, there exists a family $\mathcal{F}$ of unipotent characters for the complex reflection group $G\left(d,1,\frac{n(n+1)}{2}\right)$ such that the Malle's $\mathbb{Z}$-modular data coincide with the $\mathbb{Z}$-modular data arising from the category $\mathbb{Z}(\mathcal{T}_\xi)\rtimes \mathcal{S}$.
\end{thmintro}

Moreover, some modular data constructed by Brou\'e, Malle and Michel for spetsial exceptional complex reflection groups can be categorified by considering similar categories in type $B$.

This article is organized as follows: in the first Section, we study the algebra $\qdbl{\mathfrak{g}}$ when $q$ is a generic parameter. Then, in the second Section, we study some finite dimensional representations of $\qdbl{\mathfrak{g}}$ and extra structure arising from the Hopf algebra $\qdbl{\mathfrak{g}}$: tensor product, braiding, pivotality. The third Section contains the main constructions of this paper: we specialize the algebra $\qdbl{\mathfrak{g}}$ at a root of unity $\xi$ and consider tilting modules for this algebra, which allows us to construct the category $\mathbb{Z}(\mathcal{T}_\xi)\rtimes \mathcal{S}$. In the fourth Section, we recall the Malle's $\mathbb{Z}$-modular datum and compare it with the modular datum arising from the category $\mathbb{Z}(\mathcal{T}_\xi)\rtimes \mathcal{S}$. Finally, in the last Section, we study some exceptional examples arising from complex reflection groups.

\subsection*{Acknowledgements}
  I warmly thank my advisor C. Bonnaf\'e for many fruitful discussions and his constant support and R. Rouquier for suggesting me to consider supercategories in this work. I also thank G. Malle for many valuable comments on a first version of this paper. The paper is partially based upon work supported by the NSF under grant DMS-1440140 while the author was in residence at the Mathematical Sciences Research Institute in Berkeley, California, during the Spring 2018 semester.

\section{Quantum double of Borel algebras}
\label{sec:qdbl}
In this section, we recall the construction of the quantum enveloping
algebra $\qgr{\mathfrak{g}}$ associated to a simple complex Lie
algebra $\mathfrak{g}$. Our main object of study will be the quantum
double of a Borel algebra $\mathfrak{b}$ of $\mathfrak{g}$, which
will be denoted by $\qdbl{\mathfrak{g}}$.

\subsection{Notations}
\label{sec:not}

Let $\mathfrak{g}$ be a simple complex Lie algebra of rank $n$. We fix a Cartan
subalgebra $\mathfrak{h}$ of $\mathfrak{g}$. Let $\Phi
\subseteq \mathfrak{h}^*$ be the set of roots of $\mathfrak{g}$ relative
to $\mathfrak{h}$, $\Pi=\{\alpha_1,\ldots,\alpha_n\}$ be a base of
$\Phi$, $\Phi^+$ be the positive roots relative to $\Pi$. Let
$\mathfrak{b^+}$ (resp. $\mathfrak{b^-}$) be the Borel subalgebra of
$\mathfrak{g}$ relative to $\Pi$ (resp. $-\Pi$). We denote by $h$ the Coxeter number of $\mathfrak{g}$ and by $h^\vee$ the dual Coxeter number of $\mathfrak{g}$.

Fix a symmetric bilinear form $\langle\cdot,\cdot\rangle$ on $\mathfrak{h}^*$
normalized such that $\langle\alpha,\alpha\rangle=2$ for short roots. For
$\alpha\in\Phi$, we define $\alpha^\vee\in\mathfrak{h}^*$ by
\[
  \alpha^\vee=\frac{2\alpha}{\langle\alpha,\alpha\rangle}.
\]

Let $Q$ be the root lattice, $P$ the weight lattice, $P^+$ the cone
of dominant weights, $Q^\vee$ the coroot lattice and $P^\vee$ the coweight lattice 
\begin{align*}
  P&=\{\lambda\in\mathfrak{h}^*\mid
     \langle\lambda,\alpha^\vee\rangle\in\mathbb{Z},\ \forall \alpha\in\Pi\}\\
  P^+&=\{\lambda\in\mathfrak{h}^*\mid
  \langle \lambda,\alpha^\vee\rangle\in\mathbb{N},\ \forall \alpha\in\Pi\}\\
  P^\vee&=\{\lambda\in\mathfrak{h}^*\mid
     \langle\lambda,\alpha\rangle\in\mathbb{Z},\ \forall \alpha\in\Pi\}
\end{align*}
Let $(\varpi_i)_{1\leq i \leq n}$ be the dual basis of $(\alpha_i^\vee)_{1\leq i\leq n}$ with respect to the form $\langle\cdot,\cdot\rangle$ and $\rho$ be the half sum of positive roots. We call the $\varpi_i$'s the fundamental weights. The element $\rho$ is in $P$ and is equal to the sum of fundamental weights \cite[VI.1.10 Proposition 29]{bourbaki_lie_456}.

We denote by $Q^+$ the monoid spanned by $\Pi$. We denote by $\leq$ the usual partial order on $P$: $\lambda\leq\mu$ if and only if $\mu-\lambda\in Q^+$.

For any $\alpha\in\Phi$, denote by $s_\alpha$ the reflection in the
hyperplane orthogonal to $\alpha$
\[
  s_{\alpha}(v) = v-\langle v,\alpha^\vee\rangle\alpha = v-\langle v,\alpha\rangle\alpha^\vee,\quad\forall v\in\mathfrak{h}^*.
\]
For $1\leq i \leq n$ denote by $s_i$ the reflection $s_{\alpha_i}$. All these reflections generate the Weyl group $W$ of $\mathfrak{g}$
and for any $w\in W$, $w(\Phi)=\Phi$. The form $\langle\cdot,\cdot\rangle$ is
then invariant with respect to $W$. We denote by $l(w)$ the length of $w\in W$ relative to the generating set $(s_{i})_{1\leq i \leq n}$.

For $q$ an indeterminate, define the following elements of $\mathbb{Z}[q,q^{-1}]$ 
\[
  \qint{n}{q}=\frac{q^n-q^{-n}}{q-q^{-1}}, n\in\mathbb{Z}, \quad 
  \qfact{n}{q}=\prod_{k=1}^{n}\qint{k}{q}, n\in\mathbb{N} \quad \text{and} 
  \quad \qbinom{k}{n}{q}=\prod_{i=1}^k\frac{\qint{n+1-i}{q}}{\qint{i}{q}}, n\in\mathbb{Z},k\in\mathbb{N}.
\]
We denote by $\qint{n}{\xi}$ (resp. $\qfact{n}{\xi}$, resp. $\qbinom{k}{n}{\xi}$) the evaluation of $\qint{n}{q}$ (resp. $\qfact{n}{q}$, resp. $\qbinom{k}{n}{q}$) at an invertible element $\xi$ of a ring. 
We will work over the field $\mathbb{Q}(q)$ and the ring $\mathcal{A}=\mathbb{Z}[q,q^{-1}]$. We define $q_i=q^{\frac{\langle\alpha_i,\alpha_i\rangle}{2}}$.

\subsection{Positive and negative parts of the usual quantum
  enveloping algebra}
\label{sec:pos_neg}

We start by defining the usual positive and negative part of the
quantum group $\qgr{\mathfrak{g}}$.

\begin{definition}
  \label{def:borel+}
  Let $\qgr{\mathfrak{b}^+}$ be the associative unital $\mathbb{Q}(q)$-algebra with generators $K_i^{\pm 1}$, $E_i$, $1\leq i \leq n$, and defining relations
\[
  K_i K_j=K_j K_i, \qquad K_i K_i^{-1} = 1
  = K_i^{-1}K_i,\quad \forall 1\leq i,j \leq n,
\]
\[
  K_i E_j = q^{\langle \alpha_i,\alpha_j\rangle}E_j K_i,\quad
  \forall 1\leq i,j \leq n,
\]
\[
  \sum_{r=0}^{1-\langle \alpha_i,\alpha_j^\vee\rangle}(-1)^r
  \qbinom{r}{1-\langle \alpha_i,\alpha_j^\vee\rangle}{q_i}
  E_i^{1-\langle \alpha_i,\alpha_j^\vee\rangle-r}E_j E_i^r = 0, 
  \qquad \forall 1\leq i\neq j \leq n.
\]
\end{definition}

There exist several Hopf algebra structures on $\qgr{\mathfrak{b}^+}$, and we choose the following comultiplication $\Delta$, counit $\varepsilon$ and antipode $S$:
\begin{align*}
  \Delta(K_i)&=K_i\otimes K_i & \varepsilon(K_i)
  &= 1 &  S(K_i) &= K_i^{-1},\\ 
  \Delta(E_i)&=1\otimes E_i + E_i \otimes K_i  
  & \varepsilon(E_i) &= 0 & S(E_i) &= -E_i K_i^{-1}.  
\end{align*}

\begin{definition}
  \label{def:borel-}
  Let $\qgr{\mathfrak{b}^-}$ be the associative unital $\mathbb{Q}(q)$-algebra with generators $L_i^{\pm 1}$, $F_i$, $1\leq i \leq n$, and defining relations
\[
  L_i L_j=L_j L_i, \qquad L_i L_i^{-1} = 1
  = L_i^{-1}L_i,\quad \forall 1\leq i,j \leq n,
\]
\[
  L_i F_j = q^{-\langle\alpha_i ,\alpha_j\rangle}F_j L_i,\quad
  \forall 1\leq i,j \leq n,
\]
\[
  \sum_{r=0}^{1-\langle\alpha_i ,\alpha_j^\vee\rangle}(-1)^r\qbinom{r}{1-\langle \alpha_i,\alpha_j^\vee\rangle}{q_i}F_i^{1-\langle \alpha_i,\alpha_j^\vee\rangle-r}F_j
  F_i^r = 0, \qquad \forall 1\leq i\neq j \leq n.
\]
\end{definition}

There exist severals Hopf algebra structures on $\qgr{\mathfrak{b}^-}$, and we choose the following comultiplication $\Delta$, counit $\varepsilon$ and antipode $S$:
\begin{align*}
  \Delta(L_i)&=L_i\otimes L_i & \varepsilon(L_i)
  &= 1 &  S(L_i) &= L_i^{-1},\\ 
  \Delta(F_i)&=L_i^{-1}\otimes F_i + F_i \otimes 1  
  & \varepsilon(F_i) &= 0 & S(F_i) &= -L_i F_i.  
\end{align*}

For $\lambda = \sum_{i=1}^n\lambda_i\alpha_i \in Q$, we denote by $K_\lambda$ (resp. $L_\lambda$) the element $\prod_{i=1}^n K_i^{\lambda_i}$ (resp. $\prod_{i=1}^n L_i^{\lambda_i}$).

\subsection{A Hopf pairing between $\qgr{\mathfrak{b}^+}$ and
  $\qgr{\mathfrak{b}^-}$}
\label{sec:pairing}

The following is due to Drinfeld \cite[Section 13]{drinfeld_ICM} and has been studied by Tanisaki \cite{tanisaki}.

\begin{proposition}
  \label{prop:bil_form}
  There exists a unique bilinear form $(\cdot,\cdot)\colon\qgr{\mathfrak{b}^+}\times\qgr{\mathfrak{b}^-}\rightarrow\mathbb{Q}(q)$ satisfying for all $x,x'\in\qgr{\mathfrak{b}^+}$, $y,y'\in\qgr{\mathfrak{b}^-}$ and $1\leq i \leq n$
  \begin{multicols}{2}
    \begin{enumerate}
    \item $(x,yy')=(\Delta(x),y\otimes y')$,
    \item $(xx',y)=(x'\otimes x,\Delta(y))$,
    \item $(x,1) = \varepsilon(x)$, $(1,y)=\varepsilon(y)$,
    \item $(K_i,L_j)=q^{\langle\alpha _i,\alpha_j\rangle}$,
    \item $(K_i,F_j) = 0 = (E_i,L_j)$,
    \item $(E_i,F_j) = \frac{\delta_{i,j}}{q_i-q_i^{-1}}$.
    \end{enumerate}
  \end{multicols}
\end{proposition}

This bilinear form endows the tensor product $\qgr{\mathfrak{b}^-}\otimes\qgr{\mathfrak{b}^+}$ with a Hopf algebra structure, which we denote by $\qdbl{\mathfrak{g}}$, see \cite[Section 8.2]{kilmyk}. Using Sweedler notation for the coproduct, the product of $y_1\otimes x_1$ and $y_2\otimes x_2$ in $\qgr{\mathfrak{b}^-}\otimes\qgr{\mathfrak{b}^+}$ is given by
\[
  (y_1\otimes x_1)(y_2\otimes x_2) = \sum_{(x_1)(y_2)}(x_1',y_2')(x_1''',S(y_2'''))y_1y_2''\otimes x_1''x_2.
\] 
Consequently, both $\qgr{\mathfrak{b}^-}\otimes 1$ and $1\otimes\qgr{\mathfrak{b}^+}$ are Hopf subalgebras of $\qdbl{\mathfrak{g}}$. In the following, we identify $\qgr{\mathfrak{b}^-}$ and $\qgr{\mathfrak{b}^+}$ with their images in $\qdbl{\mathfrak{g}}$ and will write $yx$ instead of $y\otimes x$.

\begin{proposition}
  The algebra $\qdbl{\mathfrak{g}}$ is the associative unital $\mathbb{Q}(s)$-algebra with generators $K_i^{\pm 1}$, $L_i^{\pm 1}$, $E_i$ and $F_i$, $1\leq i \leq n$ and defining relations
\begin{align*}
  K_i K_j&=K_j K_i, & K_i K_i^{-1} &= 1 = K_i^{-1}K_i,\\
  L_i L_j&=L_j L_i, & L_i L_i^{-1} &= 1 = L_i^{-1}L_i,
\end{align*}
\[
  K_i L_j = L_j K_i,
\]
\begin{align*}
  K_i E_j &= q^{\langle\alpha _i,\alpha_j\rangle}E_j K_i,
  & K_i F_j &= q^{-\langle\alpha _i,\alpha_j\rangle}F_j K_i,\\
  L_i E_j &= q^{\langle\alpha _i,\alpha_j\rangle}E_j L_i,
  & L_i F_j &=q^{-\langle\alpha _i,\alpha_j\rangle}F_j L_i,
\end{align*}
\[
  [E_i,F_j] = \delta_{i,j}\frac{K_i-L_i^{-1}}{q_i-q_i^{-1}},
\]
\begin{align*}
  \sum_{r=0}^{1-\langle\alpha _i,\alpha_j^\vee\rangle}(-1)^r
  \qbinom{r}{1-\langle\alpha _i,\alpha_j^\vee\rangle}{q_i}
  E_i^{1-\langle\alpha _i,\alpha_j^\vee\rangle-r}E_j E_i^r
  &= 0, \qquad 1\leq i\neq j \leq n,\\
  \sum_{r=0}^{1-\langle\alpha _i,\alpha_j^\vee\rangle}(-1)^r
  \qbinom{r}{1-\langle\alpha _i,\alpha_j^\vee\rangle}{q_i}
  F_i^{1-\langle\alpha _i,\alpha_j^\vee\rangle-r}F_j F_i^r
  &= 0, \qquad 1\leq i\neq j \leq n.
\end{align*}
\end{proposition}


The elements $z_i = K_i L_i^{-1}$ are central and $\qgr{\mathfrak{g}}$ is the quotient of $\qdbl{\mathfrak{g}}$ by the Hopf ideal generated by $(z_i-1)_{1\leq i \leq n}$.

Denote by $\qdbl{\mathfrak{g}}^{<0}$ (resp. $\qdbl{\mathfrak{g}}^{0}$, resp. $\qdbl{\mathfrak{g}}^{>0}$) the subalgebra of $\qdbl{\mathfrak{g}}$
generated by $(E_i)_{1\leq i \leq n}$ (resp. $(K_i, L_i)_{1\leq i \leq n}$, resp.
$(F_i)_{1\leq i \leq n}$). Multiplication yields an isomorphism of vector spaces
$\qdbl{\mathfrak{g}}^{<0}\otimes \qdbl{\mathfrak{g}}^{0} \otimes
\qdbl{\mathfrak{g}}^{>0} \simeq \qdbl{\mathfrak{g}}$. It is worth mentioning that the coproduct and antipode of $\qdbl{\mathfrak{g}}$ do not restrict to $\qdbl{\mathfrak{g}}^{<0}$ or $\qdbl{\mathfrak{g}}^{>0}$.

The square of the antipode is given by the conjugation by any element of the form $L_{\lambda}K_{2\rho-\lambda}$, $\lambda\in Q$ as it is easily checked on the generators.

\subsection{Graduation}
\label{sec:grad}

There exists a $Q$-graduation on $\qdbl{\mathfrak{g}}$ given by $\deg(K_i)=0=\deg(L_i)$, $\deg(E_i)=\alpha_i$ and $\deg(F_i)=-\alpha_i$. For $\lambda\in Q$, we denote by $\qdbl{\mathfrak{g}}_\lambda$ the homogeneous elements of degree $\lambda$. We have
\[
  \qdbl{\mathfrak{g}}_\lambda=\left\{v\in\qdbl{\mathfrak{g}}\ \middle\vert\  K_i v = q^{\langle\lambda,\alpha_i\rangle}vK_i,\ \forall 1\leq i \leq n\right\}.
\]

The coproduct, the counit and the antipode respect the
grading. Moreover, the same lemma as \cite[lemma 4.12]{jantzen} shows
that for $\mu\in Q$, $\mu\geq 0$
\[
  \Delta\left(\qdbl{\mathfrak{g}}^{>0}_\mu\right)
  \subseteq
  \bigoplus_{0\leq \nu\leq \mu} \qdbl{\mathfrak{g}}^{>0}_\nu
  \otimes K_\nu\qdbl{\mathfrak{g}}^{>0}_{\mu-\nu}
\]
and
\[
  \Delta\left(\qdbl{\mathfrak{g}}^{<0}_{-\mu}\right)
  \subseteq
  \bigoplus_{0\leq \nu\leq \mu} L_\nu^{-1}\qdbl{\mathfrak{g}}^{<0}_{-(\mu-\nu)}
  \otimes \qdbl{\mathfrak{g}}^{<0}_{-\nu}.
\]
Therefore, for $x\in\qdbl{\mathfrak{g}}^{>0}_\mu$, there are elements
$r_i(x)$ and $r_i'(x)$ in
$\qdbl{\mathfrak{g}}^{>0}_{\mu-\alpha_i}$ such that

\begin{equation}
  \Delta(x) \in 1\otimes x + \sum_{i=1}^n E_i\otimes K_ir_i(x)
  + \sum_{\substack{0 < \nu \leq \mu \\ \nu\not\in\Pi}}
  \qdbl{\mathfrak{g}}^{>0}_\nu \otimes
  K_\nu\qdbl{\mathfrak{g}}^{>0}_{\mu-\nu},\label{eq:coprod+}  
\end{equation}
and
\begin{equation}
    \Delta(x) \in x\otimes K_\mu + \sum_{i=1}^n r_i'(x)\otimes K_{\mu-\alpha_i}E_i
  + \sum_{\substack{0 < \nu \leq \mu \\ \mu-\nu\not\in\Pi}}
  \qdbl{\mathfrak{g}}^{>0}_\nu \otimes
  K_\nu\qdbl{\mathfrak{g}}^{>0}_{\mu-\nu}.\label{eq:coprod-}  
\end{equation}

Similarly, we can define $\rho_i(y)$ and $\rho_i'(y)$ for
$y\in\qdbl{\mathfrak{g}}^{<0}_{-\mu}$ by the following
\[
  \Delta(y) \in L_\mu^{-1}\otimes y + \sum_{i=1}^n L_{\mu-\alpha_i}^{-1}F_i\otimes \rho_i(y)
  + \sum_{\substack{0< \nu\leq \mu \\ \mu-\nu\not\in\Pi}} L_\nu^{-1}\qdbl{\mathfrak{g}}^{<0}_{-(\mu-\nu)}
  \otimes \qdbl{\mathfrak{g}}^{<0}_{-\nu},
\]
and
\[
    \Delta(y) \in y\otimes 1 + \sum_{i=1}^n  L_{\alpha_i}^{-1}\rho_i'(y)\otimes F_i
  + \sum_{\substack{0< \nu\leq \mu \\ \nu\not\in\Pi}} L_\nu^{-1}\qdbl{\mathfrak{g}}^{<0}_{-(\mu-\nu)}
  \otimes \qdbl{\mathfrak{g}}^{<0}_{-\nu}.
\]

The values of $r_i$, $r_i'$, $\rho_i$ and $\rho_i'$ can be computed by
induction. Compare with \cite[6.14,6.15]{jantzen}.

\begin{lemma}
  \label{lem:r_i-mult}
  Let $\mu,\mu' \in Q$, $\mu \geq 0$ and $\mu'\geq 0$.
  \begin{enumerate}
  \item For all $x\in\qdbl{\mathfrak{g}}_{\mu}^{>0}$ and $x'\in\qdbl{\mathfrak{g}}_{\mu'}^{>0}$,
    \[
      r_i(xx')=q^{-\langle\alpha_i,\mu\rangle}xr_i(x')+r_i(x)x', \quad
      r_i'(xx')=xr_i'(x')+q^{-\langle\alpha_i,\mu'\rangle}r_i'(x)x'.
    \]
    \item For all $y\in\qdbl{\mathfrak{g}}_{-\mu}^{<0}$ and $y'\in\qdbl{\mathfrak{g}}_{-\mu'}^{<0}$,
    \[
      \rho_i(yy')=y\rho_i(y')+q^{-\langle\alpha_i,\mu'\rangle}\rho_i(y)y', \quad
      \rho_i'(yy')=q^{-\langle\alpha_i,\mu\rangle}y\rho_i'(y')+\rho_i'(y)y'.
    \]
  \end{enumerate}
\end{lemma}

%

These elements can also be used to compute some commutators.

\begin{lemma}
  \label{lem:commutator}
  Let $\mu\in Q$, $\mu\geq 0$. Let
  $x\in\qdbl{\mathfrak{g}}^{>0}_{\mu}$ and $y \in
  \qdbl{\mathfrak{g}}^{<0}_{-\mu}$. Then
  \[
    xF_i - F_ix = \frac{K_ir_i(x)-r_i'(x)L_i^{-1}}{q_i-q_i^{-1}}
  \]
  and
  \[
    E_iy - yE_i = \frac{\rho_i(y)K_i-L_i^{-1}\rho_i'(y)}{q_i-q_i^{-1}}.
  \]
\end{lemma}

\begin{proof}
  We only show the first formula, the second one is proven
  similarly. If $x = 1$ or $x = E_i$, the formula is satisfied. Now,
  supposing that it is true for $x$ and $x'$, we show it for $xx'$:
  \begin{align*}
    xx'F_i-F_ixx'
     &= x(x'F_i-F_ix')+(xF_i-F_ix)x'\\
     &= (q_i-q_i^{-1})^{-1}(x(K_ir_i(x')-r_i'(x')L_i^{-1}) +
       (K_ir_i(x)-r_i'(x)L_i^{-1})x')\\
     &=
       (q_i-q_i^{-1})^{-1}(K_i(q^{-\langle\alpha_i,\mu\rangle}xr_i(x')
       + r_i(x)x') -
       (xr_i'(x')+q^{-\langle\alpha_i,\mu'\rangle}r_i'(x)x'))\\
     &=(q_i-q_i^{-1})^{-1}(K_ir_i(xx') - r_i'(xx')L_i^{-1}),
  \end{align*}
  where we used Lemma \ref{lem:r_i-mult} in the last equality.
\end{proof}

\subsection{Some properties of the pairing}
\label{sec:prop-pair}

We gather here some well known properties of the pairing.

\begin{proposition}
  \label{prop:Dnul}
  Let $x \in \qdbl{\mathfrak{g}}^{>0}$, $y\in\qdbl{\mathfrak{g}}^{<0}$ and $\lambda,\mu \in Q$. Then
  \begin{enumerate}
  \item $(K_\lambda x,L_\mu y)=(K_\lambda,L_\mu)(x,y)$.
  \item Suppose $\lambda,\mu\geq 0$. If $x$ is of weight $\lambda$ and $y$ of weight $-\mu$ then $(x,y)=0$ if $\lambda \neq \mu$.
  \end{enumerate}
\end{proposition}

\begin{proof}
  For the first equality, it suffices to show it for $x$ and $y$
  homogeneous. It then follows easily from (\ref{eq:coprod+}),
  (\ref{eq:coprod-}) and the fact that $(K_\lambda,y')=0=(x',L_\mu)$
  for any $x'\in\qdbl{\mathfrak{g}}^{>0}$ and
  $y'\in\qdbl{\mathfrak{g}}^{<0}$.

  For the second assertion, we proceed by induction on
  $\h(\lambda)=\sum_{i=1}^n\lambda_i$, where
  $\lambda=\sum_{i=1}^n\lambda_i\alpha_i$. We can suppose that $x$ is
  a product of the generators $E_i$. The case $\h(\lambda)=1$,
  \emph{i.e.} $\lambda\in\Pi$ is easily proved by induction on
  $\h(\mu)$. Then, writing $x=E_ix'$ and
  \[
    \Delta(y) =
    \sum_{0\leq\nu\leq\mu}L_{\nu}^{-1}y'_{-(\mu-\nu)}\otimes
    y''_{-\nu},\ y'_{-\nu},y''_{-\nu}\in\qdbl{\mathfrak{g}}^{<0}_{-\nu}, 
  \]
  one has
  \begin{align*}
    (E_ix',y) = \sum_{0\leq\nu\leq\mu}
    (x',L_{\nu}^{-1}y'_{-(\mu-\nu)})(E_i,y''_{-\nu})
    = (x',y'_{-(\mu-\alpha_i)})(E_i,y''_{-\alpha_i}).
  \end{align*}
  As $\lambda\neq \mu$, the two terms can not be simultaneously non-zero.
\end{proof}

Using these two facts and an easy induction, one can show that, for $1 \leq i \leq n$, and $r\in\mathbb{N}$,
\[
  (E_i^r,F_i^r)=\frac{q_i^{-\frac{r(r-1)}{2}}\qfact{r}{q_i}}{(q_i-q_i^{-1})^r}.
\]

We now turn to the compatibility between the pairing and the $r_i$, $r_i'$,
$\rho_i$ and $\rho_i'$ defined in section \ref{sec:grad}.

\begin{lemma}
  \label{lem:bil_r_rho}
  Let $x\in\qdbl{\mathfrak{g}}^{>0}_{\mu}$ and $y \in
  \qdbl{\mathfrak{g}}^{<0}_{-\mu}$. One has
  \begin{align*}
    (x,F_iy) &= (E_i,F_i)(r_i(x),y), & (x,yF_i)&=(E_i,F_i)(r_i'(x),y),\\
    (E_ix,y) &= (E_i,F_i)(x,\rho_i'(y)), & (xE_i,y) &= (E_i,F_i)(x,\rho_i(y)). 
  \end{align*}
\end{lemma}

\begin{proof}
  It follows easily from (\ref{eq:coprod+}),
  (\ref{eq:coprod-}) and Proposition \ref{prop:Dnul}.
\end{proof}

We end this section with \cite[Proposition 2.1.4]{tanisaki}.

\begin{proposition}
  The restriction of the pairing to $\qdbl{\mathfrak{g}}^{>0}\times\qdbl{\mathfrak{g}}^{<0}$ is non-degenerate.
\end{proposition}

\subsection{Basis for $\qdbl{\mathfrak{g}}$}
\label{sec:basis}

We introduce the divided power $E_i^{(r)}$ and $F_i^{(r)}$ for $1\leq i \leq n$ and $r \leq 0$:
\[
  E_i^{(r)}=\frac{E_i^r}{\qint{r}{q_i}}
  \quad\mathrm{and}\quad
  F_i^{(r)}=\frac{F_i^r}{\qint{r}{q_i}}.
\]
The following is the analogue of \cite[Theorem 3.1]{lusztig_qgrp_roots} for $\qdbl{\mathfrak{g}}$.

\begin{proposition}
  For all $1\leq i \leq n$, there exists a unique $\mathbb{Q}(q)$-algebra isomorphism $T_i$ such that:
\[
  T_i(K_\lambda)=K_{s_i(\lambda)} \quad T_i(L_\lambda)=L_{s_i(\lambda)},\ \forall\lambda\in Q,
\]
and, for $1\leq j \leq n$, setting $r=-\langle\alpha _j,\alpha_i^\vee\rangle$,
\begin{align*}
  T_i(E_j)&=
  \begin{dcases}
    -F_i L_i & \text{if}\ i=j,\\
    \sum_{k=0}^{r}(-1)^kq_i^{-k}E_i^{(r-k)}E_j E_i^{(k)} & \text{otherwise},
  \end{dcases}\\
  T_i(F_j)&=
  \begin{dcases}
    -K_i^{-1} E_i & \text{if}\ i=j,\\
    \sum_{k=0}^{r}(-1)^kq_i^kF_i^{(k)}F_j F_i^{(r-k)} & \text{otherwise},
  \end{dcases}
\end{align*}
\end{proposition}

We also have the analogue of \cite[Theorem 3.2]{lusztig_qgrp_roots} for $\qdbl{\mathfrak{g}}$.

\begin{proposition}
  The $(T_i)_{1 \leq i \leq n}$ satisfy the braid relations and therefore define a morphism from the braid group of $W$ to the algebra automorphisms of $\qdbl{\mathfrak{g}}$.
\end{proposition}

Choose now a reduced decomposition $w_0=s_{i_1}s_{i_2}\cdots s_{i_r}$ of the longest element $w_0$ of $W$. The positive roots are then
\[
  \Phi^+ = \left\{s_{i_1}s_{i_2}\cdots s_{i_k}(\alpha _{i_{k+1}})\ \middle\vert\  0\leq k \leq r-1\right\}.
\]
For any $\alpha=s_{i_1}s_{i_2}\cdots s_{i_k}(\alpha _{i_{k+1}})\in\Phi^+$ define
\[
  E_\alpha=T_{i_1}T_{i_2}\cdots T_{i_{k}}(E_{i_{k+1}}) 
  \quad\mathrm{and}\quad
  F_\alpha=T_{i_1}T_{i_2}\cdots T_{i_{k}}(F_{i_{k+1}}).
\]
Note that $E_\alpha$ (resp $F_\alpha$) is homogeneous of degree $\alpha$ (resp. $-\alpha$). The choice of a reduced decomposition of $w_0$ gives an order on $\Phi^+$, namely $s_{i_1}s_{i_2}\cdots s_{i_k}(\alpha _{i_{k+1}}) \preccurlyeq s_{i_1}s_{i_2}\cdots s_{i_{k+1}}(\alpha _{i_{k+2}})$ for every $k$. All products will be ordered with respect to this order.

\begin{proposition}[{\cite[Theorem 6.7]{lusztig_qgrp_roots}}]
  The elements
\[
  \prod_{\alpha\in\Phi^+}E_\alpha^{(n_\alpha)}, \ n_\alpha\in\mathbb{N}
\]
form a $\mathbb{Q}(q)$-basis of $\qdbl{\mathfrak{g}}^{>0}$.

The elements
\[
  \prod_{\alpha\in\Phi^+}F_\alpha^{(n_\alpha)}, \ n_\alpha\in\mathbb{N}
\]
form a $\mathbb{Q}(q)$-basis of $\qdbl{\mathfrak{g}}^{<0}$.
\end{proposition}

We compute now the dual basis of $\qdbl{\mathfrak{g}}^{>0}$ with
respect to $(\cdot,\cdot)$. The following can be found in \cite[8.29]{jantzen}.

\begin{proposition}
  For any $\alpha\in\Phi^+$, let $a_\alpha,b_\alpha \in \mathbb{N}$. We have
\[
  \left(\prod_{\alpha\in\Phi^+}E_\alpha^{a_\alpha},\prod_{\alpha\in\Phi^+}F_\alpha^{b_\alpha}\right)=\prod_{\alpha\in\Phi^+}\delta_{a_\alpha,b_\alpha}\frac{q_\alpha^{-\frac{a_\alpha(a_\alpha-1)}{2}}\qfact{a_\alpha}{q_\alpha}}{(q_\alpha-q_\alpha^{-1})^{a_\alpha}}.
\]
\end{proposition}

Therefore, the dual basis of
$\left(\prod_{\alpha\in\Phi^+}E_\alpha^{(n_\alpha)}\right)_{(n_\alpha)\in
  \mathbb{N}^{\Phi^+}}$ is
\[
  \left(\prod_{\alpha\in\Phi^+}q_\alpha^{\frac{n_\alpha(n_\alpha-1)}{2}}\qfact{n_\alpha}{q_\alpha}(q_\alpha-q_\alpha^{-1})^{n_\alpha}F_\alpha^{(n_\alpha)}\right)_{(n_\alpha)\in\mathbb{N}^{\Phi^+}}.
\]

\subsection{Quasi-$R$-matrix}
\label{sec:quasi-r-matrix}

The algebra $\qdbl{\mathfrak{g}}$ is not
quasi-triangular. Nevertheless, we construct a  quasi-$R$-matrix which
will endow the usual category of modules with a braiding (see
section \ref{sec:braiding}). We adapt the exposition of Jantzen
\cite[Chapter 7]{jantzen} to the case of $\qdbl{\mathfrak{g}}$.

For any $\mu \in Q$, $\mu\geq 0$, fix a basis $(u_i^{\mu})_{i\in I_{\mu}}$
of $\qdbl{\mathfrak{g}}^{>0}_{\mu}$ and let $(v_i^{\mu})_{i\in I_{\mu}}$
the dual basis in $\qdbl{\mathfrak{g}}^{<0}_{-\mu}$ with respect to the
pairing $(\cdot,\cdot)$. Set
\[
  \Theta_{\mu}=\sum_{i\in I_{\mu}}u_i^\mu\otimes v_i^\mu \in
  \qdbl{\mathfrak{g}}\otimes\qdbl{\mathfrak{g}}.
\]
Note that this does not depend on the choice of the basis of
$\qdbl{\mathfrak{g}}_\mu$. A
homogeneous basis of $\qdbl{\mathfrak{g}}^{>0}$ and its dual with
respect to $(\cdot,\cdot)$ have already been computed above. Following \cite[A.1]{rosso}, we
define an algebra automorphism $\Psi$ of
$\qdbl{\mathfrak{g}}\otimes\qdbl{\mathfrak{g}}$ by
\begin{align*}
  \Psi(K_i\otimes 1) &= K_i\otimes 1,
  & \Psi(1\otimes K_i) &= 1\otimes K_i,\\
  \Psi(L_i\otimes 1) &= L_i\otimes 1,
  & \Psi(1\otimes L_i) &= 1\otimes L_i,\\
  \Psi(E_i\otimes 1) &= E_i\otimes L_i^{-1},
  & \Psi(1\otimes E_i) &= K_i^{-1}\otimes E_i,\\
  \Psi(F_i\otimes 1) &= F_i\otimes L_i,
  & \Psi(1\otimes F_i) &= K_i\otimes F_i.                  
\end{align*}

\begin{proposition}
  \label{prop:coprod_r_mat}
  Let $\mu\in Q$, $\mu \geq 0$. We have
  \begin{align*}
    \Theta_\mu(K_i\otimes K_i) &= \Psi(K_i\otimes K_i)\Theta_\mu,\\
    \Theta_\mu(L_i\otimes L_i) &= \Psi(L_i\otimes L_i)\Theta_\mu,\\
    \Theta_\mu(1\otimes E_i) + \Theta_{\mu-\alpha_i}(E_i\otimes K_i)
    &= \Psi(E_i\otimes 1)\Theta_{\mu-\alpha_i} +
      \Psi(K_i\otimes E_i)\Theta_\mu,\\
    \Theta_\mu(F_i\otimes 1) + \Theta_{\mu-\alpha_i}(L_i^{-1}\otimes F_i)
    &= \Psi(1\otimes F_i)\Theta_{\mu-\alpha_i} +
      \Psi(F_i\otimes L_i^{-1})\Theta_\mu.      
  \end{align*}
\end{proposition}

\begin{proof}
  We follow closely the proof of \cite[Lemma 7.1]{jantzen}. The first two
  equations are trivial. We will use the following fact: for any
  $\mu\in Q$, $\mu\geq 0$, and any
  $x\in\qdbl{\mathfrak{g}}_{\mu}^{>0}$ and
  $y\in\qdbl{\mathfrak{g}}_{-\mu}^{<0}$ we have
  \[
    x = \sum_{i}(x,v_i^{\mu})u_i^{\mu}
    \quad\mathrm{and}\quad
    y = \sum_{i}(u_i^{\mu},y)v_i^{\mu}.
  \]
  We set $c_i=(q_i-q_i^{-1})^{-1}$, and we have
  \begin{align*}
    \left(1\otimes E_i\right)&\Theta_\mu
    -\Theta_\mu\left(1\otimes E_i\right)
      =\sum_{j}u_j^{\mu}\otimes\left(E_iv_j^{\mu}-v_j^{\mu}E_i\right) &&\\
    &=
      c_i\sum_{j}u_j^{\mu}\otimes\left(\rho_i\left(v_j^{\mu}\right)K_i-L_i^{-1}\rho_i'\left(v_j^{\mu}\right)\right)&&\text{by
    Lemma \ref{lem:commutator}}\\
    &=c_i\sum_{j,k}u_j^{\mu}\otimes\left[\left(u_k^{\mu-\alpha_i},\rho_i\left(v_j^{\mu}\right)\right)v_k^{\mu-\alpha_i}K_i
      -
      \left(u_k^{\mu-\alpha_i},\rho_i'\left(v_j^{\mu}\right)\right)L_i^{-1}v_k^{\mu-\alpha_i}\right]&&\\
    &= c_i\sum_{j,k}u_j^{\mu}\otimes\left[\left(u_k^{\mu-\alpha_i}E_i,v_j^{\mu}\right)v_k^{\mu-\alpha_i}K_i
      -
      \left(E_iu_k^{\mu-\alpha_i},v_j^{\mu}\right)L_i^{-1}v_k^{\mu-\alpha_i}\right]
    &&\text{by Lemma \ref{lem:bil_r_rho}}\\
    &= \sum_{k}\left(u_k^{\mu-\alpha_i}E_i\otimes
      v_k^{\mu-\alpha_i}K_i - E_iu_k^{\mu-\alpha_i}\otimes
      L_i^{-1}v_k^{\mu-\alpha_i}\right)&&\\
    &= \Theta_{\mu-\alpha_i} \left(E_i\otimes K_i\right) - \left(E_i\otimes L_i^{-1}\right)\Theta_{\mu-\alpha_i},&&
  \end{align*}
as expected because $\Psi(E_i\otimes 1) = E_i\otimes L_i^{-1}$ and
$\Psi(K_i\otimes E_i) = 1\otimes E_i$. A similar calculation shows the fourth formula.
\end{proof}

Define, as in \cite[Chapter 4]{lusztig_qgr}, the completion
$\qdbl{\mathfrak{g}}\hat{\otimes}\qdbl{\mathfrak{g}}$ of
$\qdbl{\mathfrak{g}}\otimes\qdbl{\mathfrak{g}}$ with respect to the
descending sequence of spaces
\[
  (\qdbl{\mathfrak{g}}\otimes\qdbl{\mathfrak{g}})_N=\qdbl{\mathfrak{g}}\otimes
  \sum_{\h(\mu)\geq N}\qdbl{\mathfrak{g}}^{>0}\qdbl{\mathfrak{g}}^0\qdbl{\mathfrak{g}}^{<0}_{-\mu} + \sum_{\h(\mu)\geq N}\qdbl{\mathfrak{g}}^{<0}\qdbl{\mathfrak{g}}^0\qdbl{\mathfrak{g}}^{>0}_\mu\otimes \qdbl{\mathfrak{g}}.
\]

The morphism $\Psi$ extends by continuity to
$\qdbl{\mathfrak{g}}\hat{\otimes}\qdbl{\mathfrak{g}}$ and we consider
the following element of
$\qdbl{\mathfrak{g}}\hat{\otimes}\qdbl{\mathfrak{g}}$
\[
  \Theta = \sum_{\substack{\mu\in Q\\ \mu \geq 0}} \Theta_\mu.
\]

Then we can rewrite Proposition \ref{prop:coprod_r_mat} as
\[
  \Theta \Delta(u) = \left(\Psi\circ\Delta^{\mathrm{op}}\right)(u)\Theta
\]
for any $u \in \qdbl{\mathfrak{g}}$.

We also have the analogue of \cite[Lemma 7.4]{jantzen}.

\begin{lemma}
  For $\mu \in Q$, $\mu\geq 0$, we have
  \[
    (\Delta\otimes \id)(\Theta_\mu) = \sum_{0 \leq \nu \leq \mu}
    (\Theta_\nu)_{13}(1\otimes K_\nu \otimes 1)(\Theta_{\mu-\nu})_{23},
  \]
  and
  \[
    (\id\otimes \Delta)(\Theta_\mu) = \sum_{0 \leq \nu \leq \mu}
    (\Theta_\nu)_{13}(1\otimes L_\nu^{-1}\otimes 1)(\Theta_{\mu-\nu})_{12}.
  \]
\end{lemma}

\begin{proof}
  First, note that for any $x\in\qdbl{\mathfrak{g}}_{\mu}^{>0}$
  \[
    \Delta(x) = \sum_{\substack{0\leq \nu \leq \mu \\ i,j}} \left(x,v_i^{\nu}v_j^{\mu-\nu}\right)u_i^\nu\otimes K_\nu u_j^{\mu-\nu},
  \]
  and for any $y\in\qdbl{\mathfrak{g}}_{-\mu}^{<0}$
  \[
    \Delta(y) = \sum_{\substack{0\leq \nu \leq \mu \\ i,j}} \left(u_i^{\nu}u_j^{\mu-\nu},y\right)L_\nu^{-1}v_j^{\mu-\nu}\otimes v_i^{\nu}.
  \]
  Therefore
  \begin{align*}
    (\Delta\otimes \id)\left(\Theta_\mu\right)
    &= \sum_{k}\Delta\left(u_k^{\mu}\right)\otimes v_k^{\mu} \\
    &=\sum_{\substack{0\leq \nu \leq \mu\\i,j,k}}
    \left(u_k^\mu,v_i^\nu v_j^{\mu-\nu}\right)u_i^\nu\otimes K_\nu
    u_j^{\mu-\nu}\otimes v_k^\mu\\
    &=\sum_{\substack{0\leq \nu \leq \mu\\i,j}}u_i^\nu\otimes K_\nu
    u_j^{\mu-\nu}\otimes v_i^\nu v_j^{\mu-\nu}\\
    &=\sum_{0 \leq \nu \leq \mu}
    (\Theta_\nu)_{13}(1\otimes K_\nu \otimes 1)(\Theta_{\mu-\nu})_{23}.
  \end{align*}
  The proof of the other formula is similar.
\end{proof}

We can translate this last lemma as equalities in $\qdbl{\mathfrak{g}}\hat{\otimes}\qdbl{\mathfrak{g}}$. Note that for any $x \in \qdbl{\mathfrak{g}}_{\mu}^{>0}$ and $y \in \qdbl{\mathfrak{g}}_{-\mu}^{<0}$, we have
\begin{align*}
  \Psi(x\otimes 1) &= x\otimes L_{\mu}^{-1}, & \Psi(1\otimes x) &= K_\mu^{-1}\otimes x,\\ 
  \Psi(y\otimes 1) &= y\otimes L_{\mu}, & \Psi(1\otimes y) &= K_\mu\otimes y.
\end{align*}
Therefore, we have
\begin{align*}
  (\Delta\otimes \id)(\Theta) 
  &= \sum_{\mu \geq 0}\sum_{\nu+\eta=\mu}\left(\Theta_\eta\right)_{13} \left(1\otimes K_\eta \otimes 1\right) \left(\Theta_\nu\right)_{23} \\
  &= \sum_{\mu \geq 0}\sum_{\nu+\eta=\mu} \Psi_{23}\left(\left(\Theta_{\eta}\right)_{13}\right)\left(\Theta_{\nu}\right)_{23}\\
  &= \Psi_{23}\left(\Theta_{13}\right)\Theta_{23}.
\end{align*}
Similarly, we have
\[
  (\id\otimes \Delta)(\Theta) = \Psi_{12}\left(\Theta_{13}\right)\Theta_{12}.
\]

We now show that the element $\Theta$ is invertible. Set $\Gamma_\mu = (S\otimes \id)(\Theta_\mu)(K_\mu\otimes 1)$ and $\Gamma=\sum_{\mu\geq 0}\Gamma_\mu$.

\begin{lemma}
  We have $\Gamma\Theta = 1 = \Theta\Gamma$ in $\qdbl{\mathfrak{g}}\hat{\otimes}\qdbl{\mathfrak{g}}$.
\end{lemma}

\begin{proof}
  We show that for all $\mu\geq 0$, 
\[
  \sum_{\lambda+\nu=\mu}\Gamma_\lambda\Theta_\nu = \delta_{\mu,0},
\]
and
\[
  \sum_{\lambda+\nu=\mu}\Theta_\lambda\Gamma_\nu = \delta_{\mu,0}.
\]

We start with the first formula. We may and will suppose that $\mu>0$. As
\[
  \sum_{\lambda+\nu=\mu}\Gamma_\lambda\Theta_\nu = \sum_{\substack{\lambda+\nu=\mu\\i,j}}S(u_i^\lambda)K_\lambda u_j^{\nu} \otimes v_i^\lambda v_j^\nu
\]
is in $\qdbl{\mathfrak{g}}\otimes\qdbl{\mathfrak{g}}_{-\mu}^{<0}$, it suffices to show that applying $\id\otimes (x,\cdot)$ gives zero, for all $x\in\qdbl{\mathfrak{g}}_{\mu}^{>0}$. But 
\[
  \Delta(x) = \sum_{\substack{\lambda+\nu=\mu \\ i,j}} \left(x,v_i^{\lambda}v_j^{\nu}\right)u_i^\lambda\otimes K_\lambda u_j^{\nu},
\]
so the antipode axiom gives 
\[
  0 = \varepsilon(x) = \sum_{\substack{\lambda+\nu=\mu \\ i,j}} \left(x,v_i^{\lambda}v_j^{\nu}\right)S(u_i^\lambda) K_\lambda u_j^{\nu}
\]
as desired.

For the second equality, we again apply $\id\otimes (x,\cdot)$ to show that
\[
  \sum_{\lambda+\nu=\mu}\Theta_\lambda\Gamma_\nu = \sum_{\substack{\lambda+\nu=\mu\\i,j}}u_i^{\lambda} S(u_j^\nu)K_\nu  \otimes  v_i^\lambda v_j^\nu =0.
\]
The antipode axiom also gives
\[
  0 = \varepsilon(x) = \sum_{\substack{\lambda+\nu=\mu \\ i,j}} \left(x,v_i^{\lambda}v_j^{\nu}\right)u_i^\lambda  S(u_j^{\nu})K_{-\lambda}
\]
which is, after multiplication by $K_\mu$, the result expected.
\end{proof}

We therefore have proven:

\begin{proposition}
  The element $\Theta\in\qdbl{\mathfrak{g}}\hat\otimes\qdbl{\mathfrak{g}}$ is invertible and
  \begin{itemize}
  \item for all $u\in\qgr{\mathfrak{g}}$ we have $\Theta\Delta(u) = (\Psi\circ\Delta^{\mathrm{op}})(u)\Theta$,
    \item $(\Delta\otimes\id)(\Theta)=\Psi_{23}(\Theta_{13})\Theta_{23}$,
    \item $(\id\otimes\Delta)(\Theta)=\Psi_{12}(\Theta_{13})\Theta_{12}$.
  \end{itemize}
\end{proposition}

Finally, we give an explicit form of $\Theta$ (compare with \cite[10.1.D]{chari-pressley})
\[
  \Theta=\prod_{\alpha\in\Phi^+} \left(\sum_{n=0}^{+\infty}q_\alpha^{\frac{n(n-1)}{2}}\qfact{n}{q_\alpha}(q_\alpha-q_\alpha^{-1})^{n} E_\alpha^{(n)}\otimes F_\alpha^{(n)}\right).
\]


\section{Representation theory at a generic parameter $q$}
\label{sec:rep_q_gen}
We now turn to the representation theory of the quantum double $\qdbl{\mathfrak{g}}$, which is quite similar to the one of $\qgr{\mathfrak{g}}$.

\subsection{Representations of $\qdbl{\mathfrak{g}}$}
\label{sec:rep_qgen}

We will consider the category $\mathcal{C}_q$ of finite dimensional $\qdbl{\mathfrak{g}}$-modules $M$ such that
\[
  M = \bigoplus_{(\lambda,\mu)\in P\times P}M_{(\lambda,\mu)},
\]
where $M_{\lambda,\mu}$ denote the weight space of $M$ associated to $(\lambda,\mu)\in P\times P$:
\[
  M_{\lambda,\mu}=\left\{m\in M \ \middle\vert\ K_i\cdot m=q_i^{\langle\lambda,\alpha_i^\vee\rangle}m,\ L_i\cdot m=q_i^{\langle\mu,\alpha_i^\vee\rangle}m,\ \forall 1\leq i\leq n\right\}.
\]

As $\qdbl{\mathfrak{g}}$ is a Hopf algebra, $\mathcal{C}_q$ is a monoidal rigid abelian category. For $(\lambda,\mu)\in P\times P$ and $(\lambda',\mu')\in P\times P$, we write $(\lambda',\mu')\leq(\lambda,\mu)$ if $(\lambda'-\lambda,\mu'-\mu) = \sum_{\alpha\in\Delta}n_\alpha(\alpha ,\alpha)$ with $n_\alpha\in\mathbb{N}$.

\begin{proposition}
  A simple module $M$ in $\mathcal{C}_q$ is a highest weight module.
\end{proposition}

\begin{proof}
  Let $M$ be a simple module in $\mathcal{C}_q$. Then, there exists a weight $(\lambda,\mu)$ of $M$ such that for
  all other weight $(\lambda',\mu')$ of $M$, we have
  $(\lambda,\mu)\not < (\lambda',\mu')$. Such a weight exists
  because $M$ is finite dimensional. Pick a non-zero $v \in
  M_{(\lambda,\mu)}$.
  As, by choice of $(\lambda,\mu)$, $(\lambda+\alpha_i,\mu+\alpha_i)$
  is not a weight of $M$ for any $i$, we have $E_im =0$ for any $i$. Hence $v$ is a highest weight vector.

  Consider now the $\qdbl{\mathfrak{g}}$-module of $M$ generated by $v$. As $M$ is simple and $v\neq 0$, it is the entire module $M$: $M$ is a highest weight module.
\end{proof}

\begin{proposition}
  Any highest weight $(\lambda,\mu)$ of a module in $\mathcal{C}_q$ satisfy $\lambda+\mu\in 2P^+$.
\end{proposition}

\begin{proof}
  Let $M$ be a module in $\mathcal{C}_q$ and $m\in M$ a highest weight vector of weight $(\lambda,\mu)$. Let $1\leq i \leq n$ and consider the family $(F_i^{(k)}m)_{k\in \mathbb{N}}$ of vectors in $M$. The vector $F_i^{(k)}m$ being of weight $(\lambda,\mu)-k(\alpha _i,\alpha_i)$ and $M$ being finite dimensional, there exists a unique $k\in\mathbb{N}$ such that $F_i^{(k)}m\neq 0$ and $F_i^{(k+1)}m=0$. As
\[
  [E_i,F_i^{(k+1)}] = F_i^{(k)}\frac{q_i^{-k}K_i-q_i^{k}L_i^{-1}}{q_i-q_i^{-1}},
\]
we have $0 = [E_i,F_i^{(k+1)}]m = \frac{q_i^{-k+\langle\lambda,\alpha_i^\vee\rangle}-q_i^{k-\langle\mu,\alpha_i^\vee\rangle}}{q_i-q_i^{-1}}F_i^{(k)}m$. Therefore $\langle\lambda+\mu,\alpha_i^\vee\rangle=2k$.
\end{proof}

For any $\lambda\in P$, there exists a one-dimensional $\qdbl{\mathfrak{g}}$-module with unique (hence highest) weight $(\lambda,-\lambda)$, which we denote by $L(\lambda,-\lambda)$. The elements $E_i$ and $F_i$ necessarily acts by $0$ and the action of $K_i$ (resp. $L_i$) is multiplication by $q^{\langle \lambda,\alpha_i\rangle}$ (resp. $q^{-\langle \lambda,\alpha_i\rangle}$).

We denote by $\tilde{P}$ the subset of $P\times P$ such that the sum of the two components is in $2P$. These are precisely the weights which appear as weights of objects of the category $\mathcal{C}_q$:

\begin{proposition}
  For any $(\lambda,\mu)\in P\times P$ such that $\lambda+\mu\in 2P^+$, there exists a simple module in $\mathcal{C}_q$ of highest weight $(\lambda,\mu)$.
\end{proposition}

\begin{proof}
  Using the classification of $\qgr{\mathfrak{g}}$-modules (see \cite[Proposition 10.1.1]{chari-pressley}), there exists a simple $\qdbl{\mathfrak{g}}$-module of highest weight $\left(\frac{\lambda+\mu}{2},\frac{\lambda+\mu}{2}\right)$. Tensoring this module with $L\left(\frac{\lambda-\mu}{2},-\frac{\lambda-\mu}{2}\right)$, we obtain a simple module in $\mathcal{C}_q$ of highest weight $(\lambda,\mu)$.
\end{proof}

Therefore, simple objects in $\mathcal{C}_q$ are classified by $(\lambda,\mu)\in P\times P$ such that $\lambda+\mu\in 2P^+$. We denote $\tilde{P}^+$ this set of weights and by $L(\lambda,\mu)$ the simple module of highest weight $(\lambda,\mu)$. Note that we also have a construction of simple modules ``\emph{à la Verma}''.

\begin{proposition}
  The category $\mathcal{C}_q$ is semisimple.
\end{proposition}

\begin{proof}
  It suffices to show that there are no extensions $0 \rightarrow L \rightarrow M \rightarrow N \rightarrow 0$ of objects in $\mathcal{C}_q$ with $M$ indecomposable and $L$, $N$ non-trivial. If such an extension exists, for $1\leq i \leq n$ the action of $z_i$ on $M$ has a unique eigenvalue. Therefore, for any weight $(\lambda,\mu)$ of $M$, the value $\lambda-\mu$ does not depend on the weight and is in $2P$. We tensorize the exact sequence by the invertible object $X=L\left(\frac{\mu-\lambda}{2},-\frac{\mu-\lambda}{2}\right)$ in order to obtain an exact sequence of $\qgr{\mathfrak{g}}$-modules. As the category of finite dimensional $\qgr{\mathfrak{g}}$-modules is semisimple, $M\otimes X \simeq (L\otimes X) \oplus (N\otimes X)$. Therefore $M\simeq L\oplus N$, contrary to our assumption.
\end{proof}

There exists a faithful $P$-grading on the category $\mathcal{C}_q$
given by the action of the central elements $(z_i)_{1\leq i\leq n}$
\[
  \mathcal{C}_q = \bigoplus_{\nu \in P}\mathcal{C}_{q,\nu},
\]
where $\mathcal{C}_{q,\nu}$ is additively generated by simple objects
$L(\lambda,\mu)$ with $\lambda-\mu=2\nu$, \emph{i.e.} the simple objects on which the central elements $(z_i)_{1\leq i \leq n}$ act by $2\nu$. Each component $\mathcal{C}_{q,\nu}$ is
equivalent to the category of finite dimensional
$\qgr{\mathfrak{g}}$-modules: it is clear for the trivial component
$\mathcal{C}_{q,0}$ and tensoring by $L(\nu,-\nu)$ gives an
equivalence between $\mathcal{C}_{q,0}$ and $\mathcal{C}_{q,\nu}$.

\subsection{Character formula}
\label{sec:character}

We now write a character formula, which we obtain easily using Weyl character formula (see \cite[24.3]{humphreys}). We define an action of the Weyl group $W$ on $\tilde{P}$ as follows
\[
  s_i(\lambda,\mu)=(\lambda,\mu) - \frac{1}{2}\langle\lambda+\mu,\alpha_i^\vee\rangle (\alpha_i ,\alpha_i),
\]
for any $1 \leq i \leq n$. It is easily shown, by induction on the length of $w\in W$ that
\[
  w(\lambda,\mu) = \left(w\left(\frac{\lambda+\mu}{2}\right)+\frac{\lambda-\mu}{2},w\left(\frac{\lambda+\mu}{2}\right)-\frac{\lambda-\mu}{2}\right).
\]

We denote by $e^{(\lambda,\mu)}\in\mathbb{Z}[P\times P]$ the element of the group ring of $P\times P$ over $\mathbb{Z}$ corresponding to $(\lambda,\mu)$. The character of a module $M$ in $\mathcal{C}_q$ is definted to be
\[
  \chi_M=\sum_{(\lambda,\mu)\in P\times P}\dim(M_{(\lambda,\mu)})e^{(\lambda,\mu)}.
\]
Almost every term of the sum is zero and the support of $\chi_M$ is contained in $\tilde{P}$.

The usual Weyl character formula gives the character of all simple modules in $\mathcal{C}_q$ of the form $L(\lambda,\lambda)$ with $\lambda\in P^+$:
\[
  \chi_{(\lambda,\lambda)}=\frac{\sum_{w\in W}(-1)^{l(w)}e^{(w(\lambda+\rho),w(\lambda+\rho))}}{\sum_{w\in W}(-1)^{l(w)}e^{(w(\rho),w(\rho))}}.
\]

As, for $(\lambda,\mu)\in \tilde{P}$, we have an isomorphism in $\mathcal{C}_q$
\[
  L(\lambda,\mu) \simeq 
  L\left(\frac{\lambda+\mu}{2},\frac{\lambda+\mu}{2}\right) \otimes 
  L\left(\frac{\lambda-\mu}{2},-\frac{\lambda-\mu}{2}\right),
\]
the character of $L(\lambda,\mu)$ is
\[
  \chi_{(\lambda,\mu)}=
  \frac{
    \sum_{w\in W}(-1)^{l(w)} e^{\left(w\left(\frac{\lambda+\mu}{2}+\rho\right)
      ,w\left(\frac{\lambda+\mu}{2}+\rho\right)\right)}
      }
      {
        \sum_{w\in W}(-1)^{l(w)}e^{(w(\rho),w(\rho))}
      }
      e^{\left(\frac{\lambda-\mu}{2},-\frac{\lambda-\mu}{2}\right)},
\]
which is therefore equal to
\[
  \chi_{(\lambda,\mu)}=
  \frac{
    \sum_{w\in W}(-1)^{l(w)} e^{(w(\lambda+\rho,\mu+\rho))}
      }
      {
        \sum_{w\in W}(-1)^{l(w)}e^{(w(\rho,\rho))}
      }.
\]

We rewrite this formula by introducing the \emph{dot action} $w\bullet (\lambda,\mu) = w(\lambda+\rho,\mu+\rho) - (\rho,\rho)$ which stabilizes $P\times P$:
\[
  \chi_{(\lambda,\mu)}=
  \frac{
    \sum_{w\in W}(-1)^{l(w)} e^{(w\bullet(\lambda,\mu))}
      }
      {
        \sum_{w\in W}(-1)^{l(w)}e^{(w\bullet (0,0))}
      }.
\]

\subsection{Braiding}
\label{sec:braiding}

Using the quasi-$R$-matrix of section \ref{sec:quasi-r-matrix}, we
endow the tensor category $\mathcal{C}_q$ with a braiding. We moreover have to work with an extension of $\mathbb{Q}(q)$. Let $L$ be the smallest integer such that $L\langle\lambda,\mu\rangle\in\mathbb{Z}$ for any $\lambda,\mu\in P$. Let $s$ an indeterminate such that $s^L=q$ and we now work over the filed $\mathbb{Q}(s)$.

First, as any element of $\mathcal{C}_q$ has a finite number of weights, we see
that for $N$ sufficiently large and any $\mu\in Q$, $\mu\geq 0$ and
$\h(\mu)\geq N$, the element $\Theta_\mu$ of section
\ref{sec:quasi-r-matrix} acts by $0$ on any tensor product of two
elements of $\mathcal{C}_q$. Therefore, for any $M$ and $M'$ in
$\mathcal{C}_q$, the action of $\Theta$ defines a linear map
\[
  \Theta_{M,M'}\colon M\otimes M' \rightarrow M\otimes M'.
\]
Note that this map is not $\qdbl{\mathfrak{g}}$-linear but satisfies for
any $u \in \qdbl{\mathfrak{g}}$
\[
  \Theta_{M,M'}\circ \Delta(u) = (\Psi\circ\Delta^{\mathrm{op}})(u) \circ \Theta_{M,M'},
\]
as linear endomorphisms of $M\otimes M'$. This follows immediately from
Proposition \ref{prop:coprod_r_mat}.

To construct the braiding, we need one more ingredient. For $M$ and
$M'$ two objects of $\mathcal{C}_q$, we introduce a linear map
$f_{M,M'}\colon M\otimes M' \rightarrow M\otimes M'$ defined on weight vectors
$m\in M_{\lambda,\mu}$ and $m'\in M'_{\lambda',\mu'}$ by
\[
  f_{M,M'}(m\otimes m') = q^{\langle\lambda,\mu'\rangle}m\otimes m',
\]
where we write $q^r = s^{Lr}$ for any $r \in \frac{1}{L}\mathbb{Z}$.

\begin{lemma}
  For any $u\in\qdbl{\mathfrak{g}}\otimes\qdbl{\mathfrak{g}}$, we have
  the following equality as linear endomorphisms of $M\otimes M'$
  \[
    u\circ f_{M,M'}=f_{M,M'}\circ \Psi(u).
  \]
\end{lemma}

\begin{proof}
  It suffices to show it on the generators of
  $\qdbl{\mathfrak{g}}\otimes\qdbl{\mathfrak{g}}$. It is trivial for
  $K_i\otimes 1$, $L_i\otimes 1$, $1\otimes K_i$ and $1\otimes
  L_i$. We now verify it for $E_i\otimes 1$. On the one hand,
  \[
    \left(f_{M,M'}\circ (E_i\otimes L_i^{-1})\right)_{|_{M_{\lambda,\mu}\otimes M'_{\lambda',\mu'}}} 
      = q^{\langle\lambda+\alpha_i,\mu'\rangle-\langle\alpha_i,\mu'\rangle}(E_i\otimes 1)_{|_{M_{\lambda,\mu}\otimes M'_{\lambda',\mu'}}}.
  \]
  On the other hand,
  \[
    \left((E_i\otimes 1)\circ f_{M,M'}\right)_{|_{M_{\lambda,\mu}\otimes M'_{\lambda',\mu'}}} 
    = q^{\langle\lambda,\mu'\rangle}(E_i\otimes 1)_{|_{M_{\lambda,\mu}\otimes M'_{\lambda',\mu'}}},
  \]
  which concludes the proof for $E_i\otimes 1$. The other cases are similar.
\end{proof}

Denote by $\tau_{V,W}\colon V\otimes W \rightarrow W\otimes V$ the twist of vector spaces sending $v\otimes w$ to $w\otimes v$ and define
\[
  c_{M,M'} = \tau\circ f_{M,M'} \circ \Theta_{M,M'}.
\]
It is then a morphism in $\mathcal{C}_q$ between $M\otimes M'$ and $M'\otimes M$. Indeed, for any $u\in\qdbl{\mathfrak{g}}$ we have
\begin{align*}
  c_{M,M'}\circ u_{|_{M\otimes M'}} 
  &= \tau\circ f_{M,M'} \circ \Theta_{M,M'}\circ \Delta(u)_{|_{M\otimes M'}}\\
  &= \tau\circ f_{M,M'} \circ (\Psi\circ\Delta^{\mathrm{op}})(u)_{|_{M\otimes M'}}\circ \Theta_{M,M'} \\
  &= \tau\circ \Delta^{\mathrm{op}}(u)_{|_{M\otimes M'}} \circ f_{M,M'} \circ \Theta_{M,M'} \\
  &= \Delta(u)_{|_{M'\otimes M}}\circ c_{M,M'}\\
  &= u_{|_{M'\otimes M}}\circ c_{M,M'}.
\end{align*}

\begin{proposition}
  The morphisms $c$ defined above endow $\mathcal{C}_q$ with a braiding: the hexagon axioms \cite[Definition 8.1.1]{egno} are satisfied.
\end{proposition}

\begin{proof}
  We check that $c_{M,M'\otimes M''} = (\id_{M'}\otimes c_{M,M''})\circ(c_{M,M'}\otimes \id_{M''})$. It follows from the following calculations, where we denote by $\tau_{1,23}\colon M\otimes(M'\otimes M'') \rightarrow (M'\otimes M'')\otimes M$ the twist, which is equal to $\tau_{23}\circ \tau_{12}$:
  \begin{align*}
    c_{M,M'\otimes M''} 
    &= \tau_{1,23}\circ f_{M,M'\otimes M''}\circ \Theta_{M,M'\otimes M''}\\
    &= \tau_{23}\circ \tau_{12} \circ (f_{M,M''})_{13} \circ (f_{M,M'})_{12} \circ (\id \otimes \Delta)(\Theta)_{|_{M\otimes M'\otimes M''}} \\
    &= \tau_{23}\circ \tau_{12} \circ (f_{M,M''})_{13} \circ (f_{M,M'})_{12} \circ \Psi_{12}\left(\Theta_{13}\right)_{|_{M\otimes M'\otimes M''}}\circ \left(\Theta_{12}\right)_{|_{M\otimes M'\otimes M''}} \\
    &= \tau_{23}\circ \tau_{12} \circ (f_{M,M''})_{13} \circ \left(\Theta_{13}\right)_{|_{M\otimes M'\otimes M''}} \circ (f_{M,M'})_{12}  (\Theta_{M,M'}\otimes \id_{M''})\\
    &= \tau_{23}\circ (f_{M,M''})_{23} \circ (\id_{M'}\otimes \Theta_{M,M''}) \circ \tau_{12} \circ (f_{M,M'})_{12}  (\Theta_{M,M'}\otimes \id_{M''})\\
    &= (\id_{M'}\otimes c_{M,M''})\circ(c_{M,M'}\otimes \id_{M''}).
  \end{align*}
  
  The other hexagon axiom is shown similarly.
\end{proof}

\subsection{Duality and pivotal structure}
\label{sec:duality}

The algebra $\qdbl{\mathfrak{g}}$ being a Hopf algebra, for any $\qdbl{\mathfrak{g}}$-module $M$ in $\mathcal{C}_q$, the space of linear forms has naturally a structure of left (resp. right) dual denoted by $M^*$ (resp. ${}^*M$), where the action of $u \in \qdbl{\mathfrak{g}}$ on $\varphi\in M^*$ (resp. $\varphi\in{}^*M$) is given by
\[
  (u\cdot\varphi) (m) = \varphi(S(u)\cdot m) \quad \text{(resp. } (u\cdot\varphi) (m) = \varphi(S^{-1}(u)\cdot m) \text{)}.
\]

Let $(\lambda,\mu) \in \tilde{P}^+$. The simple module $L(\lambda,\mu)$ is isomorphic to $L\left(\frac{\lambda+\mu}{2},\frac{\lambda+\mu}{2}\right)\otimes L\left(\frac{\lambda-\mu}{2},-\frac{\lambda-\mu}{2}\right)$ and
\[
  L(\lambda,\mu)^* \simeq L\left(\frac{\lambda-\mu}{2},-\frac{\lambda-\mu}{2}\right)^* \otimes L\left(\frac{\lambda+\mu}{2},\frac{\lambda+\mu}{2}\right)^*. 
\]
But the left dual of a simple $\qgr{\mathfrak{g}}$-module $L(\kappa,\kappa)$ is isomorphic to $L(-w_0(\kappa),-w_0(\kappa))$ \cite[Proposition 5.16]{jantzen} and the left dual of the invertible module $L(\kappa,-\kappa)$ is clearly isomorphic to $L(-\kappa,\kappa)$. Consequently,
\[
  L(\lambda,\mu)^* \simeq L\left(\frac{\mu-\lambda}{2},-\frac{\mu-\lambda}{2}\right)\otimes
  L\left(-w_0\left(\frac{\lambda+\mu}{2}\right),-w_0\left(\frac{\lambda+\mu}{2}\right)\right) \simeq L(-w_0(\lambda,\mu)).
\]
Similarly ${}^*L(\lambda,\mu)\simeq L(-w_0(\lambda,\mu))$, which also follows from the pivotal structure given below.

As for any $u\in\qdbl{\mathfrak{g}}$ we check that $S^2(u) = K_{2\rho}uK_{2\rho}^{-1}$, we have for any $\lambda \in Q$ a pivotal structure \cite[Definition 4.7.7]{egno} given by
\[
  \fonctionNom{a_{\lambda,M}}{M}{M^{**}}{m}{\varphi\mapsto \varphi((L_\lambda K_{2\rho-\lambda})\cdot m)}.
\]

We will consider only the pivotal structure given by $\lambda = 2\rho$, \emph{i.e.} we choose $L_{2\rho}$ as pivotal element. This pivotal structure allows us to compute left and right quantum dimensions of modules \cite[Definition 4.7.1]{egno}. It follows from the Weyl character formula that
\[
  \dim^+(L(\lambda,\mu))=\frac
  {
    \sum_{w\in W}(-1)^{l(w)} q^{\langle 2\rho,(w\bullet(\lambda,\mu))_2\rangle}
  }
  {
    \sum_{w\in W}(-1)^{l(w)}q^{\langle 2\rho,w\bullet(0)\rangle}
  }
\]
and
\[
  \dim^-(L(\lambda,\mu))=\frac
  {
    \sum_{w\in W}(-1)^{l(w)} q^{-\langle 2\rho,(w\bullet(\lambda,\mu))_2\rangle}
  }
  {
    \sum_{w\in W}(-1)^{l(w)}q^{-\langle 2\rho,w\bullet(0)\rangle}
  },
\]
where we denote by $(\lambda',\mu')_i$ the $i$th component of $(\lambda',\mu')$ for $i\in\{1,2\}$.

As the actions of $L_{2\rho}$ and $K_{2\rho}$ coincide on any $\qgr{\mathfrak{g}}$-module, the quantum dimensions of $L(\kappa,\kappa)$ viewed as a $\qdbl{\mathfrak{g}}$-module coincide with the quantum dimensions of $L(\kappa,\kappa)$ viewed as a $\qgr{\mathfrak{g}}$-module, for any $\kappa\in P^+$.


\section{Specialization at a root of unity and tilting modules}
\label{sec:spe_tilt}
The goal to this section is to construct a fusion category using the representations of $\qdbl{\mathfrak{g}}$. But the category $\mathcal{C}_q$ has an infinite number of simple objects. We adapt the construction of the fusion categories attached to quantum groups at roots of unity to $\qdbl{\mathfrak{g}}$. Recall the notation $\mathcal{A} = \mathbb{Z}[q,q^{-1}]$.

\subsection{Lusztig's integral form}
\label{sec:integral}

We introduce an $\mathcal{A}$-subalgebra of $\qdbl{\mathfrak{g}}$ which we will use to specialize $\qdbl{\mathfrak{g}}$ at a root of unity. Consider the following elements of $\qdbl{\mathfrak{g}}$:
\[
  \br{K_i}{c}{t} = \prod_{r=1}^t\frac{q_i^{c-r+1}K_i-q_i^{-c+r-1}K_i^{-1}}{q_i^r-q_i^{-r}},
  \quad \br{L_i}{c}{t} = \prod_{r=1}^t\frac{q_i^{c-r+1}L_i-q_i^{-c+r-1}L_i^{-1}}{q_i^r-q_i^{-r}},
\]
\[
  \dbr{i}{c}{t} = \prod_{r=1}^t\frac{q_i^{c-r+1}K_i-q_i^{-c+r-1}L_i^{-1}}{q_i^r-q_i^{-r}},
  \quad \br{z_i}{c}{t} = \prod_{r=1}^t\frac{q_i^{c-r+1}z_i-q_i^{-c+r-1}z_i^{-1}}{q_i^r-q_i^{-r}}
\]
for $1 \leq i \leq n$, $c\in\mathbb{Z}$ and $t\in\mathbb{N}$.

We define $\qdblres{\mathfrak{g}}$ as the $\mathcal{A}$-subalgebra of $\qdbl{\mathfrak{g}}$ generated by $E_i^{(r)}$, $F_i^{(r)}$, $K_i$, $L_i$, $\br{K_i}{c}{t}$, $\br{L_i}{c}{t}$, $\quad \br{z_i}{c}{t}$ and $\dbr{i}{c}{t}$, for $1 \leq i \leq n$, $c\in\mathbb{Z}$ and $t\in\mathbb{N}$. The coproduct, the counit and the antipode of $\qdbl{\mathfrak{g}}$ restrict to $\qdblres{\mathfrak{g}}$ and endow $\qdblres{\mathfrak{g}}$ with a structure of a Hopf algebra.

The quotient of $\qdblres{\mathfrak{g}}$ by the Hopf ideal generated by $z_i-1$, $\br{z_i}{c}{t} - \qbinom{t}{c}{q_i}$, $\dbr{i}{c}{t} - \br{L_i}{c}{t}$ and $\dbr{i}{c}{t} - \br{K_i}{c}{t}$, for $1 \leq i \leq n$ is the usual Lusztig's integral form of $\qgr{\mathfrak{g}}$ (see \cite[Section 9.3]{chari-pressley}).

The elements of the form $\dbr{i}{c}{t}$ appear naturally in the following identity, proved by induction
\[
  E_i^{(p)}F_i^{(r)}=\sum_{t=0}^{\min(p,r)}F_i^{(r-t)}\dbr{i}{2t-r-p}{t}E_i^{(p-t)},
\] 
for any $1 \leq i \leq n$, $p,r\in\mathbb{N}$.

The following formulas for the coproduct will be useful.

\begin{proposition}
  \label{prop:fml_res}
  Let $1\leq i \leq n$ and $t\in\mathbb{N}$. Then
  \begin{enumerate}
  \item $\Delta\left(\br{K_i}{0}{t}\right)=
    \displaystyle\sum_{r=0}^t\br{K_i}{0}{t-r}K_i^{-r}\otimes\br{K_i}{0}{r}K_i^{t-r}$,
  \item $\Delta\left(\br{L_i}{0}{t}\right)=
    \displaystyle\sum_{r=0}^t\br{L_i}{0}{t-r}L_i^{-r}\otimes\br{L_i}{0}{r}L_i^{t-r}$,
  \item $\Delta\left(\dbr{i}{0}{t}\right)=
    \displaystyle\sum_{r=0}^t\dbr{i}{0}{t-r}L_i^{-r}\otimes\dbr{i}{0}{r}K_i^{t-r}$,
  \item $\Delta\left(\br{z_i}{0}{t}\right)=
    \displaystyle\sum_{r=0}^t\br{z_i}{0}{t-r}z_i^{-r}\otimes\br{z_i}{0}{r}z_i^{t-r}$.
  \end{enumerate}
\end{proposition}

\begin{proof}
  We only prove the first formula, the proof for the other formulas are similar. We proceed by induction on $t$. For $t=0$, there is nothing to prove. For $t=1$, we have:
\[
  \frac{K_i-K_i^{-1}}{q_i-q_i^{-1}}\otimes K_i + K_i^{-1} \otimes \frac{K_i-K_i^{-1}}{q_i-q_i^{-1}} = \frac{K_i\otimes K_i - K_i^{-1}\otimes K_i^{-1}}{q_i-q_i^{-1}}.
\]

Now, suppose that the result is true for $t\in \mathbb{N}$. We have:
\begin{align*}
  \Delta\left(\br{K_i}{0}{t+1}\right) 
  &= \Delta\left(\br{K_i}{0}{t}\right) \frac{q_i^{-t}K_i\otimes K_i - q_i^t K_i^{-1}\otimes K_i^{-1}}{q_i^{t+1}-q_i^{-t-1}}\\
  &= \sum_{r=0}^t\left(\br{K_i}{0}{t-r}K_i^{-r}\otimes\br{K_i}{0}{r}K_i^{t-r} \frac{q_i^{-r}(q_i^{-t+r}K_i - q_i^{t-r}K_i^{-1})\otimes K_i}{q_i^{t+1}-q_i^{-t-1}}\right)\\
  & + \sum_{r=0}^t\left(\br{K_i}{0}{t-r}K_i^{-r}\otimes\br{K_i}{0}{r}K_i^{t-r} \frac{q_i^{t-r}K_i^{-1}\otimes (q_i^{-r}K_i - q_i^r K_i^{-1})}{q_i^{t+1}-q_i^{-t-1}}\right)\\
  &= \sum_{r=0}^t\left(\br{K_i}{0}{t-r+1}K_i^{-r}\otimes\br{K_i}{0}{r}K_i^{t-r+1}\frac{q_i^{-r}(q_i^{t-r+1}-q_i^{-t+r-1})}{q_i^{t+1}-q_i^{-t-1}}\right)\\
  &+ \sum_{r=0}^t\left(\br{K_i}{0}{t-r}K_i^{-r-1}\otimes\br{K_i}{0}{r+1}K_i^{t-r}\frac{q_i^{t-r}(q_i^{r+1}-q_i^{-r-1})}{q_i^{t+1}-q_i^{-t-1}}\right).
\end{align*}

We just change the indes in the second sum and obtain:
\begin{align*}
  \Delta\left(\br{K_i}{0}{t+1}\right) 
  &= \sum_{r=0}^t\left(\br{K_i}{0}{t-r+1}K_i^{-r}\otimes\br{K_i}{0}{r}K_i^{t-r+1}\frac{q_i^{-r}(q_i^{t-r+1}-q_i^{-t+r-1})}{q_i^{t+1}-q_i^{-t-1}}\right)\\
  &+ \sum_{r=1}^{t+1}\left(\br{K_i}{0}{t-r+1}K_i^{-r}\otimes\br{K_i}{0}{r}K_i^{t-r+1}\frac{q_i^{t+1-r}(q_i^{r}-q_i^{-r})}{q_i^{t+1}-q_i^{-t-1}}\right).
\end{align*}
Noticing that
\[
  \frac{q_i^{-r}(q_i^{t-r+1}-q_i^{-t+r-1})}{q_i^{t+1}-q_i^{-t-1}} + \frac{q_i^{t+1-r}(q_i^{r}-q_i^{-r})}{q_i^{t+1}-q_i^{-t-1}} = \frac{q_i^{t-2r+1}-q_i^{-t-1}+q_i^{t+1}-q_i^{t-2r+1}}{q_i^{t+1}-q_i^{-t-1}}=1,
\]
leads to the expected conclusion.
\end{proof}

We denote by $\qdblres{\mathfrak{g}}^{>0}$ (resp. $\qdblres{\mathfrak{g}}^{0}$, resp. $\qdblres{\mathfrak{g}}^{<0}$) the intersection of $\qdblres{\mathfrak{g}}$ with $\qdbl{\mathfrak{g}}^{>0}$ (resp. $\qdbl{\mathfrak{g}}^{0}$, resp. $\qdbl{\mathfrak{g}}^{<0}$).

It is not hard to show that for any $1 \leq i \leq n$, the automorphism $T_i$ restricts to $\qdblres{\mathfrak{g}}$. As in Section \ref{sec:basis}, we have:

\begin{proposition}[{\cite[Theorem 6.7]{lusztig_qgrp_roots}}]
  The elements
\[
  \prod_{\alpha\in\Phi^+}E_\alpha^{(n_\alpha)}, \ n_\alpha\in\mathbb{N}
\]
form an $\mathcal{A}$-basis of $\qdblres{\mathfrak{g}}^{>0}$.

The elements
\[
  \prod_{\alpha\in\Phi^+}F_\alpha^{(n_\alpha)}, \ n_\alpha\in\mathbb{N}
\]
form an $\mathcal{A}$-basis of $\qdblres{\mathfrak{g}}^{<0}$.
\end{proposition}

Note that the quasi-$R$-matrix at the end of Section \ref{sec:quasi-r-matrix} lies in fact in $\qdblres{\mathfrak{g}}$ and the formula for its inverse shows that it also lies in $\qdblres{\mathfrak{g}}$. We still denote by $\Theta$ this element. Moreover, the algebra endomorphism $\Psi$ of $\qdbl{\mathfrak{g}}\otimes\qdbl{\mathfrak{g}}$ restricts to an algebra endomorphism of $\qdblres{\mathfrak{g}}\otimes\qdblres{\mathfrak{g}}$, which we still denote by $\Psi$. Therefore, we still have the relations
\[
  \Theta\Delta(u) = \left(\Psi\circ\Delta^{\text{op}}\right)(u)\Theta,\quad (\id\otimes\Delta)(\Theta) = \Psi_{12}(\Theta_{13})\Theta_{12}, \quad \text{and} \quad (\Delta\otimes\id)(\Theta) = \Psi_{23}(\Theta_{13})\Theta_{23}.
\] 

\subsection{Specialization at a root of unity}
\label{sec:root_of_1}

Let $\xi\in\mathbb{C}$ be a root of unity of order $l$. Set $l'=l$ if $l$ is odd and $l'=\frac{l}{2}$ if $l$ is even so that $l'$ is the order of $\xi^2$.
\[
  \fonction{\mathcal{A}}{\mathbb{C},}{q}{\xi.}
\] 

The specialization of $\qdbl{\mathfrak{g}}$ at $\xi$ is by definition the algebra $\qdblroot{\mathfrak{g}} = \qdblres{\mathfrak{g}}\otimes_{\mathcal{A}}\mathbb{C}$. Introduce the integers $l_i$ as the order of $\xi_i = \xi^{\frac{\langle\alpha_i,\alpha_i\rangle}{2}}$ and set $l_i'=l_i$ if $l_i$ is odd and $l_i'=\frac{l_i}{2}$ if $l_i$ is even so that $l_i'$ is the order of $\xi_i^2$.

The algebra $\qdblroot{\mathfrak{g}}$ is generated by elements of the form
\[
  E_i^{(r)},F_i^{(r)},K_i^{\pm 1},L_i^{\pm 1},\br{K_i}{c}{t},\br{L_i}{c}{t},\dbr{i}{c}{t}\text{ and }\br{z_i}{c}{t},
\]
for $1\leq i \leq n$, $r,t\in\mathbb{N}$ and $c\in\mathbb{Z}$ given by the images in $\qdblroot{\mathfrak{g}}$ of the corresponding elements of $\qdblres{\mathfrak{g}}$. There exist some relations in $\qdblroot{\mathfrak{g}}$ related to the order of the $\xi_i$'s as
\[
  E_i^{l_i'} = 0, F_i^{l_i'}=0, K_i^{2l_i'}=1, L_i^{2l_i'}=1, K_i^{l_i'}=L_i^{l_i'}.
\]

We denote by $\Theta_\xi$ the specialization of $\Theta$. It is still an invertible element of some completion of $\qdblroot{\mathfrak{g}}\otimes\qdblroot{\mathfrak{g}}$, and we have an algebra endomorphism $\Psi_\xi$ which is the specialization of $\Psi$. They satisfy the usual relations 
\begin{multline*}
  \Theta_\xi\Delta(u) =
  \left(\Psi_\xi\circ\Delta^{\text{op}}\right)(u)\Theta_\xi ,\quad
  (\id\otimes\Delta)(\Theta_\xi) =
  (\Psi_\xi)_{12}((\Theta_\xi)_{13})(\Theta_\xi)_{12}, \\
  \text{and} \quad (\Delta\otimes\id)(\Theta_\xi) =
  (\Psi_\xi)_{23}((\Theta_\xi)_{13})(\Theta_\xi)_{23}.
\end{multline*}

\subsection{Representations at a root of unity}
\label{sec:rep_q_root}

The representation theory of $\qgrroot{\mathfrak{g}}$, as well the one of $\qdblroot{\mathfrak{g}}$, is more involved. As explained in \cite[11.2.A]{chari-pressley}, defining the weight space of weight $(\lambda,\mu)$ of a $\qdblres{\mathfrak{g}}$-module $M$ as
\[
  \left\{m\in M \ \middle\vert\  K_im=\xi^{\langle\lambda,\alpha_i\rangle}m, L_im=\xi^{\langle\mu,\alpha_i\rangle}m,\ \forall 1\leq i \leq n\right\}
\]
is rather unsatisfactory: one can not distinguish the weight $(\lambda,\mu)$ form the weight $(\lambda+l_i\varpi_i,\mu)$ for example.

The weight space of weight $(\lambda,\mu)$ of $M$ is then defined as
\begin{multline*}
  M_{\lambda,\mu}=\left\{m\in M\ \middle\vert\ 
    K_im=\xi_i^{\langle\lambda,\alpha_i^\vee\rangle}m,
    L_im=\xi_i^{\langle\mu,\alpha_i^\vee\rangle}m,\br{K_i}{0}{l_i'}m=\qbinom{l_i'}{\langle\lambda,\alpha_i^\vee\rangle}{\xi_i},\right. \\
    \left. \br{L_i}{0}{l_i'}m=\qbinom{l_i'}{\langle\mu,\alpha_i^\vee\rangle}{\xi_i},\
    \forall 1\leq i \leq n\right\}.
\end{multline*}

\begin{proposition}
  \label{prop:weight_spaces}
  Let $(\lambda,\mu)$ and $(\lambda',\mu')$ be two weights of a $\qdblroot{\mathfrak{g}}$-module $M$. Then $M_{\lambda,\mu}=M_{\lambda',\mu'}$ if and only if $(\lambda,\mu)=(\lambda',\mu')$.
\end{proposition}

\begin{proof}
  We show that for $\zeta$ a root of unity of order $d$ and with $d'=d$ if d is odd, $d'=\frac{d}{2}$ if $d$ is even, and $r\in\mathbb{Z}$, we can recover $r$ knowing only $\qbinom{d'}{r}{\zeta}$ and $\zeta^r$.

\begin{lemma}
  \label{lem:binomial_identity}
  Let $a,b\in\mathbb{N}$ with $a\geq b$. Let $a=a_1d'+a_0$ and $b=b_1d'+b_0$ with $0\leq a_0,b_0<d'$. Then
\[
  \qbinom{b}{a}{\zeta} = (-1)^{(d'+1)(a_1+1)b_1}\left(\zeta^{d'}\right)^{(a_1+1)b_1+a_1b_0+a_0b_1}\binom{a_1}{b_1}\qbinom{b_0}{a_0}{\zeta}.
\]
\end{lemma}

\begin{proof}
  We start with the following equality, valid for any $m\in\mathbb{N}$ and $X$ an indeterminate
\[
  \prod_{k=0}^{m-1}\left(1+\zeta^{2k}X\right)=\sum_{k=0}^m\zeta^{k(m-1)}\qbinom{k}{m}{\zeta}X^k.
\]

As $\{\zeta^{-2k}\mid 0\leq k < d'\}$ is the set of $d'$th-roots of $1$ and
\[
  \prod_{k=0}^{d'-1}\left(1+\zeta^{2k}X\right)=\zeta^{d'(d'-1)}\prod_{k=0}^{d'-1}\left(X+\zeta^{-2k}\right)
\]
we have
\[
  \prod_{k=0}^{d'-1}\left(1+\zeta^{2k}X\right)=\zeta^{d'(d'-1)}(X^{d'}-(-1)^{d'}) = \zeta^{d'(d'+1)}(X^{d'}+(-1)^{d'+1}).
\]

On the one hand we have
\[
  \prod_{k=0}^{a-1}\left(1+\zeta^{2k}X\right)=\sum_{k=0}^a\zeta^{k(r-1)}\qbinom{k}{a}{\zeta}X^k,
\]
and on the other hand
\begin{align*}
  \prod_{k=0}^{a-1}\left(1+\zeta^{2k}X\right)
  &= \left(\zeta^{d'(d'+1)}\left(X^{d'}+(-1)^{d'+1}\right)\right)^{a_1}\left(\prod_{k=0}^{a_0-1}\left(1+\zeta^{2k}X\right)\right)\\
  &= \zeta^{d'(d'+1)a_1}\left(\sum_{k=0}^{a_1}\binom{a_1}{k}(-1)^{(d'+1)(a_1-k)}X^{kd'}\right)\left(\sum_{k=0}^{a_0}\zeta^{k(a_0-1)}\qbinom{k}{a_0}{\zeta}X^k\right).
\end{align*}
Comparing the coefficient of $X^b$ gives us
\[
  \zeta^{b(a-1)}\qbinom{b}{a}{\zeta} = \zeta^{d'(d'+1)a_1+b_0(a_0-1)}(-1)^{(d'+1)(a_1-b_1)}\binom{a_1}{b_1}\qbinom{b_0}{a_0}{\zeta},
\]
which, using the fact that $\zeta^{d'(d'+1)}=(-1)^{d'+1}$, leads to the desired formula.
\end{proof}

Therefore, for $r\in\mathbb{N}$, we have
\[
  \qbinom{r}{d'}{\zeta}=
  \begin{cases}
    r_1 & \text{ if } d'=d,\\
    (-1)^{r_0+r_1+1}r_1 & \text{ if } d'=\frac{d}{2} \text{ and } d' \text{ is odd},\\
    (-1)^{r_0}r_1 & \text{ if } d'=\frac{d}{2} \text{ and } d' \text{ is even}.
  \end{cases}
\]
Using $\qbinom{d'}{-r}{\zeta} = (-1)^{d'}\qbinom{d'}{r+d'-1}{\zeta}$, we check that this is still valid for $r \in \mathbb{Z}$.

We write $r=r_1d'+r_0$ with $0\leq r_0<d'$. Let $0\leq r' < d$ be the unique integer such that $\zeta^r = \zeta^{r'}$. 

First, we suppose $d'=d$, hence $d'$ is odd. Lemma \ref{lem:binomial_identity} gives $\qbinom{d'}{r}{\zeta} = r_1$ and we have $r = \qbinom{d'}{r}{\zeta}d'+r'$.

Now, we suppose that $d$ is even. Let $r'' = r'$ if $0\leq r'<d'$ and $r''=r'-d'$ otherwise. 

We suppose that $d'=\frac{d}{2}$ is odd. Lemma \ref{lem:binomial_identity} gives $\qbinom{d'}{r}{\zeta} = (-1)^{r_0+r_1+1}r_1$. If $r''=r'$ then $r=(-1)^{r''+1}\qbinom{d'}{r}{\zeta}+r''$, and if $r''\neq r'$ then $r=(-1)^{r''}\qbinom{d'}{r}{\zeta}+r''$.

Finally, if we suppose that $d'=\frac{d}{2}$ is even, Lemma \ref{lem:binomial_identity} gives $\qbinom{d'}{r}{\zeta} = (-1)^{r_0}r_1$ and $r=(-1)^{r''}\qbinom{d'}{r}{\zeta}+r''$.

Therefore, if for all $1\leq i \leq n$, $\xi_i^{\langle\lambda,\alpha_i^{\vee}\rangle}= \xi_i^{\langle\lambda',\alpha_i^{\vee}\rangle}$ and $\qbinom{l_i'}{\langle\lambda,\alpha_i^\vee\rangle}{\xi_i}=\qbinom{l_i'}{\langle\lambda',\alpha_i^\vee\rangle}{\xi_i}$, we have $\langle\lambda,\alpha_i^{\vee}\rangle=\langle\lambda',\alpha_i^\vee\rangle$ for all $1\leq i \leq n$ and hence $\lambda = \lambda'$. Similarly $\mu=\mu'$.
\end{proof}

The formula for the coproduct given in Proposition \ref{prop:fml_res} shows that the tensor product of two weight vectors of weight $(\lambda,\mu)$ and $(\lambda',\mu')$ is again a weight vector, of weight $(\lambda+\lambda',\mu+\mu')$.

We consider the category $\mathcal{C}_\xi$ of finite dimensional $\qdblroot{\mathfrak{g}}$-modules $M$ such that
\[
  M = \bigoplus_{(\lambda,\mu)\in P\times P}M_{\lambda,\mu}.
\]
As in the generic case we have $E_i^{(r)}\cdot M_{\lambda,\mu}\subseteq M_{\lambda+r\alpha_i,\mu+r\alpha_i}$ and $F_i^{(r)}\cdot M_{\lambda,\mu}\subseteq M_{\lambda-r\alpha_i,\mu-r\alpha_i}$.

Contrary to the generic case, it may happen that $\lambda+\mu\not\in 2P$, for example if one of the $l_i$'s is odd.

One way to construct $\qdblroot{\mathfrak{g}}$-modules is specialization. Let $(\lambda,\mu)\in \tilde{P}^+$. We have defined a simple highest weight $\qdbl{\mathfrak{g}}$-module $L(\lambda,\mu)$ and we consider a sub-$\qdblres{\mathfrak{g}}$-module $L^{\text{res}}(\lambda,\mu)$ generated by a chosen highest weight vector. Define the \emph{Weyl module} by
\[
  W_\xi(\lambda,\mu) = L^{\text{res}}(\lambda,\mu)\otimes_{\mathcal{A}}\mathbb{C}.
\]
This is a highest weight $\qdblroot{\mathfrak{g}}$-module with highest weight $(\lambda,\mu)$, but it is not always a simple module. As specialization does not change the weight spaces, it is clear that the character of $W_\xi(\lambda,\mu)$ is still given by the Weyl character formula given in section \ref{sec:character}. The Weyl module is a quotient of the Verma module and it has simple head which we denote $L_{\xi}(\lambda,\mu)$.

Similarly to the case of a generic parameter $q$, the choice of an $L$-th root $\xi^{1/L}$ of $\xi$ turns the category of representations into a braided category, as in Section \ref{sec:braiding}

\subsection{Tilting modules}
\label{sec:tilting}

As for $\qgrroot{\mathfrak{g}}$, we consider the tilting modules, which have been studied for quantum groups at roots of unity by Andersen in \cite{andersen_tilting}.

\begin{definition}
  A $\qdblroot{\mathfrak{g}}$-module $M$ is \emph{tilting} if
  both $M$ and $M^*$ have a filtration with successive quotients being
  Weyl modules.
\end{definition}

We refer to \cite{sawin} for more details in the case of $\qgrroot{\mathfrak{g}}$. We remark that weights $(\lambda,\mu)$ of Weyl modules always satisfy $\lambda+\mu \in 2P$; hence only these weights can appear in a tilting module.

\begin{proposition}
  For any $(\lambda,\mu)\in \tilde{P}^+$, there exists an indecomposable tilting module $T(\lambda,\mu)$ such that $T(\lambda,\mu)_{\lambda',\mu'} = 0$ unless $(\lambda',\mu') \leq (\lambda,\mu)$ and $T(\lambda,\mu)_{\lambda,\mu}$ is of dimension one.

  Moreover, any indecomposable tilting module is isomorphic to some $T(\lambda,\mu)$.
\end{proposition}

\begin{proof}
  The existence of such tilting modules follows easily from the existence for $\qgrroot{\mathfrak{g}}$. Indeed, there exists an indecomposable tilting module for $\qgrroot{\mathfrak{g}}$ with a maximal vector of weight $\frac{\lambda+\mu}{2}$. Tensoring it with $L_\xi\left(\frac{\lambda-\mu}{2},-\frac{\lambda-\mu}{2}\right)$ gives us an indecomposable tilting module with a maximal vector of weight $(\lambda,\mu)$.

  Let $T$ be an indecomposable tilting $\qgrroot{\mathfrak{g}}$-module. Using the action of the central elements $z_i$, $\br{z_i}{0}{l_i'}$ for $1\leq i \leq n$, one can show that $\lambda-\mu$ does not depend on the weight $(\lambda,\mu)$ of $T$. Therefore, tensoring $T$ by $L_\xi\left(\frac{\mu-\lambda}{2},-\frac{\mu-\lambda}{2}\right)$ gives an indecomposable tilting module for $\qgrroot{\mathfrak{g}}$, which is isomorphic to $T\left(\kappa,\kappa\right)$ for some $\kappa\in P^+$.
\end{proof}

As a direct summand of a tilting module is again a tilting module, every tilting module is a direct sum of indecomposable tilting modules. Using the fact that the tensor product of two tilting modules for $\qgrroot{\mathfrak{g}}$ is again a tilting module, we easily show that the tensor product of two tilting modules for $\qdblroot{\mathfrak{g}}$ is again a tilting module.

The full subcategory of $\mathcal{C}_\xi$ with objects the tilting modules is far from being a fusion category: it is not abelian, nor semisimple. We want to semisimplify this category; hence we must understand which indecomposable tilting modules are of non-zero quantum dimension. As any indecomposable tilting module is isomorphic to some $T(\lambda,\mu)\simeq T\left(\frac{\lambda+\mu}{2},\frac{\lambda+\mu}{2}\right)\otimes L_\xi\left(\frac{\lambda-\mu}{2},-\frac{\lambda-\mu}{2}\right)$, we can use the results for $\qgrroot{\mathfrak{g}}$. The next theorem follows immediately from \cite[Theorem 2]{sawin}. Let $D=\max_{\alpha,\beta\in\Phi}\frac{\langle\alpha,\alpha\rangle}{\langle\beta,\beta\rangle}$. We have $D=1$ in type $A$, $D$ or $E$, $D=2$ in type $B$,$C$ or $F_4$ and $D=3$ in type $G_2$. Let $\theta_0$ be the highest root of $\Phi$ if $D\mid l'$ and be the highest short root of $\Phi$ otherwise. Let $C$ be the following set of dominant weights
\[
  C=\{ \lambda\in P^+ \mid \langle\lambda+\rho,\theta_0\rangle <l'\}.
\]
Finally, denote by $h$ the Coxeter number of $\mathfrak{g}$ and by $h^\vee$ the duax Coxeter number of $\mathfrak{g}$.

\begin{theoreme} 
  \label{thm:tilt_non_negl}
  We suppose that $l'\geq Dh^\vee$ if $D\mid l'$ and $l' > h$ otherwise. Then $T(\lambda,\mu)$ is of non-zero positive and negative quantum dimension if and only if $\frac{\lambda+\mu}{2} \in C$. Moreover, $T(\lambda,\mu)\simeq W_\xi(\lambda,\mu)\simeq L_\xi(\lambda,\mu)$.
\end{theoreme}

\subsection{The semisimple category of non-negligible tilting modules}
\label{sec:ss_tilt}

From now on, we keep the assumption on $l$ of Theorem \ref{thm:tilt_non_negl}. We construct a fusion category using a semisimplification of pivotal categories. As the category of tilting modules is neither abelian, nor spherical, we use the version of semisimplification of pivotal Karoubian categories as in \cite[2.3]{semisimplification}. The hypothesis of \cite[Theorem 2.6]{semisimplification} are satisfied since
\begin{enumerate}
  \item the category of tilting modules for the algebra $\qgrroot{\mathfrak{g}}$ is a subcategory of an abelian category,
  \item up to an invertible element, the positive and negative quantum dimensions of an indecomposable tilting module for the algebra $\qgrroot{\mathfrak{g}}$ are equal.
\end{enumerate}

We then denote by $\mathcal{T}_\xi$ the semisimplification of the
category of tilting $\qdblroot{\mathfrak{g}}$-modules. This is a
semisimple braided pivotal tensor category and the simple objects are
the indecomposable tilting modules $T(\lambda,\mu)$ such that
$\frac{\lambda+\mu}{2}\in C$. This category admits a faithful
$P$-grading
\[
  \mathcal{T}_\xi = \bigoplus_{\nu \in P}\mathcal{T}_{\xi,\nu}
\]
where $\mathcal{T}_{\xi,\nu}$ is additively generated by simple
objects $L_{\xi}(\lambda,\mu)$ with $\lambda-\mu=2\nu$. As in the case
of a generic parameter $q$, each component $\mathcal{T}_{\xi,\nu}$ is equivalent
to the category of tilting modules for $\qgrroot{\mathfrak{g}}$.

We now compute explicitly the $S$-matrix and the twist for $\mathcal{T}_\xi$. The $S$-matrix is the matrix indexed by $\Irr(\mathcal{T}_\xi)$ given by
\[
  S_{(\lambda,\mu),(\lambda',\mu')}=\Tr_{L_\xi(\lambda,\mu)\otimes L_\xi(\lambda',\mu')}^{+}(c_{L_\xi(\lambda',\mu'),L_\xi(\lambda,\mu)}\circ c_{L_\xi(\lambda,\mu),L_\xi(\lambda',\mu')}).
\]
The twist $\theta$ and the pivotal structure $a$ are related using the Drinfeld morphism $u$ (see \cite[8.35]{egno}): $a=u\theta$. The latter is given by the composition
\[
\begin{tikzcd}
  X \ar[r,"\id_X\otimes\coev_{X^*}"] & X\otimes X^* \otimes X^{**}
  \ar[r,"c_{X,X^*}\otimes \id_{X^{**}}"] & X^*\otimes X \otimes
  X^{**} \ar[r,"\ev_X\otimes\id_{X^{**}}"] & X^{**}.
\end{tikzcd}
\]

We recall the expression of the quasi-$R$-matrix $\Theta_\xi$ of
$\qdblroot{\mathfrak{g}}$
\[
  \Theta_\xi=\prod_{\alpha\in\Phi^+} \left(\sum_{n=0}^{+\infty}q_\alpha^{\frac{n(n-1)}{2}}\qfact{n}{q_\alpha}(q_\alpha-q_\alpha^{-1})^{n} E_\alpha^{(n)}\otimes F_\alpha^{(n)}\right).
\]
Recall also that we have chosen the pivotal structure given by the
element $L_{2\rho}$ so that positive and negative quantum traces of
any linear map $f\colon M\rightarrow M$ are given by
\[
  \Tr_M^+(f) = \Tr(L_{2\rho}f\mid M) \quad \text{and} \quad  \Tr_M^-(f) = \Tr(fL_{2\rho}^{-1}\mid M).
\]

\begin{proposition}\label{prop:S_mat-T_mat}
  Let $(\lambda,\mu)$ and $(\lambda',\mu')$ in $\tilde{P}^+$. Then the $S$-matrix of $\mathcal{T}_\xi$ is given by
\[
  S_{(\lambda,\mu),(\lambda',\mu')} =
  \frac{\sum_{w\in W}(-1)^{l(w)}
    \xi^{\langle 2\rho+\lambda,(w\bullet(\lambda',\mu'))_2\rangle + \langle\mu,(w\bullet(\lambda',\mu'))_1+2\rho\rangle}}
  {\sum_{w\in W}(-1)^{l(w)}
    \xi^{\langle 2\rho,w\bullet 0\rangle}},
\]
and the twist is
\[
  \theta_{L_\xi(\lambda,\mu)}=\xi^{\langle\lambda+2\rho,\mu\rangle}\id_{L_\xi(\lambda,\mu)}.
\]
\end{proposition}

\begin{proof}
  Let $L_\xi(\lambda,\mu)$ be a simple object in
  $\mathcal{T}_\xi$. For any other simple object $M$ in
  $\mathcal{T}_\xi$ the map
  \[
  \left(\id_{L_\xi(\lambda,\mu)}\otimes \Tr_M^+\right) \left(
  c_{M,L_\xi(\lambda,\mu)}\circ c_{L_\xi(\lambda,\mu),M}\right)
  \]
  is an endomorphism of the simple object $L_\xi(\lambda,\mu)$ hence
  is a scalar. We compute it on the highest weight vector
  $v_{\lambda,\mu}$ of $L_\xi(\lambda,\mu)$. Let $m\in M$ be a vector
  of weight $(\lambda',\mu')$. We have
  \[
  c_{M,L_\xi(\lambda,\mu)}\circ
  c_{L_\xi(\lambda,\mu),M}(v_{\lambda,\mu}\otimes m) =
  \xi^{\langle\lambda,\mu'\rangle}c_{M,L_\xi(\lambda,\mu)}(m\otimes
  v_{\lambda,\mu})
  \]
  as $v_{\lambda,\mu}$ is a vector of highest weight. Now, we take the
  partial trace of this expression and are interested only in the
  component on $v_{\lambda,\mu}$. As $F_\alpha^{(n)}v_{\lambda,\mu}$
  is not of weight $(\lambda,\mu)$ for $n>0$, this shows that only the term
  $1\otimes 1$ in $\Theta_\xi$ contributes on this component. Therefore
  \[
  \left(\id_{L_\xi(\lambda,\mu)}\otimes \Tr_M^+\right) \left(
  c_{M,L_\xi(\lambda,\mu)}\circ c_{L_\xi(\lambda,\mu),M} \right)=
  \Tr_M^+\left(\varphi_{\lambda,\mu}^M\right)\id_{L_\xi(\lambda,\mu)}
  \]
  where
  $\varphi_{\lambda,\mu}^M(m) = \xi^{\langle\lambda,\mu'\rangle +
    \langle\lambda',\mu\rangle}m$
  for $m$ a vector of weight $(\lambda',\mu')$.

  Finally, the $S$-matrix is given by
  \[
  S_{(\lambda,\mu),(\lambda',\mu')} =
  \dim^+(L_\xi(\lambda,\mu))\Tr^+\left(\varphi_{\lambda,\mu}^{L_\xi(\lambda',\mu')}\right).
  \]
  Using the Weyl character formula, we have
  \[
  S_{(\lambda,\mu),(\lambda',\mu')} = \frac{\sum_{w\in W}(-1)^{l(w)}
    \xi^{\langle 2\rho,(w\bullet(\lambda,\mu))_2\rangle}} {\sum_{w\in
      W}(-1)^{l(w)} \xi^{\langle 2\rho,w\bullet 0\rangle}}
  \frac{\sum_{w\in W}(-1)^{l(w)} \xi^{\langle
      2\rho+\lambda,(w\bullet(\lambda',\mu'))_2\rangle +
      \langle\mu,(w\bullet(\lambda',\mu'))_1\rangle}} {\sum_{w\in
      W}(-1)^{l(w)} \xi^{\langle 2\rho+\lambda+\mu,w\bullet
      0\rangle}}.
  \]
  Note that the assumption on $l$ ensures that the denominators are
  non-zero. As we have
  $\langle 2\rho+\lambda+\mu,w\bullet 0\rangle = \langle
  2\rho,(w^{-1}\bullet(\lambda,\mu))_2\rangle -
  \langle\mu,2\rho\rangle$ the formula for the $S$-matrix becomes
  \[
  S_{(\lambda,\mu),(\lambda',\mu')} = \frac{\sum_{w\in W}(-1)^{l(w)}
    \xi^{\langle 2\rho+\lambda,(w\bullet(\lambda',\mu'))_2\rangle +
      \langle\mu,(w\bullet(\lambda',\mu'))_1+2\rho\rangle}}
  {\sum_{w\in W}(-1)^{l(w)} \xi^{\langle 2\rho,w\bullet 0\rangle}}.
  \]

  We now turn to the twist. Unrolling the definition of the Drinfeld morphism, one can show that for $v_{\lambda,\mu}$
  a highest weight vector of $L_\xi(\lambda,\mu)$
  \[
  u_{L_\xi(\lambda,\mu)}(v_{\lambda,\mu}) = \left(\varphi\in L_{\xi}(\lambda,\mu)^* \mapsto
  \xi^{-\langle\lambda,\mu\rangle}\varphi(v_{\lambda,\mu})\right),
  \]
  whereas
  \[
  a_{L_\xi(\lambda,\mu)}(v_{\lambda,\mu}) = \left(\varphi\in L_{\xi}(\lambda,\mu)^* \mapsto
  \xi^{\langle 2\rho,\mu\rangle}\varphi(v_{\lambda,\mu})\right).
  \]
  Therefore we have
  $\theta_{L_\xi(\lambda,\mu)}=\xi^{\langle\lambda+2\rho,\mu\rangle}\id_{L_\xi(\lambda,\mu)}$.
\end{proof}

\subsection{A partial modularization}
\label{sec:fus_tilt}

Contrary to the category of tilting modules for
$\qgrroot{\mathfrak{g}}$, $\mathcal{T}_\xi$ has an infinite number of non-isomorphic simple objects. Following Müger \cite{mueger}, we first find some objects in the symmetric center of
$\mathcal{T}_\xi$, and we add isomorphism between these objects and
the unit object.

\begin{proposition}
  Let $\nu\in (l'Q^\vee)\cap P$. Then the invertible simple object
  $L_\xi(\nu,-\nu)$ lies in the symmetric center of
  $\mathcal{T}_\xi$. Its positive and negative quantum dimension is $1$ and
  its twist is $1$ or $-1$.
\end{proposition}

\begin{proof}
  For any simple object $L_\xi(\lambda,\mu)$ in $\mathcal{T}_\xi$, we
  compute
  ${c_{L_\xi(\lambda,\mu),L_\xi(\nu,-\nu)}\circ
  c_{L_\xi(\nu,-\nu), L_\xi(\lambda,\mu)}}$. As
  $L_\xi(\nu,-\nu)\otimes L_\xi(\lambda,\mu)$ is a simple object,
  the double braiding is the multiplication by a scalar. By a
  computation on a highest weight vector, the double braiding is the multiplication by
  $\xi^{\langle\nu,\mu-\lambda\rangle}$. But as $\nu$ is in
  $l'Q^\vee$ and $\mu-\lambda\in 2P$, we have $\langle\nu,\mu-\lambda\rangle \in 2l'\mathbb{Z}$, so
  that the double braiding is the identity. 

  The positive quantum dimension of $L_\xi(\nu,-\nu)$ is given by the
  action of $L_{2\rho}$, so is equal to $\xi^{-\langle
    2\rho,\nu\rangle}=1$ because $\rho\in P$.

  The twist is given by $\xi^{\langle \nu+2\rho,-\nu\rangle}$ which
  is obviously equal to $\pm 1$ since $\nu\in l'Q^\vee$.
\end{proof}

Let $\mathcal{S}$ be the tensor subcategory of $\mathcal{T}_{\xi}$ generated
by the $L_{\xi}(\nu,-\nu)$ (which we will denote by $I(\nu)$), with $\nu\in (l'Q^\vee)\cap P$ and of twist $1$. We
recall the construction of the category $\mathcal{T}_\xi\rtimes
\mathcal{S}$ of \cite[Definition 3.12]{mueger}, which is simpler in our
case since the objects in $\mathcal{S}$ are all of dimension $1$. We
choose for any $\nu,\nu'\in (l'Q^\vee)\cap P$ an isomorphism
$\varphi_{\nu,\nu'}$ in the one dimensional space $\Hom_{\mathcal{S}}(I(\nu)\otimes
I(\nu'),I(\nu+\nu'))$ such that the following diagram commutes
\[
  \begin{tikzcd}[column sep=large]
    I(\nu)\otimes I(\nu')\otimes I(\nu'')
    \ar[r,"\varphi_{\nu,\nu'}\otimes \id"]
    \ar[d,"\id\otimes\varphi_{\nu',\nu''}"] & I(\nu+\nu')\otimes
    I(\nu'')\ar[d,"\varphi_{\nu+\nu',\nu''}"] \\ 
    I(\nu)\otimes I(\nu'+\nu'') \ar[r,"\varphi_{\nu,\nu'+\nu''}"] &
    I(\nu+\nu'+\nu'') 
  \end{tikzcd}.
\]
To do so, choose in any $I(\nu)$ a non-zero vector $v_\nu$. Then
$\varphi_{\nu,\nu'}$ sends $v_\nu\otimes v_{\nu'}$ to $v_{\nu+\nu'}$.

We consider the category $\mathcal{T}_\xi\rtimes_0 \mathcal{S}$ with the same
objects as $\mathcal{T}_\xi$ and with space of morphisms
between two objects $X$ and $Y$
\[
  \Hom_{\mathcal{T}_\xi\rtimes_0 \mathcal{S}}(X,Y) = \bigoplus_{\nu\in
    P\cap Q^\vee} \Hom_{\mathcal{T}_\xi}(X,I(\nu)\otimes Y).
\]

The composition of $f\in \Hom_{\mathcal{T}_\xi}(X,I(\nu)\otimes Y)$ and
of $g\in \Hom(Y,I(\nu')\otimes Z)$ in $\mathcal{T}_\xi\rtimes
\mathcal{S}$ is given by
\[
  \begin{tikzcd}[column sep=large]
    X\ar[r,"f"] & I(\nu)\otimes Y \ar[r,"\id\otimes g"]&
    I(\nu)\otimes I(\nu') \otimes Z \ar[r,"\varphi_{\nu,\nu'}\otimes
    \id"] & I(\nu+\nu')\otimes Z
  \end{tikzcd}.
\]
Due to the compatibility of the maps $\varphi$, it is easy to check
that this defines an associative composition.

This category has tensor products: on objects the tensor product is
the same as the one in $\mathcal{T}_\xi$ and if
$f\in\Hom_{\mathcal{T}_\xi}(X,I(\nu)\otimes Y)$ and
$g\in\Hom_{\mathcal{T}_\xi}(X',I(\nu')\otimes Y')$, their tensor
product is defined as the composition
\[
  \begin{tikzcd}[column sep=huge]
    X\otimes X' \ar[r,"f\otimes g"] & I(\nu)\otimes Y \otimes I(\nu')
    \otimes Y' \ar[r,"\id\otimes c_{Y,I(\nu')}\otimes \id"] &
    I(\nu)\otimes I(\nu') \otimes Y \otimes Y'\\
    & \phantom{I(\nu)\otimes Y \otimes I(\nu')\otimes Y'}\ar[r,"\varphi_{\nu,\nu'}\otimes\id"] & I(\nu+\nu')\otimes Y
    \otimes Y'.
  \end{tikzcd}
\]
Again, the compatibility of $\varphi$ ensures that this tensor product
endows $\mathcal{T}_\xi\rtimes_0 \mathcal{S}$ with a structure of a
semisimple tensor category, see \cite[Section 3.2]{mueger} for further details.


The duality on $\mathcal{T}_\xi$ extends to a duality on $\mathcal{T}_\xi\rtimes_0\mathcal{S}$. One may check that the pivotal structure in $\mathcal{T}_\xi$ induces
a pivotal structure on $\mathcal{T}_\xi\rtimes_0
\mathcal{S}$. For this, it is crucial that the objects in $\mathcal{S}$ are of twist $1$.

Finally, the category $\mathcal{T}_\xi\rtimes_0 \mathcal{S}$ is
braided because $\mathcal{S}$ is a subcategory of the symmetric center
of $\mathcal{T}_\xi$ (see \cite[Lemma 3.10]{mueger}).

Now, in the general case, it can happen that the category constructed
above is not idempotent complete. This happens exactly when tensoring
by a non-trivial simple object of $\mathcal{S}$ has a fixed point on
the set of simple objects of $\mathcal{T}_\xi$. But the object
$I(\nu)$ is in the component $\mathcal{T}_{\xi,\nu}$ of the grading,
thus tensoring by this object has no fixed points on the set of simples,
provided that $\nu\neq 0$. Therefore the idempotent completion
$\mathcal{T}_\xi\rtimes \mathcal{S}$ of $\mathcal{T}_\xi\rtimes_0
\mathcal{S}$ is $\mathcal{T}_\xi\rtimes_0 \mathcal{S}$ itself.

\begin{proposition}
  The category $\mathcal{T}_\xi\rtimes \mathcal{S}$ is finite in the sense of \cite[Definition 1.8.6]{egno}.
\end{proposition}

\begin{proof}
  Indeed, denoting by $G$ the quotient of $P$ by
  $\{v\in (l'Q^\vee)\cap P \mid \theta_{I(\nu)}=1\}$,
  $\mathcal{T}_\xi\rtimes \mathcal{S}$ is a $G$-graded category, with
  each homogeneous component equivalent to the category of tilting
  modules for $\qgrroot{\mathfrak{g}}$
  \[
  \mathcal{T}_\xi\rtimes \mathcal{S} = \bigoplus_{\nu\in
    G}(\mathcal{T}_\xi\rtimes \mathcal{S})_\nu.
  \]
  Therefore, $\mathcal{T}_\xi\rtimes \mathcal{S}$ has $|C||G|$ simple
  objects. Note that $G$ is indeed finite since $P/((l'Q^\vee)\cap P)$
  surjects on $G$.
\end{proof}

\subsection{An integral subcategory}
\label{sec:subcat}

We consider the full subcategory of $\mathcal{T}_\xi$ additively generated by the $L_\xi(\lambda,\mu)$ with $\lambda+\mu\in 2C$ and $\mu \in Q$. This subcategory is stable by tensor product: we can easily see this at the level of $\qdblroot{\mathfrak{g}}$-modules. We denote this category by $\mathbb{Z}\left(\mathcal{T}_\xi\right)$. In the case of the tilting category for $\qgrroot{\mathfrak{sl}_{n+1}}$, Masbaum and Wenzl \cite{masbaum-wenzl} show that an analogue of this subcategory is a modular category provided that $l$ is even and $l'$ and $n+1$ are relatively prime.

As for the category $\mathcal{T}_\xi$, there is an infinite number of simple objects. But the objects of the form $I(\nu)$ with $\nu\in (l'Q^\vee)\cap Q$ are in the symmetric center of $\mathbb{Z}\left(\mathcal{T}_\xi\right)$.

\begin{lemma}
  \label{lem:twist_1}
  The twist of $I(\nu)$ is $1$ for any $\nu\in(l'Q^\vee)\cap Q$.
\end{lemma}

\begin{proof}
 We want to show that for any $\nu\in (l'Q^{\vee})\cap Q$, we have $\langle\nu,\nu\rangle \equiv 0 \mod l$. For any value of the integer $l$, we have $\langle\nu,\nu\rangle \in l'^2\mathbb{Z}$. Therfore we may, and will, suppose that $l=2l'$ with $l'$ odd.

  Suppose first that $\mathfrak{g}$ is not of type $G_2$. Then, $(l'Q^\vee)\cap Q = l'Q$ and we easily conclude because for any $\lambda\in Q$, $\langle\lambda,\lambda\rangle$ is even (recall that $\langle\cdot,\cdot\rangle$ is normalized such that $\langle\alpha,\alpha\rangle = 2$ for short simple roots).

  Now suppose that $\mathfrak{g}$ is of type $G_2$. If moreover $l'$ is not a multiple of $3$, we have $(l'Q^\vee)\cap Q = l'Q$ and the conclusion follows. We finally suppose that $l=2l'$ with $l'$ a multiple of $3$. Let $\alpha$ be the short simple root and $\beta$ be the long simple root, so that $\alpha^{\vee} = \alpha$ and $\beta^{\vee}=\frac{\beta}{3}$. We therefore have $(l'Q^{\vee})\cap Q = l'Q^\vee$ and for any $\nu=l'\nu_\alpha \alpha + \frac{l'\nu_\beta}{3}\beta\in l'Q^\vee$, we have
\[
  \langle\nu,\nu\rangle = l'^2\left(2\nu_\alpha^2-2\nu_\alpha\nu_\beta+\frac{2}{3}\nu_\beta^2\right) = 2l'\frac{l'}{3}(3\nu_\alpha^2-3\nu_\alpha\nu_\beta+\nu_\beta^2).
\]  
Hence $\langle\nu,\nu\rangle \equiv 0 \mod l$ since $l'$ is a multiple of $3$.
\end{proof}

Hence, we have shown that for any $\nu\in(l'Q^\vee)\cap Q$, $I(\nu)$
is of quantum dimension $1$ and of twist $1$. We again construct the
category $\mathbb{Z}(\mathcal{T}_\xi)\rtimes \mathcal{S}$ and we
obtain a $P/((l'Q^\vee)\cap Q)$-graded category
\[
  \mathbb{Z}(\mathcal{T}_\xi)\rtimes \mathcal{S} = \bigoplus_{\nu\in
    P/((l'Q^\vee)\cap Q)} \mathbb{Z}(\mathcal{T}_\xi)_\nu.
\]
Note that we do not have an equivalence of categories between
$\mathbb{Z}(\mathcal{T}_\xi)_\nu$ and $\mathbb{Z}(\mathcal{T}_\xi)_0$
if $\nu \not\in Q$. But we have an equivalence between
$\mathbb{Z}(\mathcal{T}_\xi)_\nu$ and
$\mathbb{Z}(\mathcal{T}_\xi)_{\nu'}$ if $\nu=\nu'$ in $P/Q$.

\begin{proposition}
  The category $\mathbb{Z}(\mathcal{T}_\xi)\rtimes \mathcal{S}$ is finite and has $|C||Q/((l'Q^\vee)\cap Q)|$ simple objects.
\end{proposition}

\begin{proof}
  If we choose some representatives $\nu_1,\ldots,\nu_k$ of
  $P/Q$ in $P/((l'Q^\vee)\cap Q)$, we have $|C|$ simple objects in
  $\bigoplus_{i=1}^k\mathbb{Z}(\mathcal{T}_\xi)_{\nu_i}$. Therefore
  $\mathbb{Z}(\mathcal{T}_\xi)\rtimes \mathcal{S}$ has
  $|C||Q/((l'Q^\vee)\cap Q)|$ simple objects.
\end{proof}

We end this part with some notation. Denote by $\tilde{C}$ the following set of weights
\[
  \tilde{C} = \{(\lambda,\mu)\in P\times Q \mid \lambda+\mu\in 2C\}.
\]
This set parametrizes the simple objects of $\mathbb{Z}(\mathcal{T}_\xi)$ and the group $l'Q^\vee\cap Q$ acts on it by
\[
  \nu \cdot (\lambda,\nu) = (\lambda+\nu,\mu-\nu)
\]
for $(\lambda,\mu)\in\tilde{C}$ and $\nu \in (l'Q^\vee)\cap Q$. The set $\overline{C} = \tilde{C}/((l'Q^\vee)\cap Q)$ parametrizes the simple objects of $\mathbb{Z}(\mathcal{T}_\xi)\rtimes \mathcal{S}$.

\subsection{Non-degeneracy of $\mathcal{T}_\xi\rtimes \mathcal{S}$ and degeneracy of $\mathbb{Z}(\mathcal{T})_\xi\rtimes \mathcal{S}$}

In this subsection, we suppose that $l=2Dd$ for $d\geq h^\vee$. Then $l'=Dd$ and $DdQ^\vee\subseteq Q$ and by Lemma \ref{lem:twist_1} for any $\nu\in DdQ^\vee$, we have $\theta_{I(\nu)} = 1$.

\begin{proposition}\label{prop:non_deg}
  If moreover $\xi=\exp\left(\frac{i\pi}{d}\right)$, the category $\mathcal{T}_\xi\rtimes \mathcal{S}$ is non-degenerate: if a simple object $X$ is such that for all object $Y$ in $\mathcal{C}$ we have
\[
  c_{X,Y}\circ c_{Y,X}=\id_{Y\otimes X}
\]
then $X\simeq \mathbf{1}$.
\end{proposition}

\begin{proof}
  Using the proof of \cite[Theorem 3.3.20]{bakalov_kirillov}, we show that the square of the $S$-matrix of $\mathcal{T}_\xi\rtimes \mathcal{S}$ is invertible. As for any $w\in W$ and $(\lambda,\mu)\in\tilde{P}$ we have
  \[
    w\bullet(\lambda,\mu) = \left(w\bullet \frac{\lambda+\mu}{2} + \frac{\lambda-\mu}{2}, w\bullet \frac{\lambda+\mu}{2} - \frac{\lambda-\mu}{2}\right),
  \]
  we can rewrite the $S$-matrix of $\mathcal{T}_\xi\rtimes \mathcal{S}$ as
  \[
    S_{(\lambda,\mu),(\lambda',\mu')}=\xi^{\langle 2\rho,\mu-\eta+\mu'-\eta'\rangle-2\langle\mu-\eta,\mu'-\eta'\rangle}\tilde{s}_{\eta,\eta'}
  \]
  where $\lambda+\mu = 2\eta$, $\lambda'+\mu'=2\eta'$ and $\tilde{s}$ is the $S$-matrix of the modular category of tilting modules for $\qgrroot{\mathfrak{g}}$ (see proof of \cite[Theorem 3.3.20]{bakalov_kirillov}). The simple objects of $\mathcal{T}_\xi\rtimes \mathcal{S}$ are indexed by $\{(\lambda,\mu)\in P\times P \mid \lambda+\mu \in 2C\}/DdQ^\vee$ which is in bijection with $C\times P/DdQ^\vee$ sending $(\lambda,\mu)$ to $(\eta,\mu)$. Therefore
  \begin{align*}
    (S^2)_{(\lambda,\mu),(\lambda'',\mu'')} &= \xi^{\langle 2\rho,\mu-\eta+\mu''-\eta''\rangle}\sum_{\eta'\in C}\tilde{s}_{\eta,\eta'}\tilde{s}_{\eta',\eta''}\xi^{-2\langle\eta',2\rho+\eta-\mu+\eta''-\mu''\rangle}\underbrace{\sum_{\mu'\in P/DdQ^\vee}\xi^{2\langle\mu',2\rho+\eta-\mu+\eta''-\mu''\rangle}}_{= |P/DdQ^\vee|\delta_{2\rho+\eta-\mu+\eta''-\mu''\in DdQ^\vee}}\\
    &= |P/DdQ^\vee|\delta_{2\rho+\eta-\mu+\eta''-\mu''\in DdQ^\vee}\xi^{\langle 2\rho,2\rho\rangle}\sum_{\eta'\in C}\tilde{s}_{\eta,\eta'}\tilde{s}_{\eta',\eta''}\\
    &= \delta_{2\rho+\eta-\mu+\eta''-\mu''\in DdQ^\vee}\delta_{\eta'',-w_0(\eta)} \kappa\\
    &= \delta_{\eta'',-w_0(\eta)}\delta_{\mu''\in 2\rho+\eta-\mu-w_0(\eta)+DdQ^\vee}\kappa,
  \end{align*}
 where $\kappa$ is a non-zero constant. As $\mu\in P/DdQ^\vee$, the matrix $S^2$ is, up to a non-zero constant, the invertible permutation matrix $(\delta_{(\lambda'',\mu''),(-2\rho,2\rho)-w_0(\lambda,\mu)})_{(\lambda,\mu),(\lambda'',\mu'')}$.
\end{proof}

The square of the $S$-matrix is not given by the duality: this is due to the fact that $\mathcal{T}_\xi\rtimes \mathcal{S}$ is not spherical, but only pivotal.

Now we turn to the category $\mathbb{Z}(\mathcal{T})_\xi\rtimes \mathcal{S}$ which is degenerate in general.

\begin{theoreme}\label{thm:int_deg}
  The symmetric center of $\mathbb{Z}(\mathcal{T})_\xi\rtimes \mathcal{S}$ contains $\lvert P/Q \rvert$ simple objects. Therefore, the category $\mathbb{Z}(\mathcal{T})_\xi\rtimes \mathcal{S}$ is non-degenerate if and only if $\mathfrak{g}$ is of type $E_8$, $F_4$ or $G_2$.
\end{theoreme}

\begin{proof}
  Using \cite[Lemma 8.20.9]{egno}, the simple object $L_\xi(\lambda,\mu)$ is in the symmetric center of $\mathbb{Z}(\mathcal{T})_\xi\rtimes \mathcal{S}$ if and only if
\[
  (h_{\lambda,\mu},h_{0,0})_{\mathbb{Z}(\mathcal{T})_\xi\rtimes \mathcal{S}}=\sum_{(\lambda',\mu')\in \overline{C}}\frac{\dim^-(L_\xi(\lambda',\mu'))}{\dim^+(L_\xi(\lambda',\mu'))}S_{(\lambda,\mu),(\lambda',\mu')}S_{(\lambda',\mu'),(0,0)} \neq 0.
\]

  As in the proof of Proposition \ref{prop:non_deg}, we let $2\eta=\lambda+\mu$ and $2\eta'=\lambda+\mu'$, which are elements in $C$. We then have
  \begin{align*}
    (h_{\lambda,\mu},h_{0,0})_{\mathbb{Z}(\mathcal{T})_\xi\rtimes \mathcal{S}}
    &=\sum_{\eta'\in C}\sum_{\mu'\in Q/DdQ^\vee} 
      \xi^{2\langle 2\rho,\eta'-\mu'\rangle}
      \xi^{\langle 2\rho,\mu-\eta+\mu'-\eta'\rangle-2\langle\mu-\eta,\mu'-\eta'\rangle}\tilde{s}_{\eta,\eta'}
      \xi^{\langle 2\rho,\mu'-\eta'\rangle}\tilde{s}_{\eta',0}\\
   &=\xi^{\langle 2\rho,\mu-\eta\rangle}\sum_{\eta'\in C}\xi^{-2\langle\eta',\eta-\mu\rangle}\tilde{s}_{\eta,\eta'}\tilde{s}_{\eta',0}\underbrace{\sum_{\mu'\in Q/DdQ^\vee}\xi^{2\langle\mu',\eta-\mu \rangle}}_{\lvert Q/DdQ^\vee\rvert\delta_{\mu-\eta\in DdP^\vee}}\\
  &= \lvert Q/DdQ^\vee\rvert\delta_{\mu-\eta\in DdP^\vee}
    \sum_{\eta'\in C}\tilde{s}_{\eta,\eta'}\tilde{s}_{\eta',0}\xi^{2\langle\rho+\eta',\mu-\eta\rangle}.
  \end{align*}
 
  \begin{lemma}\label{lem:Pdual_Qdual}
    Let $\gamma \in DdP^\vee$. For any $w\in W$, we have $\gamma-w(\gamma) \in DdQ^\vee$.
  \end{lemma}

  \begin{proof}
    See \cite[VI.1.10, Proposition 27]{bourbaki_lie_456}.
  \end{proof}

  Now, fix $\gamma\in DdP^\vee$. Following \cite[Section 3.3]{bakalov_kirillov}, we have an action of the affine Weyl group $W^a = W \ltimes DdQ^\vee$ on $P$ such that $C$ contains exacly one element for every orbit with trivial stabilizer for the dot action (note that in \cite[Section 3.3]{bakalov_kirillov}, the translation subgroup of $W^a$ is generated by $dQ^\vee$, but $Q^\vee$ is embedded in $P$ using $D^{-1}\langle\cdot,\cdot\rangle$, so that it coincides with our notations). Using the fact that $\tilde{s}_{\eta,w\bullet\eta'} = (-1)^{l(w)}\tilde{s}_{\eta,\eta'}$ and Lemma \ref{lem:Pdual_Qdual}, we see that $\tilde{s}_{\eta,\eta'}\tilde{s}_{\eta',0}\xi^{2\langle\rho+\eta',\mu-\eta\rangle}$ is invariant by the action of $W^a$. We therefore can replace the summation on $C$ by a summation on $P/DdQ^\vee$ and we obtain, using the formula for $\tilde{s}$ in the proof of \cite[Theorem3.3.20]{bakalov_kirillov}:
  \begin{align*}
    \sum_{\eta'\in C}\tilde{s}_{\eta,\eta'}\tilde{s}_{\eta',0}\xi^{2\langle\rho+\eta',\gamma\rangle}
    &= \frac{1}{\lvert W \rvert \kappa}\sum_{w,w'\in W}(-1)^{l(w)+l(w')}\sum_{\eta'\in P/DdQ^\vee}\xi^{2\langle\eta'+\rho,w(\eta+\rho)+w'(\rho)+\gamma\rangle}\\
    &= \frac{\lvert P/DdQ^\vee\rvert}{\lvert W \rvert\kappa}\sum_{w,w'\in W}(-1)^{l(w)+l(w')}\delta_{w(\eta+\rho)+w'(\rho)+\gamma \in DdQ^\vee},  
  \end{align*}
  where $\kappa$ is a non-zero constant. 

  As $\gamma\in DdP^\vee$ the stabilizer of $\gamma$ for the dot action is trivial, and there exist a unique $\tilde{\gamma}\in C$ and $\tilde{w}\in W$ such that $\gamma+\rho \in \tilde{w}(\tilde{\gamma}+\rho) + DdQ^\vee$. Now $w(\eta+\rho)+w'(\rho)+\gamma \in DdQ^\vee$ if and only if $w(\eta+\rho)\in -\gamma-w'(\rho)+DdQ^\vee$. But $w'(\rho) \in -w'(\gamma) + w'\tilde{w}(\tilde{\gamma}+\rho) +DdQ^\vee$ and therefore, using Lemma \ref{lem:Pdual_Qdual}, $w(\eta+\rho)+w'(\rho)+\gamma \in DdQ^\vee$ if and only if $\eta+\rho \in w^{-1}w'\tilde{w}w_0(-w_0(\tilde{\gamma})+\rho)+DdQ^\vee$. But this is possible if and only if $\eta=-w_0(\tilde{\gamma})$ and $w=w'\tilde{w}w_0$ and therefore  
\[
  \sum_{\eta'\in C}\tilde{s}_{\eta,\eta'}\tilde{s}_{\eta',0}\xi^{2\langle\rho+\eta',\gamma\rangle} = \kappa (-1)^{l(\tilde{w})+l(w_0)}\delta_{\eta,-w_0(\tilde{\gamma})}.
\]  

  Therefore, the simple objects in the symmetric center of $\mathbb{Z}(\mathcal{T})_\xi\rtimes \mathcal{S}$ are indexed by $(\lambda,\mu)\in\overline{C}$ where $\frac{\mu-\lambda}{2}=\gamma\in DdP^\vee/DdQ^\vee$ and $\frac{\mu+\lambda}{2} = -w_0(\tilde{\gamma})$, where $\tilde{\gamma}$ is the only element in $C$ in the orbit of $\gamma$ under the dot action of $W^a$.
\end{proof}


\section{The type $A$}
\label{sec:typeA}
In this section, we investigate in details the category
$\mathbb{Z}(\mathcal{T}_\xi)\rtimes \mathcal{S}$ for $\mathfrak{sl}_{n+1}$ at an even root of unity.

\boitegrise{{\bf Notations.} {\emph{In this section $\mathfrak{g} =
\mathfrak{sl}_{n+1}$ and $\xi$ is a primitive $2d$-th root of unity, where
$d\geq n+1$. To be consistent with the notation of Section \ref{sec:spe_tilt} we set $l=2d$ and $l'=d$. We use the conventions of \cite[Planche I]{bourbaki_lie_456} for the labelling of roots.}}}{0.8\textwidth}

\subsection{The category $\mathbb{Z}(\mathcal{T}_\xi)\rtimes
  \mathcal{S}$}
\label{sec:categ-typeA}

In type $A$, we have $Q^\vee = Q$ therefore this category has $|C||Q/dQ| =
d^n|C|$ simple objects. From the description of $C$ given in Section \ref{sec:tilting}, we have 
\[
  C = \left\{\sum_{i=1}^n\eta_i\varpi_i \in P^+\ \middle\vert\ \sum_{i=1}^n\eta_i \leq d-(n+1)\right\}
\]
so that $|C| = \binom{d-1}{n}$. Recall the notations $\tilde{C}$ and $\overline{C}$ at the end of Section \ref{sec:subcat}.

The category $\mathbb{Z}(\mathcal{T}_\xi)\rtimes\mathcal{S}$ is not
modular by Theorem \ref{thm:int_deg}. To compute the $n+1$ simple objects in the symmetric center of $\mathbb{Z}(\mathcal{T}_\xi)\rtimes\mathcal{S}$, we follow the proof of Theorem \ref{thm:int_deg}: for every $\gamma\in dP^\vee/dQ^\vee$, we find a representative $\tilde{\gamma}\in C$ of $\gamma$ for the dot action of the affine Weyl group $W^a$. The group $dP^\vee/dQ^\vee$ is generated by the image of $d\varpi_n$ and we have
\[
  s_ns_{n-1}\cdots s_1(d\varpi_n+\rho) - \rho \equiv d\varpi_n - \sum_{i=1}^ns_ns_{n-1}\cdots s_{n+2-i}(\alpha_{n+1-i}) \mod dQ^\vee.
\]
But $s_ns_{n-1}\cdots s_{n+2-i}(\alpha_{n+1-i}) = \sum_{j=n+1-i}^n\alpha_j$ and therefore
\[
  s_ns_{n-1}\cdots s_1 \bullet (d\varpi_n) \equiv d\varpi_n - (n+1)\varpi_n \mod dQ^\vee,
\]
which is indeed an element of $C$. Therefore the element 
\[
I=L_\xi((d-(n+1))\varpi_1-d\varpi_n,(d-(n+1))\varpi_1+d\varpi_n)
\]
is in the symmetric center of $\mathbb{Z}(\mathcal{T}_\xi)\rtimes
\mathcal{S}$. As $\varpi_1+\varpi_n \in Q^\vee$, we have an isomorphism $I\simeq L_\xi((2d-(n+1))\varpi_1,-(n+1)\varpi_1)$ in $\mathbb{Z}(\mathcal{T}_\xi)\rtimes \mathcal{S}$.

\begin{proposition}\label{prop:sym_center_A}
  The symmetric center of $\mathbb{Z}(\mathcal{T}_\xi)\rtimes\mathcal{S}$ is generated by $I$ as a tensor category. The object $I$ is of positive and negative quantum dimension $(-1)^n$ and of twist $1$. Moreover, tensorisation by $I$ has no fixed points on the set of simple objects of $\mathbb{Z}(\mathcal{T}_\xi)\rtimes\mathcal{S}$.
\end{proposition}

\begin{proof}
  First, the object $I$ is invertible because the object $L_\xi((d-(n+1))\varpi_1,(d-(n+1))\varpi_1)$ in $\mathcal{T}_\xi$ is invertible and the tensor product $L_\xi((d-(n+1))\varpi_1,(d-(n+1))\varpi_1)\otimes L_\xi(\eta,\eta)$, $\eta = \sum_{i=1}^n\eta_i\varpi_i \in C$, is given by
  \[
    L_\xi((d-(n+1))\varpi_1,(d-(n+1))\varpi_1)\otimes
    L_\xi(\eta,\eta) \simeq L_\xi\left(\sum_{i=1}^n\eta_{i-1}\varpi_i,\sum_{i=1}^n\eta_{i-1}\varpi_i\right),
  \]
  where we set $\eta_0 = d-(n+1)-\sum_{i=1}^n\eta_i \geq 0$ (see the proof of \cite[Lemme 5.1]{bruguieres}). Let $(\lambda,\mu) \in \tilde{C}$. Write this weight as
  \[
    \mu = \sum_{i=1}^n\mu_i\alpha_i \quad \text{and} \quad \lambda = -\mu + 2\sum_{i=1}^n\eta_i\varpi_i
  \]
  with $\mu_i \in \mathbb{Z}$, $\eta_i\in\mathbb{N}$ and $\sum_{i=1}^n\eta_i \leq d-(n+1)$.
  Therefore, using the braiding, we have
  \[
    I\otimes L_\xi(\lambda,\mu) \simeq L_\xi\left(\lambda + \sum_{i=1}^n(\eta_{i-1}-\eta_i)\varpi_i +d\varpi_1,\mu + \sum_{i=1}^n(\eta_{i-1}-\eta_i)\varpi_i -d\varpi_1\right).
  \]
  
  From this, we see that the objects $I^{\otimes k}$ for $0\leq k \leq n$ are non-isomorphic. Since $I$ is in the symmetric center of $\mathbb{Z}(\mathcal{T}_\xi)\rtimes\mathcal{S}$, which contains $\lvert P/Q \rvert = n+1$ simple objects, all the simple objects in this symmetric center are given by the powers of $I$.
  
  We can compute the quantum dimension directly in $\mathcal{C}_{\xi}$. As a $\qgrroot{\mathfrak{sl}_{n+1}}$-module, the quantum dimension of $L_\xi((d-(n+1))\varpi_1,(d-(n+1))\varpi_1)$ is $1$: this object is invertible and its quantum dimension is positive \cite[Theorem 3.3.9]{bakalov_kirillov}. Therefore, the positive quantum dimension of $I$ is the one of $L_\xi(d\varpi_1,-d\varpi_1)$ which is $\xi^{-d\langle 2\rho,\varpi_1\rangle} = (-1)^n$.
  
  The twist is given by $\xi^{\langle(2d-(n+1))\varpi_1+2\rho, -(n+1)\varpi_1\rangle} = \xi^{-2dn} = 1$. The last assertion is easy, since the grading of $I$ is $d\varpi_1 \not\in dQ$.
\end{proof}

\subsection{Dimension of $\mathbb{Z}(\mathcal{T}_\xi)\rtimes\mathcal{S}$ and renormalization}
\label{sec:renorm}

Thanks to the decomposition $L_{\xi}(\lambda,\mu)\simeq L_{\xi}\left(\frac{\lambda+\mu}{2},\frac{\lambda+\mu}{2}\right)\otimes L_{\xi}\left(\frac{\lambda-\mu}{2},-\frac{\lambda-\mu}{2}\right)$ the squared norm of 
$L_{\xi}(\lambda,\mu)$ is the same as the one of $L_{\xi}\left(\frac{\lambda+\mu}{2},\frac{\lambda+\mu}{2}\right)$. Therefore
\[
  \dim(\mathbb{Z}(\mathcal{T}_\xi)\rtimes\mathcal{S}) = d^n N,
\]
where $N$ is the dimension of the fusion category of tilting modules for $\qgrroot{\mathfrak{sl}_{n+1}}$, which is well known (see \cite[Theorem 3.3.20]{bakalov_kirillov}, except our $N$ is their $D^2$)
\[
  N = (n+1)d^n(-1)^{\lvert\Phi^+\rvert}\prod_{\alpha\in\Phi^+}\frac{1}{(\xi^{\langle\alpha,\rho\rangle}-\xi^{-\langle\alpha,\rho\rangle})^2}.
\]
The Weyl character formula gives
\[
  \prod_{\alpha\in\Phi^+}\xi^{\langle\alpha,\rho\rangle}-\xi^{-\langle\alpha,\rho\rangle} = \xi^{-2\langle\rho,\rho\rangle}\sum_{w\in\mathfrak{S}_{n+1}}(-1)^{l(w)}\xi^{\langle 2\rho,w\bullet 0\rangle}.
\]

Now, as the category $\mathbb{Z}(\mathcal{T}_\xi)\rtimes\mathcal{S}$ is degenerate, we construct a non-degenerate category from it. If $n$ is even, all elements in the symmetric center are of quantum dimension $1$ and of twist $1$: adding an isomorphism between $I$ and $\mathbf{1}$ as in Sections \ref{sec:fus_tilt} and \ref{sec:subcat} yields a non-degenerate category of dimension 
\[
  \frac{\dim(\mathbb{Z}(\mathcal{T}_\xi)\rtimes\mathcal{S})}{n+1} = d^{2n}(-1)^{\lvert\Phi^+\rvert} \xi^{-\langle 2\rho,2\rho\rangle}\left(\sum_{w\in\mathfrak{S}_{n+1}}(-1)^{l(w)}\xi^{\langle 2\rho,w\bullet 0\rangle}\right)^{-2}.
\]

If $n$ is odd, one half of the simple objects of the symmetric center are of quantum dimension $1$ and of twist $1$ and the other half is of quantum dimension $-1$ and of twist $1$. We first add an isomorphism between $I\otimes I$ and $\mathbf{1}$ and we obtain a slightly degenerate category of superdimension
\[
  \frac{\dim(\mathbb{Z}(\mathcal{T}_\xi)\rtimes\mathcal{S})}{n+1} = d^{2n}(-1)^{\lvert\Phi^+\rvert} \xi^{-\langle 2\rho,2\rho\rangle}\left(\sum_{w\in\mathfrak{S}_{n+1}}(-1)^{l(w)}\xi^{\langle 2\rho,w\bullet 0\rangle}\right)^{-2},
\]
see \cite{super_application} for more details. By adding an isomorphism of odd degree between $I$ and $\mathbf{1}$, we obtain a braided pivotal superfusion category, as in \cite[Section 3]{super_application}.

In both cases, the category obtained is not spherical and we renormalize the $S$-matrix by a factor involving the dimension on a certain object $\bar{\mathbf{1}}$ introduced in \cite{super_application}: it is an object $X$ such that the character induced by $X$ on $\Gr(\mathbb{Z}(\mathcal{T}_\xi)\rtimes\mathcal{S})$ is the negative quantum dimension. If such an object exists, it is not unique as tensorization by the symmetric center leaves invariant the character induced at the level of the Grothendieck group. Therefore, if $n$ is even, $\bar{\mathbf{1}}$ is well defined in the modularization of $\mathbb{Z}(\mathcal{T}_\xi)\rtimes\mathcal{S}$, and if $n$ is odd, it is well defined in the super category associated to $\mathbb{Z}(\mathcal{T}_\xi)\rtimes\mathcal{S}$.

\begin{proposition}
  The object $\bar{\mathbf{1}}$ belongs to the orbit of $L_{\xi}(-2\rho,2\rho)$ under tensorisation by the symmetric center.
\end{proposition}

\begin{proof}
  We show that the character of $\Gr(\mathbb{Z}(\mathcal{T}_\xi)\rtimes\mathcal{S})$ induced by $L_{\xi}(-2\rho,2\rho)$ is the negative quantum dimension. Denote by $\chi$ the character of $\Gr(\mathbb{Z}(\mathcal{T}_\xi)\rtimes\mathcal{S})$ defined by
  \[
    \chi(X) = \frac{S_{L_{\xi}(-2\rho,2\rho),X}}{\dim^+(L_{\xi}(-2\rho,2\rho))}.
  \]
  The formula for the $S$-matrix, together with $\dim^+(L_{\xi}(-2\rho,2\rho))=\xi^{\langle 2\rho,2\rho\rangle}$ gives
  \[
    \chi(L_\xi(\lambda,\mu)) = \frac{\sum_{w\in\mathfrak{S}_{n+1}}(-1)^{l(w)}\xi^{\langle\lambda+2\rho,(w\bullet(-2\rho,2\rho))_2\rangle + \langle(w\bullet(-2\rho,2\rho))_1+2\rho,\mu\rangle }}{\xi^{\langle 2\rho,2\rho\rangle}\sum_{w\in\mathfrak{S}_{n+1}}(-1)^{l(w)}\xi^{\langle 2\rho, w\bullet 0\rangle}}.
  \]
  It is easy to check that $w\bullet(-2\rho,2\rho) = w\bullet(0,0)+(-2\rho,2\rho)$ and therefore
  \begin{align*}
    \langle\lambda+2\rho,(w\bullet(-2\rho,2\rho))_2\rangle &+ \langle(w\bullet(-2\rho,2\rho))_1+2\rho,\mu\rangle - \langle 2\rho,2\rho\rangle \\&= \langle \lambda+\mu+2\rho,w(\rho)\rangle + \left\langle\frac{\lambda-\mu}{2}-\rho,2\rho \right\rangle\\
    &= \left\langle w^{-1}\bullet\frac{\lambda+\mu}{2},2\rho\right\rangle + \left\langle\frac{\lambda-\mu}{2},2\rho \right\rangle.
  \end{align*}
  Then, the value of $\chi$ at $L_\xi(\lambda,\mu)$ is given by
  \begin{align*}
    \chi(L_\xi(\lambda,\mu)) &= \xi^{\left\langle\frac{\lambda-\mu}{2},2\rho \right\rangle} \frac{\sum_{w\in\mathfrak{S}_{n+1}}(-1)^{l(w)}\xi^{\left\langle w\bullet\frac{\lambda+\mu}{2},2\rho\right\rangle}}{\sum_{w\in\mathfrak{S}_{n+1}}(-1)^{l(w)}\xi^{\langle 2\rho, w\bullet 0\rangle}}\\
    &=\dim^{-}\left(L_\xi\left(\frac{\lambda-\mu}{2},-\frac{\lambda-\mu}{2}\right)\right)\dim^{\pm}\left(L_\xi\left(\frac{\lambda+\mu}{2},\frac{\lambda+\mu}{2}\right)\right)\\
    &=\dim^{-}(L_\xi(\lambda,\mu)),
  \end{align*}
  as stated.  
\end{proof}

Finally, we renormalize $S$ with a square root of
\[
  \frac{\dim(\mathbb{Z}(\mathcal{T}_\xi)\rtimes\mathcal{S})}{n+1}\dim^{+}(\overline{\mathbf{1}}) = d^{2n}(-1)^{\lvert\Phi^+\rvert+n} \left(\sum_{w\in\mathfrak{S}_{n+1}}(-1)^{l(w)}\xi^{\langle 2\rho,w\bullet 0\rangle}\right)^{-2},
\]
which is 
\[
  d^n i^{\lvert\Phi^+\rvert+n} \left(\sum_{w\in\mathfrak{S}_{n+1}}(-1)^{l(w)}\xi^{\langle 2\rho,w\bullet 0\rangle}\right)^{-1}
\]
up to a sign, where $i$ is a square root of $-1$ in $\mathbb{C}$.

  We choose for each orbit of simple objects of $\mathbb{Z}(\mathcal{T}_\xi)\rtimes\mathcal{S}$ a representative of the orbit under tensorisation by the symmetric center, such that $\mathbf{1}$ is the representative of the orbit of $\mathbf{1}$ and $I\otimes L_{\xi}(-2\rho,2\rho)$ is the representative of the orbit of $L_{\xi}(-2\rho,2\rho)$. This choice will be explained in Section \ref{sec:comparison}.

\begin{theoreme}\label{thm:modular_quantum_A}
  If $n$ is even, the category $\mathbb{Z}(\mathcal{T}_\xi)\rtimes\mathcal{S}$ gives rise to a non-degenerate braided pivotal fusion category.
  If $n$ is odd, the category $\mathbb{Z}(\mathcal{T}_\xi)\rtimes\mathcal{S}$ gives rise to a non-degenerate braided pivotal superfusion category.
  In both cases the renormalized $S$-matrix is given by
  \[
    \tilde{S}_{(\lambda,\mu),(\lambda',\mu')} = i^{-n-\lvert\Phi^+\rvert}\frac{\sum_{w\in W}(-1)^{l(w)}
    \xi^{\langle 2\rho+\lambda,(w\bullet(\lambda',\mu'))_2\rangle + \langle\mu,(w\bullet(\lambda',\mu'))_1+2\rho\rangle}}
  {d^n}
  \]
and the twist by
  \[
    \theta_{\lambda,\mu}=\xi^{\langle\lambda+2\rho,\mu\rangle}.
  \]

\end{theoreme}


\section{Malle's $\mathbb{Z}$-modular datum}
\label{sec:modular_datum}
We refer to \cite{unip_malle} and \cite{cuntz} for most of the material of this section. Let $d$ and $n$ be positive integers. Let $\zeta$ be a primitive $d$-th of unity.

\subsection{Set up}
\label{sec:set_up}

 Let $Y=\{1,2,\ldots,nd+1\}$ of cardinal $nd+1$ and $\pi\colon Y \rightarrow \mathbb{N}$ be the map defined by
\[
  \pi(k) = \begin{cases} 
  		n & \text{if } 1\leq k \leq n+1 \\
        \left\lfloor \frac{k-n-2}{d-1} \right\rfloor & \text{if } n+2 \leq k \leq nd+1
        \end{cases}.                
\]
Let $\Psi(Y,\pi)$ be the set of maps $f \colon Y \rightarrow \{0,1,\ldots,d-1\}$ such that $f$ is strictly increasing on each $\pi^{-1}(i)$, $0\leq i \leq n$. Since for $0\leq i \leq n-1$ the set $\pi^{-1}(i)$ is of cardinal $d-1$, there exists a unique element $k_i(f)\in\{0,\ldots,d-1\}$ such that $\{0,\ldots,d-1\} = f(\pi^{-1}(i))\cup\{k_i(f)\}$. Note that $f$ is then determined by the values of $f(1)<\cdots<f(n+1)$ and of $k_0(f),\ldots,k_{n-1}(f)$. For $f\in\Psi(Y,\pi)$, we set
\[
  \varepsilon(f) = (-1)^{\left\lvert\left\{(y,y')\in Y\times Y\ \middle\vert\ y < y' \text{ and } f(y)<f(y')\right\}\right\rvert}.
\]

Let $V$ be a $\mathbb{C}$-vector space of dimension $d$ with basis $(v_i)_{0\leq i \leq d-1}$. We denote by $\mathcal{S}$ the square matrix $(\zeta^{ij})_{0\leq i,j \leq d-1}$ and we will view it as an endomorphism of $V$. We set ${\tau(d) = (-1)^{d(d-1)/2}\det(\mathcal{S}) = \prod_{0\leq i<j\leq d-1}(\zeta^i-\zeta^j)}$. We consider the automorphism of the vector space $\left(\bigwedge^{n+1}V\right)\otimes \left(\bigwedge^{d-1}V\right)^{\otimes n}$ given by $\left(\bigwedge^{n+1}\mathcal{S}\right)\otimes \left(\bigwedge^{d-1}\mathcal{S}\right)^{\otimes n}$. This space has a basis given by 
\[
  \mathbf{v}_f = (v_{f(1)}\wedge\cdots\wedge v_{f(n+1)})\otimes (v_{f(n+2)}\wedge\cdots\wedge v_{f(n+d)})\otimes \cdots \otimes (v_{f(n+2+(n-1)(d-1))}\wedge\cdots\wedge v_{f(nd+1)})
\]
for $f\in \Psi(Y,\pi)$. Denote by $\mathbf{S}$ the matrix of $\left(\bigwedge^{n+1}\mathcal{S}\right)\otimes \left(\bigwedge^{d-1}\mathcal{S}\right)^{\otimes n}$ in the basis $(\mathbf{v}_f)_{f\in\Psi(Y,\pi)}$
\[
   \left(\left(\textstyle{\bigwedge}^{n+1}\mathcal{S}\right)\otimes \left(\textstyle{\bigwedge}^{d-1}\mathcal{S}\right)^{\otimes n}\right)(\mathbf{v}_f) = \sum_{g\in\Psi(Y,\pi)}\mathbf{S}_{g,f}\mathbf{v}_g.
\]

The following lemma follows immediately from \cite[Lemma 6.2]{asymptotic_cell}.

\begin{lemma}
Let $f,g\in\Psi(Y,\pi)$. Then we have
\[
  \mathbf{S}_{f,g} = (-1)^{\sum_{i=0}^{n-1}(k_i(f)+k_i(g))}\frac{(-1)^{nd(d-1)/2}\tau(d)^n}{d^n}\prod_{i=0}^{n-1}\zeta^{-k_i(f)k_i(g)}\sum_{\sigma\in\mathfrak{S}_{n+1}}
(-1)^{l(\sigma)}\prod_{i=1}^{n+1}\zeta^{f(i)g(\sigma(i))}.
\]
\end{lemma}

\subsection{Malle's $\mathbb{Z}$-modular datum}
\label{sec:malle_z}

Following \cite{unip_malle}, we consider the family $\mathcal{F}$ of unipotent characters of $G\left(d,1,\frac{n(n+1)}{2}\right)$ parametrized by reduced $d$-symbols with values in the multiset
\[
  \{0^{d-1},1^{d-1},\ldots,(n-1)^{d-1},n^{n+1}\}.
\]
These $d$-symbols are in bijection with the set
\[
  \Psi^{\#}(Y,\pi) = \left\{f\in \Psi(Y,\pi)\ \middle\vert\ \sum_{y\in Y} f(y) \equiv n\binom{d}{2} \mod d\right\}.
\]
Indeed to $f\in\Psi^{\#}(Y,\pi)$ we associate the reduced $d$-symbol $S=(S_0,\ldots,S_{d-1})$ with entries $i\in S_j$ for all $j\in f(\pi^{-1}(i))$. We define for $f\in\Psi^{\#}(Y,\pi)$
\[
  \Fr(f) = \zeta_{*}^{nd(1-d^2)}\prod_{y\in Y}\zeta_{*}^{-6(f(y)^2+df(y))},
\]
where $\zeta_{*}$ is a primitive $12d$-th root of unity such that $\zeta_{*}^{12} = \zeta$. Denote by $\mathbb{T}$ the diagonal matrix with entries $(\Fr(f))_{f\in\Psi^{\#}(Y,\pi)}$.

Following \cite[Section 5]{cuntz}, we denote by $\mathbb{S} = (\mathbb{S}_{f,g})_{f,g\in\Psi^{\#}(Y,\pi)}$ the square matrix defined by
\[
  \mathbb{S}_{f,g} = \frac{(-1)^{n(d-1)}}{\tau(d)^n}\overline{\mathbf{S}_{f,g}}.
\]

Define $f_{\mathrm{sp}}\in\Psi^{\#}(Y,\pi)$ by $f_{\mathrm{sp}}(i) = i-1$ for $1\leq i \leq n+1$ and $k_j(f_{\mathrm{sp}}) = j+1$ for $0\leq j \leq n-1$. Note that $\mathbf{S}_{f_{\mathrm{sp}},g} \neq 0$ for all $g \in \Psi(Y,\pi)$ as it is, up to a non-zero constant, the value of a Vandermonde determinant. The following is due to Malle \cite[4.15]{unip_malle}.

\begin{proposition}
We have
\begin{enumerate}
  \item $\mathbb{S}^4 = (\mathbb{ST})^3 = [\mathbb{S}^2,\mathbb{T}] = 1$.
  \item $\mathbb{S}$ is symmetric and unitary.
  \item For all $f,g,h \in \Psi^{\#}(Y,\pi)$
  \[
    N_{f,g}^h = \sum_{k\in\Psi^{\#}(Y,\pi)}\frac{\mathbb{S}_{f,k}\mathbb{S}_{g,k}\overline{\mathbb{S}_{h,k}}}{\mathbb{S}_{f_{\mathrm{sp}},k}} \in \mathbb{Z}.
  \]
\end{enumerate}
\end{proposition}

The special symbol $f_{\mathrm{sp}}$ also parametrizes a complex representation of $G\left(d,1,\frac{n(n+1)}{2}\right)$ and the complex conjugate of this representation is parametrized by the cospecial symbol, see \cite[2.A, 2.D]{unip_malle} for more details. Explicitly, the cospecial symbol $f_{\mathrm{cosp}}$ is given by
\[
  f_{\mathrm{cosp}}(i) =
  \begin{cases}
    0 & \text{if}\ i=1\\
    d-n+i-2 & \text{otherwise}
  \end{cases}
  \quad \text{and} \quad
  k_i(f_{\mathrm{cosp}}) = d-i-1.
\]

Given a matrix $\mathbb{S}$ as above, we can associate a fusion algebra $A$: it is a free $\mathbb{Z}$-algebra with a base $(b_f)_{f\in\Psi^{\#}(Y,\pi)}$ such that the multiplication on this base is given by
\[
  b_f\cdot b_g = \sum_{k\in\Psi^{\#}(Y,\pi)}N_{f,g}^kb_k.
\]
This multiplication is associative, and Cuntz made a remarkable conjecture about the algebra $A$:

\begin{conjecture}[{\cite[Vermutung 5.1.6]{cuntz_phd}}]
  \label{conj:cuntz}
  Let $A^{\mathrm{abs}}$ be a free $\mathbb{Z}$-module equipped with a base $(b_f^{\mathrm{abs}})_{f\in\Psi^{\#}(Y,\pi)}$. We define a multiplication on the base by
\[
  b^{\mathrm{abs}}_f\cdot b^{\mathrm{abs}}_g = \sum_{k\in\Psi^{\#}(Y,\pi)}\left\lvert N_{f,g}^k\right\rvert b^{\mathrm{abs}}_k.
\]

Then this multiplication is associative.
\end{conjecture}

\begin{remark}
  \label{rk:change_signs}
  If there exists $(d_f)_{f\in \Psi^{\#}(Y,\pi)}$ with $d_f\in\{\pm 1\}$ such that the structure constants of the ring $A$ relative to the base $(d_fb_f)_{f\in\Psi^{\#}(Y,\pi)}$ are positive, then this conjecture is clearly true: one have $\lvert N_{f,g}^k\rvert = d_fd_gd_k N_{f,g}^k$ and the associativity of $A^{\mathrm{abs}}$ follows from the associativity of $A$.

  If $d$ is even and $n$ is odd, it seems impossible to find these $(d_f)_{f\in\Psi^{\#}(Y,\pi)}$, see \cite[Section 4]{super_application} when $n=1$.
\end{remark}

We will compare this modular datum to the one constructed in Theorem \ref{thm:modular_quantum_A}, which will give a proof of the Conjecture \ref{conj:cuntz}. For this, we rewrite the expressions for $\mathbb{S}$ and $\mathbb{T}$. First, note that $f \in \Psi(Y,\pi)$ is in $\Psi^{\#}(Y,\pi)$ if and only if
\[
  \sum_{i=1}^{n+1} f(i) \equiv \sum_{i=0}^{n-1} k_i(f) \mod d.
\]
Therefore, we get rid of $f(n+1)$ in the expressions of $\mathbb{S}$ and $\mathbb{T}$.

\begin{proposition}
  For $f,g \in \Psi^{\#}(Y,\pi)$ we have
  \[
     \Fr(f) = \zeta^{\sum_{i=1}^n(k_{i-1}(f) - f(i))\left(\sum_{j=1}^{i}f(j) - \sum_{j=1}^{i-1}k_{j-1}(f)\right)}
  \]
  and
  \begin{multline*}
     \mathbb{S}_{f,g} = \frac{1}{d^n}\varepsilon(f)\varepsilon(g)(-1)^{\sum_{i=1}^{n}(k_{i-1}(f)+k_{i-1}(g))}\\ 
     \sum_{\sigma\in\mathfrak{S}_{n+1}}(-1)^{l(\sigma)}\zeta^{\sum_{i=1}^{n}\left[(k_{i-1}(f)-f(i))\left(\sum_{j=1}^{i}(\sigma\cdot g)(j) - \sum_{j=1}^{i-1}k_{j-1}(\sigma\cdot g)\right) 
     + (k_{i-1}(\sigma\cdot g)-(\sigma \cdot g)(i))\left(\sum_{j=1}^{i}f(j) - \sum_{j=1}^{i-1}k_{j-1}(f)\right)\right]}
  \end{multline*}
\end{proposition}

\begin{proof}
 Let us begin by the value of $\Fr(f)$. By definition, $\Fr(f) = \zeta_{*}^{\alpha}$ where
 \[
   \alpha = nd(1-d^2)-6\sum_{y\in Y}(f(y)^2+df(y)),
 \]
 which we consider as an element of $\mathbb{Z}/12d\mathbb{Z}$. First, as  $\{0,\ldots,d-1\} = f(\pi^{-1}(i))\cup\{k_i(f)\}$ for $0 \leq i \leq n-1$, we have
 \[
  \sum_{y\in Y}(f(y)^2+df(y)) = \sum_{i=1}^{n+1}(f(i)^2+df(i)) + n \sum_{i=0}^{d-1}(i^2+di) - \sum_{i=1}^{n}(k_{i-1}(f)^2+dk_{i-1}(f)).
 \]
 But $6\sum_{i=0}^{d-1}(i^2+di) = (d-1)d(2d-1) + 3d^2(d-1) = d(d-1)(5d-1)$ and therefore
 \[
   d(1-d^2)-6\sum_{i=0}^{d-1}(i^2+di) = 6d^2(1-d) \equiv 0 \mod 12d,
 \]
 so that
 \[
   \alpha = -6\left(\sum_{i=1}^{n+1}(f(i)^2+df(i)) - \sum_{i=1}^{n}(k_{i-1}(f)^2+dk_{i-1}(f))\right).
 \]
 Fix $\eta\in\mathbb{Z}$ such that $f(n+1) = \sum_{i=1}^n(k_{i-1}(f)-f(i))+\eta d$, so that we have
 \begin{align*}
 f(n+1)^2+df(n+1) &= \left(\sum_{i=1}^n(k_{i-1}(f)-f(i))\right)^2 + d\sum_{i=1}^{n}(k_{i-1}(f)-f(i)) \\
 & \qquad\qquad\qquad\qquad+ 2d\eta\sum_{i=1}^{n}(k_{i-1}(f)-f(i)) + d^2\underbrace{\eta(1+\eta)}_{\equiv 0 \mod 2} \\
  & \equiv \left(\sum_{i=1}^n(k_{i-1}(f)-f(i))\right)^2 + d\sum_{i=1}^{n}(k_{i-1}(f)-f(i)) \mod 2d.
 \end{align*}
 Finally
 \begin{align*}
   \alpha &= -6\left(\left(\sum_{i=1}^n(k_{i-1}(f)-f(i))\right)^2 + \sum_{i=1}^n(k_{i-1}(f)^2-f(i)^2)\right)\\
          &= 12\left(\sum_{i=1}^n f(i)(k_{i-1}(f)-f(i)) - 2\sum_{1\leq j < i \leq n}(k_{i-1}(f) - f(i))(k_{j-1}(f) - f(j))\right)\\
          &= 12\sum_{i=1}^n(k_{i-1}(f) - f(i))\left(\sum_{j=1}^{i}f(j) - \sum_{j=1}^{i-1}k_{j-1}(f)\right)
 \end{align*}
 which gives the expected formula for $\Fr(f)$.
 
 We now turn to the formula for $\mathbb{S}$. Since $\overline{\tau(d)} = (-1)^{(d-1)(d-2)/2}\tau(d)$ we have
 \[
   \mathbb{S}_{f,g} = \frac{(-1)^{\sum_{i=0}^{n-1}(k_i(f)+k_i(g))}}{d^n}\prod_{i=0}^{n-1}\zeta^{k_i(f)k_i(\sigma\cdot g)}\sum_{\sigma\in\mathfrak{S}_{n+1}}
\varepsilon(\sigma)\prod_{i=1}^{n+1}\zeta^{-f(i)(\sigma \cdot g)(i)},
 \]
 where $\varepsilon(\sigma)$ is the parity of the permutation $\sigma$. Fix $\sigma\in\mathfrak{S}_{n+1}$ and let $h = \sigma\cdot g$. Then
 \begin{align*}
 	\sum_{i=1}^{n}&k_{i-1}(f)k_{i-1}(h) - \sum_{i=1}^{n+1}f(i)h(i) \\&= \sum_{i=1}^{n}(k_{i-1}(f)k_{i-1}(h) - f(i)h(i)) - \sum_{1\leq i,j \leq n}(k_{i-1}(f)-f(i))(k_{j-1}(h)-h(j))\\
 	&= \sum_{i=1}^{n}(k_{i-1}(f)-f(i))h(i) + \sum_{i=1}^{n}(k_{i-1}(h)-h(i))f(i) - \sum_{1\leq i\neq j \leq n}(k_{i-1}(f)-f(i))(k_{j-1}(h)-h(j)) \\
 	&=\sum_{i=1}^{n}\left[(k_{i-1}(f)-f(i))\left(\sum_{j=1}^{i}h(j) - \sum_{j=1}^{i-1}k_{j-1}(h)\right) 
     + (k_{i-1}(h)-h(i))\left(\sum_{j=1}^{i}f(j) - \sum_{j=1}^{i-1}k_{j-1}(f)\right)\right].
 \end{align*}
 Hence, taking the sum over $\sigma\in\mathfrak{S}_{n+1}$, we obtain the second formula.
\end{proof}

\subsection{Ennola $d$-ality}
\label{sec:ennola}

For each $d$-symbol $S$ in the family $\mathcal{F}$, Malle defined a polynomial $\gamma_S(q)$, which has similar properties to the degree of a unipotent character for a finite group of Lie type. These polynomials satisfy the Ennola property: there exists a bijection $\mathcal{E}\colon\mathcal{F}\rightarrow\mathcal{F}$ such that for every symbol $S$ in the family $\mathcal{F}$, $\gamma_S(\zeta q) = \pm \gamma_{\mathcal{E}(S)}$. The symbol $\mathcal{E}(S)$ is explicitly given in the proof of \cite[Folgerung 3.11]{unip_malle}. We describe it at the level of functions in $\Psi^{\#}(Y,\pi)$. Let $f\in\Psi^{\#}(Y,\pi)$, its Ennola dual $\mathcal{E}(f)$ is given by the unique function $g$ such that
\begin{multline*}
  k_i(g) = \left(k_i(f) + i - \frac{n(n+3)}{2}\right)^{\text{res}} \quad \text{and} \\ \left\{g(i)\ \middle\vert\ 1\leq i \leq n+1\right\} = \left\{\left(f(i) - \frac{n(n+1)}{2}\right)^{\text{res}}\ \middle\vert\  1\leq i \leq n+1\right\},
\end{multline*}
$(k)^{\text{res}}$ being the remainder in the Euclidean division of $k$ by $d$.

\subsection{Comparison with the modular datum of the category $\mathbb{Z}(\mathcal{T}_\xi)\rtimes\mathcal{S}$}
\label{sec:comparison}

We now compare the modular datum of Malle with the modular datum of Theorem \ref{thm:modular_quantum_A}, from which we use the notation. To each function $f\in\Psi^{\#}(Y,\pi) $ we associate $\mu = \sum_{i=1}^n \mu_i\alpha_i\in Q$ and $\lambda = -\mu + 2\sum_{i=1}^n \eta_i\varpi_i\in P$ with
\[
  \mu_i = \sum_{j=1}^{i-1}k_{j-1}(f) - \sum_{j=1}^i f(j) \quad \text{and} \quad \eta_i = f(i+1)-f(i)-1.
\]
As $f$ takes its values in $\{0,\ldots,d-1\}$ and is strictly increasing on $\{1,\ldots,n+1\}$ we have $\sum_{i=1}^n\eta_i \leq d-1-n$ and therefore $\lambda+\mu\in 2C$. We then obtain a map
\[
  \fonctionNom{\iota}{\Psi^{\#}(Y.\pi)}{\overline{C}}{f}{(\lambda_f,\mu_f)}.
\]
Note that the special symbol $f_0$ is sent to $(0,0)$ and the cospecial symbol is sent to $(-2\rho+(2d-(n+1))\varpi_1,2\rho-(n+1)\varpi_1)$. We define a left inverse to $\iota$ as follows. Let $(\lambda,\mu)\in\overline{C}$ and define $f_{\lambda,\mu}$ as the unique function $f\in\Psi(Y,\pi)$ such that
\[
  \{f(1),\ldots,f(n+1)\} = \left\{\left(-\langle\mu,\varpi_1\rangle+\sum_{j=1}^{i-1}\left\langle\frac{\lambda+\mu}{2}+\rho,\alpha_j\right\rangle\right)^{\text{res}}\ \middle\vert\ 1\leq i \leq n+1\right\}
\]
and
\[
  k_{i-1}(f) = \left(\sum_{j=1}^{i}\left\langle\frac{\lambda-\mu}{2}+\rho,\alpha_j\right\rangle\right)^{\text{res}}.
\]
An easy computation shows that $f_{\lambda,\mu}$ belongs to $\Psi^{\#}(Y,\pi)$ and that $f_{\iota(f)} = f$ for any $f\in\Psi^{\#}(Y,\pi)$. A straightforward computation shows the following lemma.

\begin{lemma}
  Let $f,g\in\Psi^{\#}(Y,\pi)$. With the previous notation we have
  \begin{multline*}
    \langle \lambda_f + 2\rho,\mu_g \rangle + \langle \lambda_g + 2\rho,\mu_f \rangle
    = -\sum_{i=1}^{n}(k_{i-1}(f)-f(i))\left(\sum_{j=1}^{i}g(j) - \sum_{j=1}^{i-1}k_{j-1}(g)\right) \\ 
     - \sum_{i=1}^{n}(k_{i-1}(g)-g(i))\left(\sum_{j=1}^{i}f(j) - \sum_{j=1}^{i-1}k_{j-1}(f)\right).
  \end{multline*}
\end{lemma}

Now, let $\zeta = \xi^{-2}$. It follows immediately that
\[
  \mathbb{T}_{f} = \theta_{L_{\xi}(\lambda_f,\mu_f)}
\]
and noticing that $\iota(\sigma\cdot f) = s\bullet \iota(f)$ (where we extend $\iota$ to all functions $f\colon Y\rightarrow \{0,\ldots,d-1\}$ such that $f$ is injective on each $\pi^{-1}(i)$, the action of $\mathfrak{S}_{n+1}$ being induced by the action on $\{1,\ldots,n+1\}\subset Y$) we have
\[
  \mathbb{S}_{f,g} = \frac{\sum_{w\in\mathfrak{S}_{n+1}}(-1)^{l(w)}\xi^{\langle 2\rho,w\bullet 0\rangle}}{d^n}\varepsilon(f)\varepsilon(g)(-1)^{\sum_{i=1}^{n}(k_{i-1}(f)+k_{i-1}(g))} S_{(\lambda_f,\mu_f),(\lambda_g,\mu_g)}.
\]

Let $D$ be the diagonal matrix with entries $(\varepsilon(f)(-1)^{\sum_{i=1}^nk_{i-1}(f)})_{f\in\Psi^{\#}(Y,\pi)}$. We can now state the main theorem of this paper.

\begin{theoreme}\label{thm:main_result}
  If $n$ is even (resp. odd) the braided pivotal fusion category (resp. braided pivotal superfusion category) of \ref{thm:modular_quantum_A} is a categorification of the Malle $\mathbb{Z}$-modular datum: \[
   \mathbb{S}_{f,g} = i^{-\lvert \Phi^+\rvert - n} D\tilde{S}_{\iota(f),\iota(g)}D^{-1} \quad \text{and} \quad \mathbb{T}_{f} = \theta_{L_{\xi}(\iota(f))}.
 \]
\end{theoreme}

\begin{remark}
  The image of $\iota$ contains exactly one representative of each orbit of simple objects of $\mathbb{Z}(\mathcal{T}_\xi)\rtimes\mathcal{S}$ under tensorisation by the symmetric center. As $I\otimes L_\xi(-2\rho,2\rho)$ is in the image of $\iota$, this justifies our choice of $\bar{\mathbf{1}}$ made in Section \ref{sec:categ-typeA}.
\end{remark}

\begin{corollary}
  The Conjecture \ref{conj:cuntz} is true: the multiplication on $A^{\mathrm{abs}}$ is associative.
\end{corollary}

\begin{proof}
  If $n$ is even, this follows from the Remark \ref{rk:change_signs}. If $n$ is odd, then the Malle's $\mathbb{Z}$-modular datum is given by a slightly degenerate category, or similarly by a supercategory $\mathcal{D}$ \cite[Section 3]{super_application}. We have at our disposal the super-Grothendieck group of $\mathcal{D}$, which is a $\mathbb{Z}[\varepsilon]/(\varepsilon^2-1)$-module. We then have ring isomorphisms
\[
  \sGr(\mathcal{D})/(\varepsilon+1)\simeq A\quad \text{and}\quad \sGr(\mathcal{D})/(\varepsilon-1)\simeq A^{\mathrm{abs}}.
\]
\end{proof}

\begin{remark}
  If $d$ and $n$ are odd, we can find a modular subcategory of $\mathbb{Z}(\mathcal{T}_\xi)\rtimes\mathcal{S}$ which is a categorification of the Malle's $\mathbb{Z}$-modular datum: we simply take the full subcategory of $\mathbb{Z}(\mathcal{T}_\xi)\rtimes\mathcal{S}$ with objets of degree in $Q/dQ$. Therefore, there exists $(d_f)_{f\in\Psi^{\#}(Y,\pi)}$ with $d_f=\pm 1$ such that $d_fd_gd_hN_{f,g}^h$ is positive for any $f,g,h\in\Psi^{\#}(Y,\pi)$. This was also conjectured by Cuntz.
\end{remark}

\subsection{Categorification of the Ennola property}
\label{ref:ennola_categorified}

The Ennola $d$-ality gives a bijection $\mathcal{E}$ on $\Psi^{\#}(Y,\pi)$ and satisfies $\mathcal{E}^d=\id$. Consider $(-\gamma,\gamma)$ in $\overline{C}$ given by $\gamma = \frac{(n+1)(n+2)}{2}\varpi_1-\rho$.

\begin{proposition}
  Tensoring by $L_{\xi}(-\gamma,\gamma)$ is a categorification of the Ennola $d$-ality. For any $f\in\Psi^{\#}(Y,\pi)$ we have
  \[
    L_\xi(\iota(f))\otimes L_{\xi}(-\gamma,\gamma) \simeq L_\xi(\iota(\mathcal{E}(f))),
  \]
  the isomorphism being understood in the non-degenerate (super)fusion category associated to $\mathbb{Z}(\mathcal{T}_\xi)\rtimes\mathcal{S}$.
\end{proposition}

As $\gamma$ belongs to $Q^\vee$, it is clear that $L_{\xi}(-\gamma,\gamma)^{\otimes d}\simeq \mathbf{1}$.

\begin{proof}
  We write $\gamma$ on the basis of simple roots
  \[
    \gamma = \frac{1}{2}\sum_{i=1}^n (n-i+1)(n-i+2)\alpha_i.
  \]
  Let $f\in\Psi^{\#}(Y,\pi)$ and let $(\lambda,\mu) = \iota(f)+(-\gamma,\gamma)$. Writing as usual $\mu = \sum_{i=1}^n\mu_i\alpha_i$ and $\lambda = -\mu +2\sum_{i=1}^n\eta_i\varpi_i$ we have
  \[
    \mu_i = \sum_{j=1}^{i-1}k_{j-1}(f) - \sum_{j=1}^i f(j) + \frac{(n-i+1)(n-i+2)}{2} \quad \text{and} \quad \eta_i = f(i+1)-f(i)-1.
  \]
  By definition, $f_{\lambda,\mu}$ is the unique function $g\in\Psi^{\#}(Y,\pi)$ such that
\[
  \{g(1),\ldots,g(n+1)\} = \left\{\left(-\langle\mu,\varpi_1\rangle+\sum_{j=1}^{i-1}\left\langle\frac{\lambda+\mu}{2}+\rho,\alpha_j\right\rangle\right)^{\text{res}}\ \middle\vert\ 1\leq i \leq n+1\right\}
\]
and
\[
  k_{i-1}(g) = \left(\sum_{j=1}^{i}\left\langle\frac{\lambda-\mu}{2}+\rho,\alpha_j\right\rangle\right)^{\text{res}}.
\]
  But as
  \[
    -\langle\mu,\varpi_1\rangle+\sum_{j=1}^{i-1}\left\langle\frac{\lambda+\mu}{2}+\rho,\alpha_j\right\rangle = f(i)-\frac{n(n+1)}{2}
  \]
  and
  \begin{align*}
    \sum_{j=1}^{i}\left\langle\frac{\lambda-\mu}{2}+\rho,\alpha_j\right\rangle &= f(i+1) - \frac{n(n+1)}{2} + \mu_{i+1}-\mu_i\\
    &= k_{i-1}(f) -\frac{n(n+1}{2}+\frac{(n-i)(n-i+1)}{2} - \frac{(n-i+1)(n-i+2)}{2}\\
    &= k_{i-1}(f) + i-1 - \frac{n(n+3)}{2}
  \end{align*}
  we have $f_{\lambda,\mu} = f_{\iota(\mathcal{E}(f))}$. As the fibres of the map $(\lambda,\mu)\mapsto f_{\lambda,\mu}$ are exactly the orbits under tensorisation by the symmetric center on the set of simple objects of $\mathbb{Z}(\mathcal{T}_\xi)\rtimes \mathcal{S}$, there exists $k\in\mathbb{Z}$ such that $L_\xi(\iota(f))\otimes L_{\xi}(-\gamma,\gamma) \simeq L_\xi(\iota(\mathcal{E}(f)))\otimes I^{\otimes k}$.
\end{proof}


\section{$\mathbb{Z}$-modular data associated to exceptional complex reflection groups}
\label{sec:exceptional}
The notion of ``unipotent characters'' has been defined for some exceptional complex reflection group by Broué, Malle and Michel \cite{spetsesI}, \cite{spetsesII}, which are called spetsial. The notion of families of such characters, as well as of Fourier transform and eigenvalues of the Frobenius exists, and are available in the package CHEVIE of GAP \cite{chevie}, \cite{chevie_jean}. In this section, we categorify some of these $\mathbb{Z}$-modular data using subcategories of $\mathbb{Z}(\mathcal{T}_\xi)\rtimes\mathcal{S}$, for a well suited simple complex Lie algebra $\mathfrak{g}$ and $\xi$ a well chosen root of unity. The unipotent characters of a complex reflection group $G$ are obtained in CHEVIE and displayed with the following command (we do the example of $G_4$):

\begin{center}
\begin{BVerbatim}[fontsize=\footnotesize]
gap> G:=ComplexReflectionGroup(4);;
gap> U:=UnipotentCharacters(G);;
gap> Display(U,rec(byFamily:=true));
Unipotent characters for G4
     Name |               Degree FakeDegree Eigenvalue    Label
________________________________________________________________
*phi{1,0} |                    1          1          1          
________________________________________________________________ 
*phi{2,1} |(3-ER(-3))/6qP'3P4P"6        qP4          1     1.E3       
#phi{2,3} |(3+ER(-3))/6qP"3P4P'6      q^3P4          1   1.E3^2
Z3:2      |     -ER(-3)/3qP1P2P4          0       E3^2  E3.E3^2
________________________________________________________________
*phi{3,2} |              q^2P3P6    q^2P3P6          1         
________________________________________________________________
*phi{1,4} | -ER(-3)/6q^4P"3P4P"6        q^4          1  1.-E3^2
phi{2,5}  |         1/2q^4P2^2P6      q^5P4          1   1.E3^2
G4        |        -1/2q^4P1^2P3          0         -1 -E3^2.-1
Z3:11     |   -ER(-3)/3q^4P1P2P4          0       E3^2   E3.-E3
#phi{1,8} |  ER(-3)/6q^4P'3P4P'6        q^8          1  -1.E3^2
\end{BVerbatim}
\end{center}

The unipotents characters are ordered by families. In the first column, one can read the name of the unipotent characters; In the second column, one can read the degree of the unipotent character; in the third column, one can read the fake degree, for the principal series unipotent characters; in the fourth column, one can read the eigenvalues of the Frobenius; finally, in the last column, one can read some labelling compatible with a categorification of the modular datum associated with the family.

Cyclotomic polynomials are denoted by \verb|Pk|, \verb|Pk'|, \verb|Pk"| and \verb|E3=E(3)| is a third root of unity.

One can access a family, its Fourier matrix and the eigenvalues of the Frobenius as follows (continuing the same example):

\begin{center}
\begin{BVerbatim}[fontsize=\footnotesize]
gap> f:=U.families[2];
Family("RZ/3^2",[6,5,8])
gap> f.fourierMat;
[ [ -2/3*E(3)-1/3*E(3)^2, -1/3*E(3)-2/3*E(3)^2, -1/3*E(3)+1/3*E(3)^2 ], 
  [ -1/3*E(3)-2/3*E(3)^2, -2/3*E(3)-1/3*E(3)^2, 1/3*E(3)-1/3*E(3)^2 ], 
  [ -1/3*E(3)+1/3*E(3)^2, 1/3*E(3)-1/3*E(3)^2, 1/3*E(3)-1/3*E(3)^2 ] ]
gap> f.eigenvalues
[ 1, 1, E(3)^2 ]
\end{BVerbatim}
\end{center}

\subsection{Two families attached to $G_{27}$}
\label{sec:G27}

We consider the complex reflection group denoted by $G_{27}$ in the classification of Shephard-Todd \cite{shephard_todd}. Two families of unipotent characters are of size $18$, the second one and the last but one. The $\mathbb{Z}$-modular datum they define are complex conjugate to each other, hence we will only consider the second family of unipotent characters of $G_{27}$. The $\mathbb{Z}$-modular datum is in fact the tensor product of a $\mathbb{Z}$-modular datum of size $3$ and of an $\mathbb{N}$-modular datum of size $6$.

\subsubsection{The $\mathbb{Z}$-modular datum of size $3$}
\label{sec:card3}

The Fourier matrix and the eigenvalues of the Frobenius are
\[
S=\frac{1}{3}
\begin{pmatrix}
  1-\zeta_3&1-\zeta_3^2&\zeta_3-\zeta_3^3\\
  1-\zeta_3^2&1-\zeta_3&\zeta_3^2-\zeta_3\\
  \zeta_3-\zeta_3^2&\zeta_3^2-\zeta_3&\zeta_3-\zeta_3^2
\end{pmatrix} \quad \text{and} \quad
T=
\begin{pmatrix}
  1&0&0\\
  0&1&0\\
  0&0&\zeta_3^2
\end{pmatrix}
\]
where $\zeta_3$ is a primitive third root of unity. It is easily checked that it coincides with the $\mathbb{Z}$-modular datum associated to the non-trivial family of the cyclic complex reflection group $G(3,1,1)$. By Theorem \ref{thm:main_result}, this $\mathbb{Z}$-modular datum is categorifed by the braided pivotal slightly degenerate fusion category $\mathbb{Z}(\mathcal{T}_\xi)\rtimes\mathcal{S}$ for $\mathfrak{g}=\mathfrak{sl}_2$ and $\xi$ a sixth root of unity such that $\xi^2 = \zeta_3^{-1}$.

\begin{remark}
  Up to conjugation by a diagonal matrix with coefficients in $\{\pm 1\}$, we recover the modular datum associated to the second family of unipotent characters of $G_4$, which we gave in the beginning of this Section.
\end{remark}

\subsubsection{The $\mathbb{N}$-modular datum of size $6$}
\label{sec:card6}

The Fourier matrix and the eigenvalues of the Frobenius are
\[
S=\frac{1}{10}
\left(\begin{smallmatrix}
  -\zeta_5^4+\zeta_5^3+\zeta_5^2-\zeta_5&-\zeta_5^4+\zeta_5^3+\zeta_5^2-\zeta_5&2(-\zeta_5^4+\zeta_5^3+\zeta_5^2-\zeta_5)&2(-\zeta_5^4+\zeta_5^3+\zeta_5^2-\zeta_5)&-5&-5\\
  -\zeta_5^4+\zeta_5^3+\zeta_5^2-\zeta_5&-\zeta_5^4+\zeta_5^3+\zeta_5^2-\zeta_5&2(-\zeta_5^4+\zeta_5^3+\zeta_5^2-\zeta_5)&2(-\zeta_5^4+\zeta_5^3+\zeta_5^2-\zeta_5)&5&5\\
  2(-\zeta_5^4+\zeta_5^3+\zeta_5^2-\zeta_5)&2(-\zeta_5^4+\zeta_5^3+\zeta_5^2-\zeta_5)&2(3\zeta_5^4+2\zeta_5^3+2\zeta_5^2+3\zeta_5)&2(-2\zeta_5^4-3\zeta_5^3-3\zeta_5^2-2\zeta_5)&0&0\\
  2(-\zeta_5^4+\zeta_5^3+\zeta_5^2-\zeta_5)&2(-\zeta_5^4+\zeta_5^3+\zeta_5^2-\zeta_5)&2(-2\zeta_5^4-3\zeta_5^3-3\zeta_5^2-2\zeta_5)&2(3\zeta_5^4+2\zeta_5^3+2\zeta_5^2+3\zeta_5)&0&0\\
  -5&5&0&0&5&-5\\
  -5&5&0&0&-5&5\\
\end{smallmatrix}\right)
\]
and
\[
  T=
  \begin{pmatrix}
    1&0&0&0&0&0\\
    0&1&0&0&0&0\\
    0&0&\zeta_5^3&0&0&0\\
    0&0&0&\zeta_5^2&0&0\\
    0&0&0&0&-1&0\\
    0&0&0&0&0&1
  \end{pmatrix},
\]
where $\zeta_5$ is a primitive fifth root of unity.



\boitegrise{{\bf Notation.} {\emph{In this section, $\mathfrak{g}$ is of type $B_2$ with Cartan matrix
\[
\begin{pmatrix}
  2&-2\\
  -1&2
\end{pmatrix}.
\]
  The short simple root is denoted by $\alpha_1$ and the long simple root is denoted by $\alpha_2$. The fundamental weights are $\varpi_1=\alpha_1+\frac{1}{2}\alpha_2$ and $\varpi_2=\alpha_1+\alpha_2$.}}}{0.8\textwidth}

Let $\xi$ be a primitive twentieth root of unity such that $\xi^4=\zeta_5$ and consider the full subcategory $\mathcal{C}$ of
$\mathbb{Z}(\mathcal{T}_\xi)\rtimes\mathcal{S}$ for $\mathfrak{g}$ of
type $B_2$ generated by the objects of grading $0$ and $5\varpi_1$. As
$10\varpi_1 \in 10Q^\vee\cap Q$, $\{0,5\varpi_1\}$ is a subgroup of
$P/(10Q^\vee\cap Q)$. Hence $\mathcal{C}$ is stable by tensor
product. The fundamental chamber $C$ is given by
\[
   C = \left\{\lambda_1\varpi_1 + \lambda_2\varpi_2 \in P^+\ \middle\vert\ \lambda_1+\lambda_2 < 3\right\} = \{0,\varpi_1,\varpi_2,\varpi_1+\varpi_2,2\varpi_1,2\varpi_2\}.
\]

Therefore, there are $6$ simple objects in $\mathcal{C}$ which are
labelled by:
  \begin{itemize}
  \item in degree $0$: $(0,0),(\varpi_2,\varpi_2),(2\varpi_1,2\varpi_1)$ and $(2\varpi_2,2\varpi_2)$,
  \item in degree $5\varpi_1$: $(6\varpi_1,-4\varpi_1)$ and $(6\varpi_1+\varpi_2,-4\varpi_1+\varpi_2)$.
  \end{itemize}

We choose the following set of representatives of the fusion category $\mathcal{C}$:
\[
  \{L_\xi(0,0),L_\xi(2\varpi_2,2\varpi_2),L_\xi(2\varpi_1,2\varpi_1),L_\xi(\varpi_2,\varpi_2),L_\xi(6\varpi_1+\varpi_2,-4\varpi_1+\varpi_2),L_\xi(6\varpi_1,-4\varpi_2)\}
\]
With this order, the $S$-matrix and the $T$-matrix of the twist of $\mathcal{C}$ are
\[
S_{\mathcal{C}}=\left(\begin{smallmatrix}
  1&1&2&2&\zeta_5^4-\zeta_5^3-\zeta_5^2+\zeta_5&\zeta_5^4-\zeta_5^3-\zeta_5^2+\zeta_5\\
  1&1&2&2&-\zeta_5^4+\zeta_5^3+\zeta_5^2-\zeta_5&-\zeta_5^4+\zeta_5^3+\zeta_5^2-\zeta_5\\
  2&2&2\zeta_5^4+2\zeta_5&2\zeta_5^3+2\zeta_5^2&0&0\\
  2&2&2\zeta_5^3+2\zeta_5^2&2\zeta_5^4+2\zeta_5&0&0\\
  \zeta_5^4-\zeta_5^3-\zeta_5^2+\zeta_5&-\zeta_5^4+\zeta_5^3+\zeta_5^2-\zeta_5&0&0&-\zeta_5^4+\zeta_5^3+\zeta_5^2-\zeta_5&\zeta_5^4-\zeta_5^3-\zeta_5^2+\zeta_5\\
  \zeta_5^4-\zeta_5^3-\zeta_5^2+\zeta_5&-\zeta_5^4+\zeta_5^3+\zeta_5^2-\zeta_5&0&0&\zeta_5^4-\zeta_5^3-\zeta_5^2+\zeta_5&-\zeta_5^4+\zeta_5^3+\zeta_5^2-\zeta_5
\end{smallmatrix}\right),
\]
and
\[
  T_{\mathcal{C}}=
  \begin{pmatrix}
    1&0&0&0&0&0\\
    0&1&0&0&0&0\\
    0&0&\zeta_5^3&0&0&0\\
    0&0&0&\zeta_5^2&0&0\\
    0&0&0&0&-1&0\\
    0&0&0&0&0&1
  \end{pmatrix}
\]

Renormalizing $S_{\mathcal{C}}$ with the square root of $20$, which is the dimension of $\mathcal{C}$, gives us the negative of the matrix $S$.

\begin{theoreme}
  The $S$-matrix and the $T$-matrix of the modular category $\mathcal{C}$ satisfy
\[
  \frac{S_{\mathcal{C}}}{\sqrt{20}} = -S \quad \text \quad T_{\mathcal{C}} = T.
\]
\end{theoreme}

\begin{remark}
  The category $\mathcal{C}$ is indeed ribbon since every object is
  self dual.
\end{remark}


\subsection{Two families attached to $G_{24}$}
\label{sec:G24}

We consider the complex reflection group denoted by $G_{24}$ in the classification of Shephard-Todd \cite{shephard_todd}. Two families of unipotent characters are of size $7$, the second one and the last but one. The $\mathbb{Z}$-modular datum they define are complex conjugate to each other, hence we will only consider the last but one family of unipotent characters of $G_{24}$.

The Fourier matrix and the eigenvalues of the Frobenius are
\[
  S=\frac{1}{14}\left(
    \begin{smallmatrix}
      -\sqrt{-7}&\sqrt{-7}&7&7&-2\sqrt{-7}&-2\sqrt{-7}&-2\sqrt{-7}\\
      \sqrt{-7}&-\sqrt{-7}&7&7&2\sqrt{-7}&2\sqrt {-7}&2\sqrt{-7}\\
      7&7&7&-7&0&0&0\\
      7&7&-7&7&0&0&0\\
      -2\sqrt{-7}&2\sqrt{-7}&0&0&2\zeta_7^6+4\zeta_7^4-4\zeta_7^3-2\zeta_7&-4\zeta_7^6+2\zeta_7^5-2\zeta_7^2+4\zeta_7&-4\zeta_7^5-2\zeta_7^4+2\zeta_7^3+4\zeta_7^2\\
      -2\sqrt{-7}&2\sqrt{-7}&0&0&-4\zeta_7^6+2\zeta_7^5-2\zeta_7^2+4\zeta_7&-4\zeta_7^5-2\zeta_7^4+2\zeta_7^3+4\zeta_7^2&2\zeta_7^6+4\zeta_7^4-4\zeta_7^3-2\zeta_7\\
      -2\sqrt{-7}&2\sqrt{-7}&0&0&-4\zeta_7^5-2\zeta_7^4+2\zeta_7^3+4\zeta_7^2&2\zeta_7^6+4\zeta_7^4-4\zeta_7^3-2\zeta_7&-4\zeta_7^6+2\zeta_7^5-2\zeta_7^2+4\zeta_7
    \end{smallmatrix}\right)
\]
and
\[
 T=
 \begin{pmatrix}
   1&0&0&0&0&0&0\\
   0&1&0&0&0&0&0\\
   0&0&1&0&0&0&0\\
   0&0&0&-1&0&0&0\\
   0&0&0&0&\zeta_7^3&0&0\\
   0&0&0&0&0&\zeta_7^5&0\\
   0&0&0&0&0&0&\zeta_7^6
 \end{pmatrix},
\]
where $\zeta_7$ is a primitive seventh root of unity and $\sqrt{-7}=\zeta_7+\zeta_7^2-\zeta_7^3+\zeta_7^4-\zeta_7^5-\zeta_7^6$.

In \cite[Section 8.2.1 Beispiel 15]{cuntz_phd}, Cuntz showed that the fusion algebra associated to this modular datum is related to the Verlinde algebra of type $B_3$ at a twenty-eighth root of unity.

Let $\xi$ be a primitive twenty-eighth root of unity such that $\xi^4=\zeta_7$ and consider the category $\mathbb{Z}(\mathcal{T}_\xi)\rtimes\mathcal{S}$ for $\mathfrak{g}$ of
type $B_3$.

\boitegrise{{\bf Notations.} {\emph{In this section, $\mathfrak{g}$ is of type $B_3$ with Cartan matrix
\[
\begin{pmatrix}
  2&-2&0\\
  -1&2&-1\\
  0&-1&2
\end{pmatrix}.
\]
  The short simple root is denoted by $\alpha_1$ and the two long simple roots are denoted by $\alpha_2$ and $\alpha_3$. The fundamental weights are $\varpi_1=\frac{3}{2}\alpha_1+\alpha_2+\frac{1}{2}\alpha_3$, $\varpi_2=2\alpha_1+2\alpha_2+\alpha_3$ and $\varpi_3=\alpha_1+\alpha_2+\alpha_3$.
}}}{0.8\textwidth}

We have $\varpi_1=\frac{3}{2}\alpha_1^\vee+2\alpha_2^\vee+\alpha_3^\vee$ and
therefore the subgroup of $P/(14Q^\vee\cap Q)$ generated by $7\varpi_1$
is $\{0,7\varpi_1,14\varpi_1,21\varpi_1\}$. Let $\mathcal{C}$ be the full subcategory of $\mathbb{Z}(\mathcal{T}_\xi)\rtimes\mathcal{S}$ generated by the objects of grading in $\{0,7\varpi_1,14\varpi_1,21\varpi_1\}$. This category is then stable by tensor product. The fundamental chamber $C$ is
\[
   C = \left\{\lambda_1\varpi_1 + \lambda_2\varpi_2 +\lambda_3\varpi_3\in P^+\ \middle\vert\ \lambda_1+2\lambda_2 +\lambda_3< 3\right\} = \{0,\varpi_1,\varpi_2,\varpi_3,2\varpi_1,2\varpi_3,\varpi_1+\varpi_3\}.
\]

Therefore, there are $14$ simple objects in $\mathcal{C}$ which are
labelled by:
  \begin{itemize}
  \item in degree $0$: $(0,0),(\varpi_2,\varpi_2),(\varpi_3,\varpi_3),(2\varpi_1,2\varpi_1)$ and $(2\varpi_3,2\varpi_3)$,
  \item in degree $7\varpi_1$: $(8\varpi_1,-6\varpi_1)$ and $(8\varpi_1+\varpi_3,-6\varpi_1+\varpi_3)$,
  \item in degree $14\varpi_1$: $(14\varpi_1,-14\varpi_1),(14\varpi_1+\varpi_2,-14\varpi_1+\varpi_2),(14\varpi_1+\varpi_3,-14\varpi_1+\varpi_3),\\(16\varpi_1,-12\varpi_1)$ and $(14\varpi_1+2\varpi_3,-14\varpi_1+2\varpi_3)$
  \item in degree $21\varpi_1$: $(22\varpi_1,-20\varpi_1)$ and $(22\varpi_1+\varpi_3,-20\varpi_1+\varpi_1)$.
  \end{itemize}

  \begin{proposition}
    The category $\mathcal{C}$ is slightly degenerate. The non-unit
    objecct of its symmetric center is
    $L_{\xi}(14\varpi_1+2\varpi_3,-14\varpi_1+2\varpi_3)$ which is of
    dimension $-1$ and of twist $1$.
  \end{proposition}

  We choose the following set of representatives of the superfusion category $\widehat{\mathcal{C}}$:
\begin{multline*}
  \{L_\xi(0,0),L_\xi(14\varpi_1,-14\varpi_1),L_\xi(22\varpi_1,-20\varpi_1),L_\xi(22\varpi_1+\varpi_1,-20\varpi_1+\varpi_3),\\L_\xi(\varpi_3,\varpi_3),L_\xi(\varpi_2,\varpi_2),L_\xi(2\varpi_1,2\varpi_1)\}.
\end{multline*}
With this order, the $S$-matrix and the $T$-matrix of the twist of $\widehat{\mathcal{C}}$ are
\[
  S_{\widehat{\mathcal{C}}}=
  \begin{pmatrix}
    1&-1&\sqrt{-7}&\sqrt{-7}&2&2&2\\
    -1&1&\sqrt{-7}&\sqrt{-7}&-2&-2&-2\\
    \sqrt{-7}&\sqrt{-7}&\sqrt{-7}&-\sqrt{-7}&0&0&0\\
    \sqrt{-7}&\sqrt{-7}&-\sqrt{-7}&\sqrt{-7}&0&0&0\\
    2&-2&0&0&2\zeta_7^6+2\zeta_7&2\zeta_7^5+2\zeta_7^2&2\zeta_7^4+2\zeta_7^3\\
    2&-2&0&0&2\zeta_7^5+2\zeta_7^2&2\zeta_7^4+2\zeta_7^3&2\zeta_7^6+2\zeta_7\\
    2&-2&0&0&2\zeta_7^4+2\zeta_7^3&2\zeta_7^6+2\zeta_7&2\zeta_7^5+2\zeta_7^2
  \end{pmatrix}
\]
and
\[
   T_{\widehat{\mathcal{C}}}=
 \begin{pmatrix}
   1&0&0&0&0&0&0\\
   0&1&0&0&0&0&0\\
   0&0&1&0&0&0&0\\
   0&0&0&-1&0&0&0\\
   0&0&0&0&\zeta_7^3&0&0\\
   0&0&0&0&0&\zeta_7^5&0\\
   0&0&0&0&0&0&\zeta_7^6
 \end{pmatrix}
\]

This category is not spherical and the object $\bar{\mathbf{1}}$ is $L_\xi(14\varpi_1,-14\varpi_1)$, which is of dimension $-1$. Therefore following \cite[Theorem 2.7]{super_application}, we renormalise $S_{\widehat{\mathcal{C}}}$ with the square root of $-28$, which is $\sdim(\mathcal{C})\dim^{+}(\bar{\mathbf{1}})$.

\begin{theoreme}
  The $S$-matrix and the $T$-matrix of the superfusion category $\widehat{\mathcal{C}}$ satisfy
\[
  \frac{S_{\widehat{\mathcal{C}}}}{i\sqrt{28}} = S \quad \text{and} \quad T_{\widehat{\mathcal{C}}} = T,
\]
where $i=\xi^7$.
\end{theoreme}



\bibliographystyle{smfalpha}
\bibliography{biblio}

\end{document}